\documentclass[reqno,11pt]{article}
\usepackage{mathrsfs}
\usepackage{amssymb}
\usepackage{amsthm}
\usepackage{amsmath}
\usepackage{amsfonts}
\usepackage{mathtools}
\usepackage{enumerate}
\usepackage{bm}

\usepackage{graphicx}

\usepackage{bbm}
\usepackage{mathrsfs}
\usepackage{amssymb}
\usepackage[thicklines]{cancel}

\UseRawInputEncoding

\usepackage[usenames,dvipsnames]{xcolor}
\usepackage{soul}

\usepackage{txfonts}

\usepackage{psfrag}
\usepackage[numbers,sort&compress]{natbib}
\usepackage{setspace,color,wrapfig}
\usepackage[colorlinks,linkcolor=blue,citecolor=cyan]{hyperref}

\numberwithin{equation}{section}
\newtheorem{theorem}{Theorem}[section]

\newtheorem{corollary}[theorem]{Corollary}
\newtheorem{proposition}[theorem]{Proposition}

\theoremstyle{definition}

\newtheorem{remark}[theorem]{Remark}
\theoremstyle{remark}

\begin{document}
\title{ Subsonic time-periodic solutions to radially symmetric spiral  flows   in an annulus}
\author{Huimin Yu\thanks{School of Mathematics, Shandong Normal University, Jinan, Shandong Province, China, 250014. Email:hmyu@sdnu.edu.cn}\and
Zihao Zhang\thanks{School of Mathematics, Jilin University, Changchun, Jilin Province, China, 130012. Email: zhangzihao@jlu.edu.cn}}
\date{}
\maketitle
\def\be{\begin{eqnarray}}
\def\ee{\end{eqnarray}}
\def\no{\nonumber}
\newcommand{\de}{{\mathrm{d}}}
\def\div{{\rm div\,}}
\def\curl{{\rm curl\,}}
\def\om{\omega}
\def\th{\theta}
\def\la{\lambda}
\newcommand{\ro}{{\rm rot}}
\newcommand{\sr}{{\rm supp}}
\newcommand{\sa}{{\rm sup}}
\newcommand{\va}{{\varphi}}
\newcommand{\md}{\mathcal{D}}
\newcommand{\me}{\mathcal{M}}
\newcommand{\ml}{\mathcal{V}}
\newcommand{\mg}{\mathcal{G}}
\newcommand{\mh}{\mathcal{H}}
\newcommand{\mf}{\mathcal{F}}
\newcommand{\ms}{\mathcal{S}}
\newcommand{\mt}{\mathcal{T}}
\newcommand{\mn}{\mathcal{N}}
\newcommand{\mb}{\mathcal{P}}
\newcommand{\mm}{\mathcal{B}}
\newcommand{\mj}{\mathcal{J}}
\newcommand{\mk}{\mathcal{K}}
\newcommand{\my}{\mathcal{U}}
\newcommand{\mw}{\mathcal{W}}
\newcommand{\mq}{\mathcal{Q}}
\newcommand{\ma}{\mathcal{A}}
\newcommand{\mc}{\mathcal{C}}
\newcommand{\mi}{\mathcal{I}}
\newcommand{\n}{\nabla}
\newcommand{\e}{\tilde}
\newcommand{\m}{\Omega}
\newcommand{\h}{\hat}
\newcommand{\x}{\bar}
%\begin{abstract}
 \newcommand{\q}{{\rm R}}
%\end{abstract}
%fariyaminarsh663@gmail.com f4wg84wnf73b
\newcommand{\p}{{\partial}}
\newcommand{\z}{{\varepsilon}}
\renewcommand\figurename{\scriptsize Fig}
\pagestyle{myheadings} \markboth{\hfill Existence and stability of subsonic time-periodic solution  \hfill}{\hfill Existence and stability of subsonic  time-periodic solution\hfill}\maketitle
\begin{abstract}
         We investigate subsonic  time-periodic solutions to the radially symmetric  non-isentropic Euler system  with nonzero angular velocity in an annulus. Under the assumption that the boundary conditions are dissipative on the inner circle and non-dissipative on the outer circle,  we   establish the existence and stability of time-periodic solutions which are close to steady radially  symmetric subsonic spiral   flows.
  It is worth noting that there is no  smallness  restriction on the background subsonic  solutions. The main difficulties arise from the derivative loss caused by the coupling with the radial entropy derivative and the additional zeroth-order terms induced by the nonzero angular velocity in the diagonalized system. One of the key ingredients of
the analysis is to rewrite the radial derivative of the entropy as a derivative along the first or fourth characteristic direction. Another one is to control the accumulation of perturbations along characteristic curves   by suitably restricting the width of the annulus.
\end{abstract}
\begin{center}
\begin{minipage}{5.5in}
\end{minipage}
\end{center}

\section{Introduction and  main results}\noindent
\par The motion of  inviscid, compressible flows
is governed by the following  Euler system:
\begin{equation}\label{1-1-ep}
\begin{cases}
\rho_t+\div(\rho {\bf u})=0,\\
(\rho{\bf u})_t+\div(\rho {\bf u}\otimes {\bf u}+ P I_n) =0,\\
(\rho{E})_t+\div(\rho{E}{\bf u}+P{\bf u})=0,\\
\end{cases}
\end{equation}
where ${\bf u}$ is the velocity, $\rho$ is the density, $P$ is the pressure, ${E}$ is the total
 energy,    $I_n$ is an $n\times n $ identity matrix, respectively.  For the ideal polytropic gas, the equation of state and the
energy are of the form
\begin{equation*}
P= e^S\rho^{\gamma} \quad{\rm {and}}\quad {E}=\frac1 2|{\bf u}|^2+\frac{ P}{(\gamma-1)\rho},
\end{equation*}
where $S$ is  the entropy and $ \gamma> 1 $ is the adiabatic exponent. Denote   the Mach
number $M$ by
$$M=\frac{|{\bf u}|}{c(\rho, S)}, $$
where
$ c(\rho, S)=\sqrt{\gamma e^S \rho^{\gamma-1}}$ is called the local sound speed. The flow is subsonic if $M<1$, sonic if $M=1$,
and supersonic if $M>1$.
\par The compressible Euler equations are not only fundamental mathematical models in fluid dynamics but also play a significant role in various engineering applications, including the design of projectiles, rockets, and aircraft. Over the past decades, the Euler equations have attracted considerable attention.  In particular,  the formation of finite-time singularities and the existence of global weak solutions for the isentropic Euler equations with inhomogeneous or geometric source terms have been extensively investigated, see \cite{C96,C17,C19,D89,K97,L64,L06,L13}. Furthermore, there have been many  interesting results on   smooth Euler flows through various nozzles, such as subsonic, subsonic-sonic, sonic-supersonic  and transonic flows, one may refer
to \cite{CX14,DWX14,DXY11,WX2013,WX15,WX2016,WX2019,WX2020,wang22, WYZ21,WX24, WX23} and references therein.
\par  In recent years, significant progress has also been made in the study of time-periodic solutions. For general quasilinear hyperbolic systems with time-periodic dissipative boundary conditions, the existence and stability of time-periodic classical solutions were established in \cite{F,Q20}. Concerning the compressible Euler equations,  the response process for supersonic waves was investigated in \cite{YH19,YZS23}, where all characteristics propagate forward in both space and time. After a sufficiently large start-up time
$T^\ast >0$, the inflow boundary condition completely governs the entire region, thereby ensuring the existence of time-periodic supersonic solutions. The corresponding non-isentropic analysis can be found in \cite{MSY23}.    However, the situation becomes  more complicated for subsonic flows, since some characteristics propagate forward while others propagate backward. The interaction between different families of characteristics makes the description of the response process and the construction of periodic solutions nontrivial. By developing a new type of iterative scheme, the well-posedness of subsonic time-periodic solutions for compressible Euler equations with source terms was established in \cite{FMY24,QYZ23-1,QYZ23,ZY25}. Most recently, the existence and stability of time-periodic transonic shock solutions have been investigated in \cite{ZQY26}.

\par In this paper, we are concerned with subsonic  time-periodic solutions to the radially symmetric compressible Euler system  with nonzero angular velocity in an annulus under temporal periodic boundary conditions prescribed on the inner and outer boundaries. Our goal is to establish the existence and  stability of time-periodic solutions which are closed to  steady radially symmetric subsonic spiral flows.
\subsection{Steady radially  symmetric subsonic spiral flows }\noindent
 \par The research on steady compressible flows is important for a variety of engineering applications, such as the design of projectiles, rockets, aircraft, and so on. The    steady Euler system takes the following form:
\begin{equation}\label{1-1}
\begin{cases}
\div(\rho {\bf u})=0,\\
\div(\rho {\bf u}\otimes {\bf u}+ P I_n) =0,\\
\div(\rho{E}{\bf u}+P{\bf u})=0.\\
\end{cases}
\end{equation}
 We first focus on steady radially symmetric  flows  in an annulus
$$\Omega=\left\{(x_1,x_2): r_0<r=\sqrt{x_1^2+x_2^2}<r_1\right\},$$
 where   $0<r_0<r_1$  are  positive constants.
 Introduce the polar coordinates
$(r, \theta)$:
 \begin{eqnarray}\label{coor}
r=\sqrt{x_1^2+x_2^2},\quad \theta=\arctan \frac{x_2}{x_1}.
\end{eqnarray}
Assume that the velocity,  the pressure and the entropy  are of the form
\begin{equation*}
\begin{aligned}
{\bf u}({\textbf x})= U_{1}(r,\th)\mathbf{e}_r+U_{2}(r,\th)\mathbf{e}_\theta, \ \
P({\textbf x})=P(r,\th),\ \ S({\textbf x})=S(r,\th),
\end{aligned}
\end{equation*}
 with
\begin{equation*}
\mathbf{e}_r =
\begin{pmatrix}
\cos \th  \\
  \sin \th \\
  \end{pmatrix}, \quad
\mathbf{e}_{\th } =
\begin{pmatrix}
  - \sin \th  \\
  \cos \th\\
  \end{pmatrix}, \quad
  \end{equation*}
where $ U_1 $ and  $ U_2  $ represent the radial velocity and the angular velocity, respectively.
Then
the steady compressible   Euler system \eqref{1-1} in the polar coordinates can be rewritten as
\begin{eqnarray}\label{1-2}
\begin{cases}
\begin{aligned}
&\partial_r(\rho U_1)+\partial_{\theta}(\rho U_2)+\frac{\rho U_1}r=0,\\
&\rho\left(U_1\partial_r +\frac{U_2}{r}\partial_{\theta}\right)U_1+\partial_r P-\frac{\rho U_2^2}r=0,\\
&\rho\left(U_1\partial_r +\frac{U_2}{r}\partial_{\theta}\right)U_2
+\frac{\partial_{\theta}P}{r}+\frac{\rho U_1U_2}{r}=0,\\
&\rho\left(U_1\partial_r +\frac{U_2}{r}\partial_{\theta}\right)S=0.\\
\end{aligned}
\end{cases}
\end{eqnarray}
Furthermore, we define the  Bernoulli function $B$   as
\begin{equation*}
B=\frac{1}{2}|{\bf u}|^2+\frac{\gamma P}{(\gamma-1)\rho}=\frac{1}{2}|{\bf u}|^2+\frac{\gamma e^S\rho^{\gamma-1}}{\gamma-1}.
\end{equation*}
\par  By the hodograph transformation, Courant and Friedrichs
    \cite[Section 104]{CF48} transformed the two-dimensional  potential equation into a linear second order differential equation on the velocity magnitude and the flow angle, and found some radially symmetric flows including circulatory flows, radial
flows and their superpositions spiral flows. These studies were further continued by the authors in \cite{WXY21}, where they studied radially symmetric transonic spiral flows on the physical plane and gave a complete classification of all possible flow patterns  with or without shocks.  The structural stability of these flows under perturbations of suitable boundary conditions in an annulus was proved in \cite{WYZ21}. The authors in \cite{WZ22} established the existence of the subsonic and subsonic-sonic spiral flow outside a porous body.  Furthermore, the authors in \cite{CMZ23}  investigated  radially
symmetric  spiral flows to the compressible Euler-Poisson system for semiconductors  in an annulus and gave a specific classification of the flow patterns. The structural stability of smooth   subsonic, supersonic and transonic spiral flows  to the steady Euler-Poisson system was established in \cite{WZ25-1,WZ25-2,WYZ24}. Recently, the authors in \cite{G26} proved
the global stability and time periodicity of radially symmetric steady supersonic spiral flows in an
annulus.

  \par For steady radially symmetric spiral flows,
 the  flow is described by smooth functions of the form
\begin{equation*}
{\bf u}({\textbf x})=\bar U_1(r)\mathbf e_r+\bar U_2(r)\mathbf e_{\th},
 \ \ P({\textbf x})=\bar P(r),\ \  S({\textbf x})=\bar S(r),\\
\end{equation*}
which solves the following system:
\begin{equation}\label{1-3}
	\begin{cases}
\begin{aligned}
		&( \bar \rho  \bar U_1)'(r)+\frac{( \bar \rho  \bar U_1)(r)}r=0 , \\
     	 &\bigg(\bar U_1 \bar U_1'+\frac{1}{\bar\rho}\bar P'\bigg)(r)-\frac{ \bar U_2^2(r)}{r}  =0,\\
    	 &(\bar U_1 \bar U_2')(r)+\frac{1}{r} (\bar U_1\bar U_2)(r)=0,\\
     	 &(\bar U_1 \bar S')(r)=0,\\
    	  \end{aligned}
	\end{cases}
\end{equation}
with the boundary conditions at the entrance $ r = r_0 $:
\begin{equation}\label{1-4}
\bar\rho(r_0)=\rho_0,\ \
\bar U_{1}(r_0)=U_{1,0},\ \
\bar U_{2}(r_0)=U_{2,0},\ \ \bar S(r_0)=S_0.
\end{equation}
Then the problem \eqref{1-3} with \eqref{1-4} is equivalent to
\begin{equation}\label{1-5}
\begin{cases}
\begin{aligned}
&(\bar \rho  \bar U_1)(r)=\frac{ r_0 \rho_0 U_{1,0}}{r}, \ \
\bar U_{2}(r)=\frac{r_0U_{2,0}}{r},\\
&\bar S(r)=  S_0, \ \
\bar B(r)= B_0=\frac{1}{2}(U_{1,0}^2+U_{2,0}^2)+\frac{ \gamma e^{S_0}\rho_0^{\gamma-1}}{\gamma-1}.\\
\end{aligned}
\end{cases}
\end{equation}
 We have the following well-posedness, which has been established in \cite{WXY21}.
\begin{proposition}\label{pro1}
    Fix $ \gamma> 1 $ and $ r_0> 0 $. Given positive constants    $\rho_0$, $ U_{1,0}$, $U_{2,0}$ and $S_0 $ satisfying
    \begin{equation}\label{1-6-in}
    U_{1,0}^2+U_{2,0}^2< c_0^2=\gamma e^{S_0}\rho_0^{\gamma-1},
    \end{equation}
    then
  the initial value problem \eqref{1-3} and \eqref{1-4} has a unique smooth  solution $(\bar\rho,\bar  U_{1}, \bar U_{2},  \bar S)$ $[r_0,+\infty) $  satisfying
  \begin{equation}\label{1-6}
   (\bar U_1^2+\bar U_2^2)(r)< \bar c^2(r), \ \  r\in[r_0,+\infty) ,
  \end{equation}
  where $\bar c^2=\gamma e^{S_0}\bar\rho^{\gamma-1}$.
\end{proposition}
\par For any $ r_1>r_0 $, let
 $(\bar\rho,\bar  U_{1}, \bar U_{2},  \bar S)$ be the background subsonic  solutions to \eqref{1-3}  in $\Omega $  corresponding to the entrance data $(\rho_0,U_{1,0},U_{2,0},S_0)$. This paper is going to establish the existence and stability of  subsonic  time-periodic solutions with nonzero angular velocity on the perturbation of the background subsonic solutions.
 \subsection{Mathematical problem and main results}\noindent
 For unsteady radially symmetric flows, \eqref{1-2} can be rewritten  as
\begin{equation}\label{1-7}
	\begin{cases}
\begin{aligned}
		&\p_t\rho+\partial_r(\rho U_1)+\frac{\rho U_1}r=0,\\
&\p_tU_1+ U_1\partial_r U_1+\frac{\partial_r P}{\rho}-\frac{ U_2^2}r=0,\\
&\p_tU_2+ U_1\partial_rU_2
+\frac{ U_1U_2}{r}=0,\\
&\p_tS+U_1\partial_r S=0.\\
    	  \end{aligned}
	\end{cases}
\end{equation}
Set $\bm {\Phi}(t,r)=(\rho,  U_{1},  U_{2},   S)^{\top}(t,r)$. Then \eqref{1-7} can be rewritten as
\begin{equation}\label{1-8}
\bm {\Phi}_t +A_1(\bm {\Phi})\p_{r}\bm {\Phi}+A_2(\bm {\Phi})
=0,
\end{equation}
where
\begin{equation*}
A_1(\bm {\Phi})=\left(\begin{array}{cccc}
U_1 & \rho & 0 & 0\\
\gamma e^S\rho^{\gamma-2} & U_1 &0& \rho^{\gamma-1} \\
 0& 0  & U_1& 0\\
 0& 0 & 0 & U_1
 \end{array}\right),\quad A_2(\bm {\Phi})=\left(\begin{array}{cccc}
\frac{\rho U_1}{ r} \\
 \frac{ U_2^2}{ r} \\
 \frac{ U_1U_2}{ r}\\
 0\\
\end{array}\right).
 \end{equation*}
 The eigenvalues  are
\begin{equation*}
\lambda_1=U_1-c,\ \  \lambda_2=\lambda_3=U_1, \ \ \lambda_4=U_1+c.
\end{equation*}
Set
$$\bar\lambda_1=\bar U_1-\bar c<0, \ \
\bar\lambda_2=\bar\lambda_3=\bar U_1>0,\ \
\bar\lambda_4=\bar U_1+\bar c>0.$$
  Then one has
  \begin{equation}\label{1-8-ei}
\lambda_1<0 < \lambda_2=\lambda_3< \lambda_4, \ \ \bm {\Phi}\in \mathfrak{W}
\end{equation}
for a small neighborhood of  $\mathfrak{W} $ of $(\x\rho,  \x U_{1}, \x U_{2}, \x  S)^{\top}$.

\par Define the generalized Riemann invariants
\begin{equation}\label{1-9}
R_1=U_1-\frac{2c}{\gamma-1}+\mathfrak {c}(S-S_0), \  \ R_2=U_2,\ \ R_3=S, \ \ R_4=U_1+\frac{2c}{\gamma-1}-\mathfrak {c}(S-S_0),
\end{equation}
where $ \mathfrak {c}(r)=\frac{\bar c(r)}{\gamma(\gamma-1)}$.
Then the system \eqref{1-7} can be rewritten  into the following diagonal form:
 \begin{equation}\label{1-10}
	\begin{cases}
\begin{aligned}
		&\p_t R_1+\lambda_1\p_r R_1=
 \frac{c U_1}{r} + \frac{U_2^2}{r} + \frac{(c-\bar{c})c}{\gamma(\gamma-1)} \partial_r S + (S-S_0) \lambda_1\mathfrak{c}'(r) \\
 &=\frac{\gamma-1}{8r} (R_4^2 - R_1^2) + \frac{(\gamma-1)\mathfrak{c}}{4r} (S-S_0)(R_1 + R_4) + \frac{R_2^2}{r} \\
 &\quad+ \frac{(c-\bar{c})c}{\gamma(\gamma-1)} \partial_r R_3+ (S-S_0) \lambda_1\mathfrak{c}'(r),\\
&\p_t R_2+\lambda_2\p_r R_2=-\frac{U_1U_2}{2r}=-\frac{R_2(R_1+R_4)}{2r},\\
&\p_t R_3+\lambda_3\p_r R_3=0,\\
&\p_t R_4+\lambda_4\p_r R_4
 =- \frac{c U_1}{r} + \frac{U_2^2}{r} + \frac{(c-\bar{c})c}{\gamma(\gamma-1)} \partial_r S - (S-S_0)\lambda_4 \mathfrak{c}'(r). \\
&=-\frac{\gamma-1}{8r} (R_4^2 - R_1^2) - \frac{(\gamma-1)\mathfrak{c}}{4r} R_3(R_1 + R_4) + \frac{R_2^2}{r} \\
&\quad+   \frac{(c-\bar{c})c}{\gamma(\gamma-1)}   \partial_r R_3- (S-S_0) \lambda_4\mathfrak{c}'(r).\\
    	  \end{aligned}
	\end{cases}
\end{equation}
We also write \eqref{1-3} into the equations of Riemann
invariants:
\begin{equation}\label{1-11}
	\begin{cases}
\begin{aligned}
		&\bar\lambda_1 \bar R_1'=\frac{\bar c \bar U_1}{r} + \frac{\bar U_2^2}{r}=\frac{\bar R_2^2}r+\frac{\gamma-1}{8r}(\bar R_4^2-\bar R_1^2),\\
&\bar\lambda_2\bar  R_2'=-\frac{\bar U_1 \bar U_2}{2r},\\
 &\bar R_3=S_0,\\
&\bar\lambda_4 \bar R_4'=-\frac{\bar c \bar U_1}{r} + \frac{\bar U_2^2}{r}=\frac{\bar R_2^2}r-\frac{\gamma-1}{8r}(\bar R_4^2-\bar R_1^2),
    	  \end{aligned}
	\end{cases}
\end{equation}
where
\begin{equation*}
	\begin{aligned}
 \bar R_1=\bar U_1-\frac{2\bar c}{\gamma-1}, \  \ \bar R_2=\bar U_2,\ \ \bar R_3= S_0, \ \ \bar R_4=\bar U_1+\frac{2\bar c}{\gamma-1}.
  \end{aligned}
	\end{equation*}
Furthermore, the initial data and boundary conditions of \eqref{1-10} can be imposed as follows:
\begin{equation}\label{1-12}
	\begin{cases}
t=0: (R_{1},R_2,R_3,R_4)(0,r)=(R_{1,0},R_{2,0},R_{3,0},R_{4,0})(r),\\
r=r_0: \begin{cases}
R_2(t,r_0) = G_{2}(t)+K_2
(R_{1}(t,r_0)-\x{R}_{1}(r_0)),\\
R_3(t,r_0) =  G_{3}(t)+K_3
(R_{1}(t,r_0)-\x{R}_{1}(r_0)), \\
R_4(t,r_0) = G_{4}(t)+K_4
(R_{1}(t,r_0)-\x{R}_{1}(r_0)), \\
\end{cases} \\
r=r_1: R_{1}(t,r_1)=G_{1}(t)-(R_{4}(t,r_1)-\x R_{4}(r_1)),\\
 \end{cases}
	\end{equation}
where  $(K_2,K_3,K_4)$ are given  constants  satisfying
\begin{equation}\label{1-12-ds}
   |K_{i}|<1, \ \ i=2,3,4,
\end{equation}
and  $(G_{1},G_2,G_3,G_4)(t)$ are $C^{1}$ time-periodic functions  with the period $ T>0$.

\par Let
\begin{equation*}
	\begin{aligned}
\bm\Psi=(\Psi_{1},\Psi_{2},\Psi_{3},\Psi_4)^{\top}
=(R_{1}-\x R_{1},R_{2}-\x R_{2},R_{3}-\x R_{3},R_{4}-\x R_{4})^{\top},\
\bar{\bf R}=(\x R_{1},\x R_{2},\x R_{3},\x R_{4})^{\top}.
 \end{aligned}
	\end{equation*}
 Then one has
\begin{equation*}
	\begin{aligned}
&\lambda_1(\bm\Psi+\bar{\bf R})=U_1-c=\frac{\gamma+1}4(\Psi_1+\bar R_1-\mathfrak {c}\Psi_3)+\frac{3-\gamma}4(\Psi_4+\bar R_4+\mathfrak {c}\Psi_3),\\
&\lambda_2(\bm\Psi+\bar{\bf R})=\lambda_3(\bm\Psi+\bar{\bf R})
=U_1=\frac12(\Psi_1+\bar R_1+\Psi_4+\bar R_4),\\
&\lambda_4(\bm\Psi+\bar{\bf R})=U_1+c=\frac{3-\gamma}4(\Psi_1+\bar R_1-\mathfrak {c}\Psi_3)+\frac{\gamma+1}4(\Psi_4+\bar R_4+\mathfrak {c}\Psi_3).\\
  \end{aligned}
	\end{equation*}
Furthermore, it follows from the third equation  in \eqref{1-10} that
\begin{equation*}\begin{aligned} &\partial_t \Psi_3 + \lambda_1(\bm\Psi+\bar{\bf R}) \partial_r \Psi_3=-c\partial_r\Psi_3,\ \partial_t \Psi_3 + \lambda_4(\bm\Psi+\bar{\bf R}) \partial_r \Psi_3=c\partial_r \Psi_3.\end{aligned}\end{equation*}
Thus $ \Psi_i $ satisfies
\begin{equation}\label{1-13-mo-1}
\begin{cases}
\begin{aligned}
	&\partial_t \Psi_1 + \lambda_1(\bm\Psi+\bar{\bf R}) \partial_r \Psi_1
= \bigg( -\frac{\gamma+1}{4}\bar R_1' - \frac{\gamma-1}{4r}\bar R_1 \bigg)\Psi_1 + \frac{2\bar R_2}{r} \Psi_2  \\
&\  + \bigg( -\frac{3-\gamma}{4}\bar R_1' + \frac{\gamma-1}{4r}\bar R_4 \bigg) \Psi_4+ \frac{\Psi_2^2}{r}+ \frac{\gamma-1}{8r}(\Psi_4^2 - \Psi_1^2) \\
&\ +\frac{(\gamma-1)\mathfrak {c} }{4r}\Psi_3(\Psi_1+ \Psi_4) +\bigg(\frac{\gamma-1}{2}\bar R_1' + \frac{\gamma-1}{4r}(\bar R_1+\bar R_4)\bigg)\mathfrak {c}\Psi_3\\
 &\ +\lambda_1( \bm\Psi+\bar{\bf R})\mathfrak {c}'(r)\Psi_3-\bigg(\frac{{\Psi}_4 - {\Psi}_1}{4\gamma} + \frac{(\gamma-1)\mathfrak{c}\Psi_3}{2}\bigg)\bigg(\partial_t \Psi_3 + \lambda_1( \bm\Psi+\bar{\bf R}) \partial_r \Psi_3\bigg),\\
 &\partial_t \Psi_2 +  \lambda_2( \bm\Psi+\bar{\bf R}) \partial_r \Psi_2
=-\frac{\lambda_2(\bm\Psi+\bar{\bf R})}r\Psi_2,\\
&\partial_t \Psi_3 +  \lambda_3(\bm\Psi+\bar{\bf R}) \partial_r \Psi_3=0,\\
&\partial_t {\Psi}_4 + \lambda_4( \bm\Psi+\bar{\bf R}) \partial_r {\Psi}_4 =   \left( -\frac{3-\gamma}{4} \bar R_4' + \frac{\gamma-1}{4r}\bar R_1 \right) {\Psi}_1 + \frac{2\bar R_2}{r} \Psi_2   \\
&\  + \bigg( -\frac{\gamma+1}{4} \bar R_4' - \frac{\gamma-1}{4r}\bar R_4 \bigg) \Psi_4+ \frac{\Psi_2^2}{r} - \frac{\gamma-1}{8r}(\Psi_4^2 - \Psi_1^2)\\
&\  -\frac{(\gamma-1)\mathfrak {c} }{4r }\Psi_3(\Psi_1+ \Psi_4) -\bigg(\frac{\gamma-1}{2}\bar R_4' + \frac{\gamma-1}{4r}(\bar R_1+\bar R_4)\bigg)\mathfrak {c}\Psi_3 \\
&\ -\lambda_4( \bm\Psi+\bar{\bf R})\mathfrak {c}'\Psi_3 +\bigg(\frac{{\Psi}_4 - {\Psi}_1}{4\gamma} + \frac{(\gamma-1)\mathfrak{c}\Psi_3}{2}\bigg)\bigg(\partial_t \Psi_3 + \lambda_4(  \bm\Psi+\bar{\bf R})\partial_r \Psi_3\bigg).\\
\end{aligned}
\end{cases}
\end{equation}
The corresponding initial data and boundary conditions are
\begin{equation}\label{1-14}
\begin{cases}
t=0: (\Psi_{1},\Psi_{2},\Psi_{3},\Psi_4)(0,r)=(\Psi_{1,0},\Psi_{2,0},\Psi_{3,0},
\Psi_{4,0})(r)\\
r=r_0: \begin{cases}
\Psi_2(t,r_0) = H_{2}(t)+K_2
\Psi_{1}(t,r_0),\\
\Psi_3(t,r_0) = H_{3}(t)+K_3
\Psi_{1}(t,r_0),\\
\Psi_4(t,r_0) = H_{4}(t)+K_4
\Psi_{1}(t,r_0),\\
\end{cases} \\
r=r_1: \Psi_{1}(t,r_1)=H_{1}(t)-\Psi_{4}(t,r_1),
\end{cases}
\end{equation}
where \begin{equation*}\begin{aligned}
& (\Psi_{1,0},\Psi_{2,0},\Psi_{3,0},
\Psi_{4,0})(r)=(R_{1,0}-\x R_{1},R_{2,0}-\x R_{2},R_{3,0}-\x R_{3},R_{4,0}-\x R_4)(r),\\
&H_{1}(t)=G_{1}(t)-\x{\Psi}_{1}(r_1),  \ \ (H_{2},H_{3},H_4)(t)=(G_{2},G_{3},G_4)(t)
-(\x{\Psi}_{2},\x{\Psi}_{3},\x{\Psi}_{4})(r_0).\end{aligned}\end{equation*}
Note that
\begin{equation*}
\lambda_{1}(\bar{\bf R})<0<\lambda_{2}(\bar{\bf R})=\lambda_{3}(\bar{\bf R})
<\lambda_{4}(\bar{\bf R}).
\end{equation*}
By \eqref{1-8-ei}, one gets
$$
\lambda_{1}(\bm\Psi+\bar{\bf R})<0
<\lambda_{2}(\bm\Psi+\bar{\bf R})=\lambda_{3}(\bm\Psi+\bar{\bf R})
<\lambda_{4}(\bm\Psi+\bar{\bf R}), \ \ \bm\Psi\in \mathfrak{V},
$$
where $\mathfrak{V}$ is a small neighborhood of $O=(0,0,0,0)^{\top}$ corresponding to $\mathfrak{W}$. Denote $$\mu_{i}(\bm\Psi+\bar{\bf R})=\lambda_{i}^{-1}(\bm\Psi+\bar{\bf R}), \ \ i=1,2,3,4.$$ Then there exists a positive constant $\mk$ depending only on $( \gamma,r_0,r_1,\rho_0,U_{1,0},U_{2,0},S_0) $  such that
\begin{equation}\label{1-15}
\max_{i=1,2,3,4}\sup_{\bm\Psi\in \mathfrak{V}}|\mu_{i}(\bm\Psi+\bar{\bf R})|\leq \mk.
\end{equation}
\par  The main results in this paper are stated as follows.

\begin{theorem}\label{th1}
(Existence of the time-periodic solution) Let $(\bar\rho,\bar  U_{1}, \bar U_{2},  \bar S)$ be the background subsonic solutions associated with the   data $(\rho_0,U_{1,0},U_{2,0},S_0) $. There exists a positive constant $\xi_0$ depending only on the $( K_2,K_3,K_4,\gamma,\rho_0,U_{1,0},U_{2,0},S_0) $ such that for
\begin{equation}\label{1-16-re}
r_1\in\Big(r_0,r_0(1+\xi_0)\Big),
\end{equation}
the  following holds: there exist  positive  constants $\sigma_{1}$ and  $\mc_1$ such that for any  $0<\sigma<\sigma_{1}$ and  any  $C^{1}$ functions $(H_{1},H_{2},H_3,H_4)(t)$ defined in \eqref{1-14} satisfying
\begin{equation}\label{1-16}
H_{i}(t+T)=H_{i}(t),\ \ t>0, \ \
\|H_{i}\|_{C^{1}(\mathbb{R}_{+})}\leq \sigma,
\end{equation}
there exists a $C^{1}$ smooth function ${\bm\Psi}_{0}(r)=\bigg(\Psi_{1,0}(r),\Psi_{2,0}(r),\Psi_{3,0}(r),
\Psi_{4,0}(r)\bigg)^{\top}$ satisfying
\begin{equation}\label{1-19}
\|\Psi_{i,0}\|_{C^{1}([r_0,r_1])}\leq \mc_1\sigma,
\end{equation}
such that the initial-boundary value problem \eqref{1-13-mo-1}  and \eqref{1-14} admits a $C^{1}$ time-periodic solution ${\bm\Psi}^{(T)}(t,r)=\bigg(\Psi_{1}^{(T)},\Psi_{2}^{(T)},\Psi_{3}^{(T)},
\Psi_{4}^{(T)}\bigg)^{\top}$ on $D=\bigg\{(t,r):t\in\mathbb{R}_{+},r\in[r_0,r_1]\bigg\}$, which satisfies
\begin{equation}\label{1-21}
{\bm\Psi}^{(T)}(t+T,r)={\bm\Psi}^{(T)}(t,r),\ \ \forall(t,r)\in D,
\end{equation}
and
\begin{equation}\label{1-22}
\|\Psi_{i}^{(T)}\|_{C^{1}(D)}\leq \mc_1\sigma.
\end{equation}
%\begin{equation}\label{1-22-hy}
%\|\Psi_{2}^{(T)}\|_{C^{0}(D)}\leq \mc_1M_{2},\ \ \max\bigg\{\|\partial_{t}\Psi_{2}^{(T)}\|_{C^{0}(D)},\ \ \|\partial_{r}\Psi_{2}^{(T)}\|_{C^{0}(D)}\bigg\}\leq\mc_1\sigma,
%\end{equation}
\end{theorem}
\begin{theorem}\label{th2}($C^{0}$ Stability of the time-periodic solution) There exist positive  constants $\sigma_{2}\in(0,\sigma_{1})$ and $\mc_2$ such that for any given $\sigma\in(0,\sigma_{2})$ and any given $C^{1}$ smooth functions ${\bm\Psi}_{0}(r)$ and $H_{i}(t)$ $(i=1,2,3,4)$ satisfying \eqref{1-16} with certain compatibilities, the initial-boundary value problem \eqref{1-13-mo-1} and \eqref{1-14} has a unique global $C^{1}$ classical solution ${\bm\Psi}(t,r)$ on $D$ satisfying
\begin{equation}\label{1-24}
\|\bm\Psi(t,\cdot)-{\bm\Psi}^{(T)}(t,\cdot)\|_{C^{0}([r_0,r_1])}\leq \mc_2\sigma\varepsilon^{[t/T_{0}]},\ \ \forall t\geq 0,
\end{equation}
where ${\bm\Psi}^{(T)}$, depending on $(H_{1},H_{2},H_3,H_4)(t)$, is the time-periodic solution given through Theorem \ref{th1}, $\varepsilon\in(0,1)$ is a constant, $T_{0}=\max\limits_{i=1,2,3,4}\sup\limits_{\bm\Psi\in \mathfrak{V}}\frac{r_1-r_0}{|\lambda_{i}(\bm\Psi+\bar{\bf R})|}$ and $[t/T_{0}]$ denotes the largest integer smaller than $t/T_{0}$.
\end{theorem}
The uniqueness of the time-periodic solution is a direct consequence from Theorem \ref{th2} by taking $t\to+\infty$.

\begin{corollary}\label{co1}(Uniqueness of the time-periodic solution) There exists a constant $\sigma_{3}\in(0,\sigma_{2})$, such that for any given $\sigma\in(0,\sigma_{3})$ and any given $C^{1}$ smooth functions $H_{i}(t)$ $(i=1,2,3,4)$ satisfying \eqref{1-16}, the corresponding time-periodic solution $\bm\Psi={\bm\Psi}^{(T)}(t,r)$ obtained in Theorem \ref{th1} is unique.
\end{corollary}
\begin{theorem}\label{th3} (Regularity of the time-periodic solution) If $(H_{1},H_{2},H_3,H_4)(t)$  satisfy \eqref{1-16} and possess further $W^{2,\infty}$ regularity with
\begin{equation}\label{1-25}
\max_{i=1,2,3,4}\|H^{\prime\prime}_{i}\|_{L^{\infty}(\mathbb{R}_{+})}\leq M_3<+\infty,
\end{equation}
then there exist constants $\mc_3>0$ and $\sigma_{4}\in(0,\sigma_{1})$, such that for any given $\sigma\in(0,\sigma_{4})$, the time-periodic solution $\bm\Psi={\bm\Psi}^{(T)}(t,r)$ provided by Theorem \ref{th1} is also a $W^{2,\infty}$ function with
\begin{equation}\label{1-26}
\max\bigg\{\|\partial_{t}^{2}{\bm\Psi}^{(T)}
\|_{L^{\infty}(D)},\|\partial_{t}\partial_{r}{\bm\Psi}^{(T)}
\|_{L^{\infty}(D)},\|\partial_{r}^{2}{\bm\Psi}^{(T)}\|_{L^{\infty}(D)}\bigg\}
\leq(1+\mk)^{2}\mc_3<+\infty,
\end{equation}
where $\mk$ is defined in \eqref{1-15}.
\end{theorem}
\begin{theorem}\label{th4} ($C^{1}$ Stability of the time-periodic solution) Assume that \eqref{1-25} holds, then there exist constants $\mc_4>\mc_2$ and $\sigma_{5}\in(0,\min\{\sigma_{2},\sigma_{4}\})$, such that for any given $\sigma\in(0,\sigma_{5})$, we have not only the $C^{0}$ convergence result  in Theorem \ref{th2}, but also the $C^{1}$ exponential convergence as
\begin{equation}\label{1-26-es}
\max\bigg\{\|\partial_{t}\bm\Psi(t,\cdot)-\partial_{t}
{\bm\Psi}^{(T)}(t,\cdot)\|_{C^{0}([r_0,r_1])},\|\partial_{r}\bm\Psi(t,\cdot)-\partial_{r}
{\bm\Psi}^{(T)}(t,\cdot)\|_{C^{0}([r_0,r_1])}\bigg\}
\leq(1+\mk)
\mc_4\sigma\varepsilon^{[t/T_{0}]},\ \forall t\geq 0.
\end{equation}
 \end{theorem}
 \begin{remark}\label{re-1}
 {\it The boundary condition \eqref{1-12} on the outer circle is non-dissipative, which has physical significance. Indeed, it simplifies to
$$ \bar U_1(t,r_1)=\frac12\bigg(G_{1}(t)+\x R_{4}(r_1)\bigg).$$
 which represents a physically admissible boundary condition. }\end{remark}
 \begin{remark}\label{re-2}
 {\it The main contribution of this paper is to establish the existence and stability of time-periodic solutions without imposing any smallness assumptions on the steady radially symmetric subsonic spiral flows. This is in contrast with several recent works \cite{FMY24,QYZ23-1,ZY25}, which require both a smallness assumption on the background flow and the restriction $1<\gamma<3$. The reason lies in the fundamentally different structures of the background flows. For non-constant background flows, the linearized system typically takes the form of a hyperbolic balance law with  variable coefficients, which may generate growth effects and weaken the dissipative structure. As shown in \cite{QYZ23-1,ZY25}, the derivative of the background solution is directly related to the damping coefficient and gives rise to potentially amplifying terms in the perturbation system. To ensure that these growth effects are dominated by the boundary dissipation, a smallness assumption on the background flow is therefore indispensable. In contrast,
for spiral flows,  the background solution and its derivatives are completely determined by the radial geometry and the initial data. Consequently, the variable coefficients arising from the background state remain uniformly bounded and do not produce uncontrollable growth. This intrinsic stabilising mechanism,  together with a suitable restriction on the width of the annulus, enables us to control the accumulation of perturbations along characteristic curves and thereby remove the smallness assumption entirely.
To the best of our knowledge, this is the first result on subsonic time-periodic solutions for compressible Euler flows with nonzero angular velocity in an annulus that does not require any smallness condition on the background flow.
}
 \end{remark}
\begin{remark}\label{re-3}
 {\it We hope
that the results and the analysis in this paper would  inspire the  research on
time-periodic transonic shock solutions to the radially symmetric spiral  flows   in an annulus. This is also one of the principal
goals in the forthcoming work  and  will be reported in \cite{YZ26}.}\end{remark}
\subsection{Key ideas for the proof of main results}\label{pr}\noindent
\par Our approach is based on the characteristic decomposition  and an  elaborate iterative scheme.  The main difficulties arise from the lack of dissipation on the outer boundary, the derivative loss caused by the coupling
with the radial entropy derivative and the additional zeroth-order terms induced by the nonzero angular velocity.
To overcome  these difficulties,
the iteration scheme is constructed according to the characteristic directions. Specifically, $R_1$ is solved  by integrating along the backward characteristics, while $R_2$, $R_3$ and $R_4$ are propagated from the inner boundary, where dissipative boundary conditions are imposed. Furthermore, Suitable integrating factors are introduced to define weight functions, which enable us to control the accumulation of perturbations and incorporate the boundary dissipation into a priori estimates. The energy estimates are then established by the monotonicity of these weight functions, the smallness of the boundary data, and a suitable restriction on the annulus width to prevent excessive perturbation growth along characteristics. A key observation is that the derivatives of the background solution remain uniformly bounded by the conservation laws, without any smallness assumption on the background flow.
Once the time-periodic solution is constructed, its stability follows from a contraction argument over successive time intervals, which yields exponential decay. Higher regularity is obtained by a compactness argument when the boundary data possess additional regularity.
\par This paper will be arranged as follows.  In Section 2, we establish the existence of time-periodic solutions by constructing an iterative scheme. The stability and uniqueness of the time-periodic solution are proved in Section 3. Section 4 and Section 5 are devoted to the higher regularity and stability of the time-periodic solutions under the assumption that the boundary data possess higher regularity.
\section{Existence of the time-periodic solutions}\noindent
\par Firstly, by multiplying $$\mu_{i}(\bm\Psi+\bar{\bf R})=\lambda_{i}^{-1}(\bm\Psi+\bar{\bf R})$$ on both sides of \eqref{1-13-mo-1}, interchanging  the positions of $t$ and $r$  and using $$\mathfrak {c}'=\frac1{2\gamma}\frac{\bar c(\bar U_1^2+\bar U_2^2)}{r(\bar c^2-\bar U_1^2)}, $$  one has
\be\no
	&&\partial_r \Psi_1+\mu_{1}(\bm\Psi+\bar{\bf R})\partial_t \Psi_1
= \mu_{1}(\bm\Psi+\bar{\bf R})\bigg(\bigg( -\frac{\gamma+1}{4}\bar R_1' - \frac{\gamma-1}{4r}\bar R_1 \bigg)\Psi_1 + \frac{2\bar R_2}{r} \Psi_2  \\\no
&&\quad+ \bigg( -\frac{3-\gamma}{4}\bar R_1' + \frac{\gamma-1}{4r}\bar R_4 \bigg) \Psi_4  + \frac{\Psi_2^2}{r} + \frac{\gamma-1}{8r}(\Psi_4^2 - \Psi_1^2)+\frac{(\gamma-1)\mathfrak  c}{4r }\Psi_3(\Psi_1+ \Psi_4)\bigg) \\\no
&&\quad +\mu_1(\bm\Psi+\bar{\bf R})\bigg(\frac{\gamma-1}{2}\bar R_1' + \frac{\gamma-1}{4r}(\bar R_1+\bar R_4)\bigg)\mathfrak  c\Psi_3+\mathfrak {c}'(r)\Psi_3 \\\no
  &&\quad -\bigg(\frac{{\Psi}_4 - {\Psi}_1}{4\gamma} + \frac{(\gamma-1)\mathfrak {c}\Psi_3}{2}\bigg)\bigg(  \partial_r \Psi_3+\mu_1(\bm\Psi+\bar{\bf R})\partial_t \Psi_3\bigg)\\\no
&&=-\frac{ (\gamma+1)\bar R_2^2 + \dfrac{\gamma-1}{8}\bigg(2(\bar R_1+\bar R_4)^2 + (\gamma-1)(\bar R_4-\bar R_1)^2\bigg) }{ r\big( (\gamma+1)\bar R_1 + (3-\gamma)\bar R_4 \big) }\mu_1(\bm\Psi+\bar{\bf R})\Psi_1\\\no
&&\quad +\frac{ (\gamma-3)\bar R_2^2 + \dfrac{\gamma-1}{8}\bigg(2(\bar R_1+\bar R_4)^2 - (\gamma-1)(\bar R_4-\bar R_1)^2\bigg) }{ r\big( (\gamma+1)\bar R_1 + (3-\gamma)\bar R_4 \big) }\mu_1(\bm\Psi+\bar{\bf R})\Psi_4\\\no
&&\quad+\bigg(\frac{2\bar U_2}{r} \Psi_2+\frac{\Psi_2^2}{r}+ \frac{\gamma-1}{8r}(\Psi_4^2 - \Psi_1^2)+\frac{\bar  c}{4\gamma r }\Psi_3(\Psi_1+ \Psi_4)\bigg)\mu_{1}(\bm\Psi+\bar{\bf R})\\\no
&&\quad +\mu_1(\bm\Psi+\bar{\bf R})\frac{\bar c}{2\gamma}\frac{\bar U_1^2+\bar U_2^2}{r(\bar c-\bar U_1)}\Psi_3+\frac{\bar c}{2\gamma}\frac{\bar U_1^2+\bar U_2^2}{r(\bar c^2-\bar U_1^2)}\Psi_3, \\\no
  &&\quad -\bigg(\frac{{\Psi}_4 - {\Psi}_1}{4\gamma} + \frac{\bar{c}\Psi_3}{2\gamma}\bigg)\bigg(  \partial_r \Psi_3+\mu_1(\bm\Psi+\bar{\bf R})\partial_t \Psi_3\bigg),\\\no
&&=\frac1r\bigg(-(\ma_{11}+\ma_{12})(r)\Psi_1+(\ma_{13}-\ma_{11}-\ma_{12})(r)\Psi_4
+\ma_{14}(r)\Psi_2
+\ma_{15}(r)\Psi_3\bigg)\\\no
&&\quad+\bigg(\mu_1(\bm\Psi+\bar{\bf R})-\mu_1(\bar{\bf R})\bigg)\bigg(-\mm_{11}(r)\Psi_1+\mm_{12}(r)\Psi_2+\mm_{13}(r)\Psi_3+\mm_{14}(r)\Psi_4\bigg)\\\no
&&\quad+\mu_1(\bm\Psi+\bar{\bf R})\bigg(\frac{\Psi_2^2}{r}+\frac{\gamma-1}{8r}(\Psi_4^2 - \Psi_1^2)+\frac{\bar c}{4\gamma r }\Psi_3(\Psi_1+ \Psi_4)\bigg)\\\label{3-1}
&&\quad -\bigg(\frac{{\Psi}_4 - {\Psi}_1}{4\gamma} + \frac{\bar{c}\Psi_3}{2\gamma}\bigg)\bigg(  \partial_r \Psi_3+\mu_1(\bm\Psi+\bar{\bf R})\partial_t \Psi_3\bigg),\\\no
&&\partial_r \Psi_4 + \mu_4(\bm\Psi+\bar{\bf R}) \partial_t \Psi_4
=
\mu_4(\bm\Psi+\bar{\bf R}) \bigg(\bigg( -\frac{3-\gamma}{4} \bar R_4' + \frac{\gamma-1}{4r}\bar R_1 \bigg) \Psi_1 +  \frac{2\bar R_2}{r} \Psi_2 \\\no
&&\quad +\bigg( -\frac{\gamma+1}{4} \bar R_4' - \frac{\gamma-1}{4r}\bar R_4 \bigg) \Psi_4  + \frac{\Psi_2^2}{r} - \frac{\gamma-1}{8r}(\Psi_4^2 - \Psi_1^2)
-\frac{(\gamma-1)\mathfrak {c} }{4r }\Psi_3(\Psi_1+ \Psi_4)\bigg)\\\no
 &&\quad -\mu_4(\bm\Psi+\bar{\bf R})\bigg(\frac{\gamma-1}{2}\bar R_4' + \frac{\gamma-1}{4r}(\bar R_1+\bar R_4)\bigg)\mathfrak {c}\Psi_3-\mathfrak {c}'\Psi_3  \\\no
&&\quad +\bigg(\frac{{\Psi}_4 - {\Psi}_1}{4\gamma} + \frac{(\gamma-1)\mathfrak{c}\Psi_3}{2}\bigg)\bigg(\partial_r \Psi_3 + \mu_4(\bm\Psi+\bar{\bf R}) \partial_t \Psi_3\bigg)\\\no \no \ee \be\no
&&=-\frac{ (\gamma+1)\bar R_2^2 + \dfrac{\gamma-1}{8}\bigg(2(\bar R_1+\bar R_4)^2 + (\gamma-1)(\bar R_4-\bar R_1)^2\bigg) }{ r\big( (\gamma+1)\bar R_4 + (3-\gamma)\bar R_1 \big) }\mu_4(\bm\Psi+\bar{\bf R})\Psi_4\\\no
&&\quad +\frac{ (\gamma-3)\bar R_2^2 + \dfrac{\gamma-1}{8}\bigg(2(\bar R_1+\bar R_4)^2 - (\gamma-1)(\bar R_4-\bar R_1)^2\bigg) }{ r\big( (\gamma+1)\bar R_4+ (3-\gamma)\bar R_1 \big) }\mu_4(\bm\Psi+\bar{\bf R})\Psi_1\\\no
&&\quad+\bigg(\frac{2\bar U_2}{r} \Psi_2+\frac{\Psi_2^2}{r}- \frac{\gamma-1}{8r}(\Psi_4^2 - \Psi_1^2)
-\frac{\bar{c} }{4\gamma r }\Psi_3(\Psi_1+ \Psi_4)\bigg)\mu_{4}(\bm\Psi+\bar{\bf R})\\\no
&&\quad -\mu_4(\bm\Psi+\bar{\bf R})\frac{\bar c}{2\gamma}\frac{\bar U_1^2+\bar U_2^2}{r(\bar c+\bar U_1)}\Psi_3-\frac{\bar c}{2\gamma}\frac{\bar U_1^2+\bar U_2^2}{r(\bar c^2-\bar U_1^2)}\Psi_3 \\\no
&&\quad +\bigg(\frac{{\Psi}_4 - {\Psi}_1}{4\gamma} + \frac{\bar{c}\Psi_3}{2\gamma}\bigg)\bigg(\partial_r \Psi_3 + \mu_4(\bm\Psi+\bar{\bf R}) \partial_t \Psi_3\bigg)\\\no
&&=\frac1r\bigg(-\ma_{41}(r)\Psi_4+(\ma_{42}-\ma_{41})(r)\Psi_1 +\ma_{43}(r) \Psi_2-\ma_{44}(r)\Psi_3\bigg)\\\no
&&\quad +\bigg(\mu_4(\bm\Psi+\bar{\bf R})-\mu_4(\bar{\bf R})\bigg)\bigg(\mm_{41}(r)\Psi_1-\mm_{44}(r)\Psi_4+\mm_{42}(r)
 \Psi_2-\mm_{43}(r)\Psi_3\bigg)\\\no
&&\quad+\mu_4( \bm\Psi+\bar{\bf R})\bigg(\frac{\Psi_2^2}{r} - \frac{\gamma-1}{8r}(\Psi_4^2 - \Psi_1^2)
-\frac{\bar {c} }{4\gamma r }\Psi_3(\Psi_1+ \Psi_4)\bigg)\\\label{3-2}
&&\quad +\bigg(\frac{{\Psi}_4 - {\Psi}_1}{4\gamma} + \frac{\bar{c}\Psi_3}{2\gamma}\bigg)\bigg(\partial_r \Psi_3 + \mu_4( \bm\Psi+\bar{\bf R}) \partial_t \Psi_3\bigg),\\ \label{3-3}
&& \partial_r (r\Psi_2)+  \mu_2( \bm\Psi+\bar{\bf R})\partial_t (r\Psi_2)
= 0,\\\label{3-4}
&& \partial_r \Psi_3+  \mu_3( \bm\Psi+\bar{\bf R})\partial_t \Psi_3
= 0,
\ee
 where
 \begin{gather*}
\ma_{11}(r)=\frac{ (\gamma+1)\bar U_2^2 +(\gamma-1)\bar U_1^2+2\bar c^2}{ 4(\bar U_1-\bar c) (\bar U_1+\bar c)}, \ \ \ma_{12}(r)=\frac{2\bar c\bigg( (\gamma+1)\bar U_2^2 +(\gamma-1)\bar U_1^2+2\bar c^2\bigg)}{ 4(\bar U_1-\bar c)^2(\bar U_1+\bar c) },\\ \ma_{13}(r)=\frac{ 2(\gamma-1)(\bar U_1^2+\bar U_2^2)}{ 4(\bar U_1-\bar c)^2 },\ \ma_{14}(r)=\frac{2\bar U_2}{\bar U_1-\bar c}, \\ \ma_{15}(r)=\frac{\bar c^2(\bar U_1^2+\bar U_2^2)}{\gamma (\bar c+\bar U_1)(\bar c-\bar U_1)^2},\ \
 \ma_{41}(r)=\frac{ (\gamma+1)\bar U_2^2 +(\gamma-1)\bar U_1^2+2\bar c^2}{ 4(\bar U_1+\bar c)^2 },\\ \ma_{42}(r)=\frac{ 2(\gamma-1)(\bar U_1^2+\bar U_2^2)}{ 4(\bar U_1+\bar c)^2 },\ \
\ma_{43}(r)=\frac{2\bar U_2}{\bar U_1+\bar c},\\ \ma_{44}(r)=\frac{\bar c^2(\bar U_1^2+\bar U_2^2)}{\gamma (\bar c+\bar U_1)^2(\bar c-\bar U_1)},\ \  \
\mm_{11}(r)=\frac{ (\gamma+1)\bar U_2^2 +(\gamma-1)\bar U_1^2+2\bar c^2}{ 4r(\bar U_1-\bar c) },\\ \mm_{12}(r)=\mm_{42}(r)=\frac{2\bar U_2}{r}, \ \ \
\mm_{13}(r)=\frac{\bar c}{2\gamma}\frac{\bar U_1^2+\bar U_2^2}{r(\bar c-\bar U_1)}, \\ \mm_{14}(r)=\frac{ (\gamma-3)\bar U_2^2+(\gamma-1)\bar U_1^2 -2\bar c^2}{ 4r(\bar U_1-\bar c)  },\ \ \mm_{41}(r)=\frac{ (\gamma-3)\bar U_2^2+(\gamma-1)\bar U_1^2 -2\bar c^2}{ 4r(\bar U_1+\bar c)  },\\
\mm_{43}(r)=\frac{\bar c}{2\gamma}\frac{\bar U_1^2+\bar U_2^2}{r(\bar c+\bar U_1)},\ \ \mm_{44}(r)=\frac{ (\gamma+1)\bar U_2^2 +(\gamma-1)\bar U_1^2+2\bar c^2}{ 4r(\bar U_1+\bar c) }.
 \end{gather*}
\par Next, we consider the following initial-value problem of the linearized system:\begin{equation}\label{3-6}
\begin{cases}
\begin{aligned}
	&\partial_r \Psi_4^{(k)}+\mu_4(\bm\Psi^{(k-1)}+\bar{\bf R})\partial_t \Psi_4^{(k)}+\frac{\ma_{41}(r)}r\Psi_4^{(k)}\\
&=\frac1r\bigg(
(\ma_{42}-\ma_{41})(r)\Psi_1^{(k-1)} +\ma_{43}(r) \Psi_2^{(k-1)}
 -\ma_{44}(r)\Psi_3^{(k-1)}\bigg)\\
 &\ \ +\bigg(\mu_{4}( \bm\Psi^{(k-1)}+\bar{\bf R}) -\mu_4(\bar{\bf R})\bigg)\bigg(\mm_{41}(r)\Psi_1^{(k-1)}-\mm_{44}(r)\Psi_4^{(k-1)}+\mm_{42}(r)
 \Psi_2^{(k-1)} \\
&\quad -\mm_{43}(r)\Psi_3^{(k-1)}\bigg)+\mu_4( \bm\Psi^{(k-1)}+\bar{\bf R})\bigg(\frac{(\Psi_2^{(k-1)})^2}{r} - \frac{\gamma-1}{8r}\bigg((\Psi_4^{(k-1)})^2 - (\Psi_1^{(k-1)})^2\bigg) \\
&\quad -\frac{\bar {c} }{4\gamma r }\Psi_3^{(k-1)}(\Psi_1^{(k-1)}+ \Psi_4^{(k-1)})\bigg)+\bigg(\frac{{\Psi}_4^{(k-1)} - {\Psi}_1^{(k-1)}}{4\gamma}+ \frac{\bar{c}\Psi_3^{(k-1)}}{2\gamma}\bigg)\bigg(\partial_r \Psi_3^{(k-1)} \\
&\quad + \mu_4( \bm\Psi^{(k-1)}+\bar{\bf R}) \partial_t \Psi_3^{(k-1)}\bigg),\\
&r=r_0: \Psi_{4}^{(k)}(t,r_0)=  \mh_4 (t)+K_4 {\Psi}_1^{(k-1)}(t,r_0) ,
\end{aligned}
\end{cases}
\end{equation}
\begin{equation}\label{3-5}
\begin{cases}
\begin{aligned}
	&\partial_r \Psi_1^{(k)}+\mu_{1}( \bm\Psi^{(k-1)}+\bar{\bf R})\partial_t \Psi_1^{(k)}+\frac{\ma_{11}(r)}r\Psi_1^{(k)}\\
&=\frac1r\bigg(-\ma_{12}(r)\Psi_1^{(k-1)}+(\ma_{13}-\ma_{11}-\ma_{12})(r)\Psi_4^{(k)}
+\ma_{14}(r)\Psi_2^{(k-1)}
+\ma_{15}(r)\Psi_3^{(k-1)}\bigg) \\
&\  +\bigg(\mu_1(\bm\Psi^{(k-1)}+\bar{\bf R})-\mu_1(\bar{\bf R})\bigg)\bigg(-\mm_{11}(r)\Psi_1^{(k-1)}+\mm_{12}(r)\Psi_2^{(k-1)}+ \mm_{13}(r)\Psi_3^{(k-1)}
\\
 &\quad  +\mm_{14}(r)\Psi_4^{(k)}\bigg) +\mu_1( \bm\Psi^{(k-1)}+\bar{\bf R})\bigg(\frac{(\Psi_2^{(k-1)})^2}{r}+\frac{\gamma-1}{8r}\bigg((\Psi_4^{(k)})^2  - (\Psi_1^{(k-1)})^2\bigg)\\
&\quad  +\frac{\bar  c}{4\gamma r }\Psi_3^{(k-1)}(\Psi_1^{(k-1)}+ \Psi_4^{(k)})\bigg)-\bigg(\frac{{\Psi}_4^{(k)} - {\Psi}_1^{(k-1)}}{4\gamma} + \frac{\bar c\Psi_3^{(k-1)}}{2\gamma}\bigg)\bigg(  \partial_r \Psi_3^{(k-1)}\\
&\quad+\mu_1( \bm\Psi^{(k-1)}+\bar{\bf R})\partial_t \Psi_3^{(k-1)}\bigg),\\
&r=r_1: \Psi_{1}^{(k)}(t,r_1)=\mh_{1}(t)-\Psi_{4}^{(k)}(t,r_1),
\end{aligned}
\end{cases}
\end{equation}
\begin{equation}\label{3-7}
\begin{cases}
\begin{aligned}
 &\partial_r (r\Psi_2^{(k)})+  \mu_2( \bm\Psi^{(k-1)}+\bar{\bf R})\partial_t (r\Psi_2^{(k)})
= 0,\\
&r=r_0:\Psi_{2}^{(k)}(t,r_0)=  \mh_2(t)+K_2{\Psi}_1^{(k-1)}(t,r_0) ,\\
\end{aligned}
\end{cases}
\end{equation}
and
\begin{equation}\label{3-8}
\begin{cases}
\begin{aligned}
 &\partial_r \Psi_3^{(k)}+  \mu_3(\bm\Psi^{(k-1)}+\bar{\bf R})\partial_t \Psi_3^{(k)}
= 0,\\
&r=r_0:\Psi_{3}^{(k)}(t,r_0)=\mh_3(t)+K_3{\Psi}_1^{(k-1)}(t,r_0) ,\\
\end{aligned}
\end{cases}
\end{equation}
where
\begin{align*}
\mh_i(t)=
\begin{cases}
H_{i}(t),\quad t\geq0,\\
\tilde H_{1}(t),\quad t<0,
\end{cases}
\end{align*}
are the time-periodic extensions of $H_i(t)$ $(i=1,2,3,4)$.
\par Next, we start the iteration \eqref{3-6}-\eqref{3-8}  from
\begin{equation}\label{3-9}
\Psi_i^{(0)}(t,r)=0, \ \ i=1,2,3,4.
\end{equation}
\begin{proposition}\label{pr1}
There exist a small  constant $\sigma_1 > 0$, positive constants $M_1, M_2 $ and $\eta \in (0,1)$, such that for any given $\sigma \in (0, \sigma_1)$ and $k \in \mathbb{N}^+$, the following  hold:
\begin{equation}\label{3-10}
{\Psi}_i^{(k)}(t+T,r)={\Psi}_i^{(k)}(t,r),\ \ \forall(t,r)\in \mathbb{R}\times[r_0,r_1], \ \  i=1,2,3,4,
\end{equation}
\begin{equation}\label{3-11}
\mathop{\max}\limits_{i=1,2,3,4}\Big\{\|\Psi_{i}^{(k)}\|_{C^{0}(D)}, \|\p_t\Psi_{i}^{(k)}\|_{C^{0}(D)}\Big\}\leq M_1\sigma,\ \ \mathop{\max}\limits_{i=1,2,3,4} \|\p_r\Psi_{i}^{(k)}\|_{C^{0}(D)}\leq M_2\sigma,
\end{equation}
\begin{equation}\label{3-15}
\mathop{\max}\limits_{i=1,2,3,4} \|{\Psi}_i^{(k)}-{\Psi}_i^{(k-1)}\|_{C^{0}(D)}\leq M_1\sigma \eta^k,
\end{equation}
\begin{equation}\label{3-16}
\mathop{\max}\limits_{i=1,2,3,4}
\bigg\{\varpi\bigg(\delta|\partial_{t}\Psi_{i}^{(k)}\bigg)+
\varpi\bigg(\delta|\partial_{r}\Psi_{i}^{(k)}\bigg)\bigg\}
\leq\bigg(\frac{1}{3}+\frac{1}{2}\mk\bigg)\mn(\delta),
\end{equation}
where
\begin{align}
%\|\Psi^{(l)}\|_{C^{1}(D)}&\mathop{=}\limits^{def.}\mathop{\max}\limits_{i=1,2}\{\|\Psi_{i}^{(l)}\|_{C^{0}(D)},\|\partial_{t}\Psi_{i}^{(l)}\|_{C^{0}(D)}
% ,\|\partial_{x}\Psi_{i}^{(l)}\|_{C^{0}(D)}\},\notag\\
\varpi(\delta|h)&=
\mathop{\sup}\limits_{\substack{|t^{1}-t^{2}|\leq\delta\\|r^{1}-r^{2}|
\leq\delta}}|h(t^{1},r^{1})-h(t^{2},r^{2})|,\notag
\end{align}
and $\mn(\delta)$ is a continuous function of $\delta\in(0,1)$ which is independent of $k$ and satisfies
$$
\mathop{\lim}\limits_{\delta\rightarrow0^{+}}\mn(\delta)=0.
$$
\end{proposition}
\begin{proof} Assume that
\begin{equation}\label{3-10-1}
{\Psi}_i^{(k-1)}(t+T,r)={\Psi}_i^{(k-1)}(t,r),\ \ \forall(t,r)\in \mathbb{R}\times[r_0,r_1], \ \ i=1,2,3,4,
\end{equation}
\begin{equation}\label{3-11-1}
\mathop{\max}\limits_{i=1,2,3,4}\Big\{\|\Psi_{i}^{(k-1)}\|_{C^{0}(D)}, \|\p_t\Psi_{i}^{(k-1)}\|_{C^{0}(D)}\Big\}\leq M_1\sigma, \end{equation}\begin{equation}\label{3-11-1-r} \mathop{\max}\limits_{i=1,2,3,4} \|\p_r\Psi_{i}^{(k-1)}\|_{C^{0}(D)}\leq M_2\sigma,
\end{equation}
\begin{equation}\label{3-15-1}
\mathop{\max}\limits_{i=1,2,3,4}\|{\Psi}_i^{(k-1)}-{\Psi}_i^{(k-2)}\|_{C^{0}(D)}\leq M_1\sigma \eta^{k-1}, \ \ \forall k\geq 2,
\end{equation}
\begin{equation}\label{3-16-1}
\mathop{\max}\limits_{i=1,2,3,4}\varpi\bigg(\delta|\partial_{t}\Psi_{i}^{(k-1)}(\cdot,r)
\bigg)
\leq\frac{\mn(\delta)}{8[\mathcal{K}+1]},\ \ \forall r\in[r_0,r_1],
\end{equation}
\begin{equation}\label{3-17-1}
\mathop{\max}\limits_{i=1,2,3,4}\bigg\{\varpi\bigg(\delta|\partial_{t}\Psi_{i}^{(k-1)}
\bigg)+
\varpi\bigg(\delta|\partial_{r}\Psi_{i}^{(k-1)}\bigg)\bigg\}
\leq\bigg(\frac{1}{3}+\frac{1}{2}\mk\bigg)\mn(\delta),
\end{equation}
where $[\mathcal{K}+1]$ represents the integer part of $\mathcal{K}+1$ and
$$\varpi(\delta|h(\cdot,r))=\mathop{\max}\limits_{|t^{1}-t^{2}|
\leq\delta}|h(t^{1},r)-h(t^{2},r)|.$$
We will show that
\begin{equation}\label{3-10-2}
{\Psi}_i^{(k)}(t+T,r)={\Psi}_i^{(k)}(t,r),\ \ \forall(t,r)\in \mathbb{R}\times[r_0,r_1], \ \ i=1,2,3,4,
\end{equation}
\begin{equation}\label{3-11-2}
\mathop{\max}\limits_{i=1,2,3,4}\Big\{\|\Psi_{i}^{(k)}\|_{C^{0}(D)}, \|\p_t\Psi_{i}^{(k)}\|_{C^{0}(D)}\Big\}\leq M_1\sigma,\end{equation} \begin{equation}\label{3-11-2-r} \mathop{\max}\limits_{i=1,2,3,4} \|\p_r\Psi_{i}^{(k)}\|_{C^{0}(D)}\leq M_2\sigma,
\end{equation}
\begin{equation}\label{3-15-2}
\mathop{\max}\limits_{i=1,2,3,4}\|{\Psi}_i^{(k)}-{\Psi}_i^{(k-1)}\|_{C^{0}(D)}\leq M_1\sigma \eta^{k},
\end{equation}
\begin{equation}\label{3-16-2}
\mathop{\max}\limits_{i=1,2,3,4}\varpi\bigg(\delta|\partial_{t}\Psi_{i}^{(k)}
(\cdot,r)\bigg)
\leq\frac{\mn(\delta)}{8[\mathcal{K}+1]},\ \ \forall r\in[r_0,r_1],
\end{equation}
\begin{equation}\label{3-17-2}
\mathop{\max}\limits_{i=1,2,3,4}\bigg\{\varpi\bigg(\delta|\partial_{t}\Psi_{i}^{(k)}\bigg)+
\varpi\bigg(\delta|\partial_{r}\Psi_{i}^{(k)}\bigg)\bigg\}
\leq\bigg(\frac{1}{3}+\frac{1}{2}\mk\bigg)\mn(\delta).
\end{equation}
\par Next, we separately verify the above estimates.
\par{ \bf  Step 1. Proof of \eqref{3-10-2}.} Define the characteristic curve $t=t_{i}^{(k)}(s;t,r)(i=1,2,3,4)$ by
\begin{align}
\left\{
\begin{aligned}
&\frac{\de t_{i}^{(k)}}{\de s}(s;t,r)=\mu_{i}
(\bm\Psi^{(k-1)}+\bar{\bf R})(t_{i}^{(k)}(s;t,r),s),\\
&t_{i}^{(k)}(r;t,r)=t.
\end{aligned}\right.\label{c24}
\end{align}
 Denote
 \begin{equation*}
 \begin{aligned}
 F_1(r)=e^{-\int_{r}^{r_1}\frac{\ma_{11}(s)}{s}\de s},\ \
 F_2(r)=e^{\int_{r_0}^{r}\frac{\ma_{41}(s)}{s}\de s},\\
 \end{aligned}
 \end{equation*}
 from which one has
 \begin{gather*}
F_1'(r)=F_1(r)\frac{ (\gamma+1)\bar U_2^2 +(\gamma-1)\bar U_1^2+2\bar c^2}{ 4r(\bar U_1^2-\bar c^2) }<0,\
 F_2'(r)=F_2(r)\frac{ (\gamma+1)\bar U_2^2 +(\gamma-1)\bar U_1^2+2\bar c^2}{ 4r(\bar U_1+\bar c)^2 }>0.
 \end{gather*}
 Then it is easy to obtain that
 \begin{equation*}
1\leq F_1(r)\leq F_1(r_0), \ \  1\leq F_2(r)\leq F_1(r_0).\end{equation*} Here we used  $(\bar U_1+\bar c)^2>(\bar c^2-\bar U_1^2).$
 Furthermore, a direct computation yields that
 \begin{equation}\label{3-18-b}
 \begin{cases}
 \begin{aligned}
 &(\gamma+1)\bar U_2^2 +(\gamma-1)\bar U_1^2+2\bar c^2=2(\gamma-1)B_0+2\bar U_2^2 \leq 2(\gamma-1)B_0+2U_{2,0}^2,\\
&(\bar c^2-\bar U_1^2)'=\gamma(\gamma-1)e^{S_0}\bar \rho^{\gamma-2}\frac{\bar U_1^2+\bar U_2^2}{r(\bar c^2-\bar U^2)}\bar\rho +\frac{\bar c^2+ \bar U_{2}^2}{r(\bar c^2-\bar U_1^2)}U_{1}^2>0.
\end{aligned}
 \end{cases}
 \end{equation}
 Then one has
 \begin{equation}\label{3-18-b-de}
 F_1(r_0)\leq e^{\frac{(\gamma-1)B_0+U_{2,0}^2}{2( c_0^2- U_{1,0}^2)}\ln \frac{r_1}{r_0}}= \bigg(\frac{r_1}{r_0}\bigg)^{\frac{(\gamma-1)B_0+U_{2,0}^2}{2( c_0^2- U_{1,0}^2)}}:=\me.
  \end{equation}
    Thus the problem \eqref{3-6}-\eqref{3-5} can be rewritten  as follows:
\begin{equation}\label{3-6-1}
\begin{cases}
\begin{aligned}
	&\partial_r \bigg(F_2\Psi_4^{(k)}\bigg)+\mu_{4}( \bm\Psi^{(k-1)}+\bar{\bf R})\partial_t \bigg(F_2\Psi_4^{(k)}\bigg) \\ &=\frac{F_2}r\bigg(
(\ma_{42}-\ma_{41})(r)\Psi_1^{(k-1)} +\ma_{43}(r) \Psi_2^{(k-1)}-\ma_{44}(r)\Psi_3^{(k-1)}\bigg)+F_2\bigg(\mu_{4}( \bm\Psi^{(k-1)}+\bar{\bf R})\\
&\quad  -\mu_4(\bar{\bf R})\bigg)\bigg(\mm_{41}(r)\Psi_1^{(k-1)}-\mm_{44}(r)\Psi_4^{(k-1)}+\mm_{42}(r)
 \Psi_2^{(k-1)}-\mm_{43}(r)\Psi_3^{(k-1)}\bigg)\\
&\ +F_2\mu_4( \bm\Psi^{(k-1)}+\bar{\bf R})\bigg(\frac{(\Psi_2^{(k-1)})^2}{r} - \frac{\gamma-1}{8r}\bigg((\Psi_4^{(k-1)})^2 - (\Psi_1^{(k-1)})^2\bigg)-\frac{\bar c  }{4\gamma r }\Psi_3^{(k-1)}(\Psi_1^{(k-1)}\\
&\quad + \Psi_4^{(k-1)})\bigg) +F_2\bigg(\frac{{\Psi}_4^{(k-1)} - {\Psi}_1^{(k-1)}}{4\gamma}+ \frac{\bar c\Psi_3^{(k-1)}}{2\gamma}\bigg)\bigg(\partial_r \Psi_3^{(k-1)} + \mu_4( \bm\Psi^{(k-1)}+\bar{\bf R}) \partial_t \Psi_3^{(k-1)}\bigg),\\
&r=r_0: \Psi_{4}^{(k)}(t,r_0)= K_4 {\Psi}_1^{(k-1)}(t,r_0) + \mh_4 (t),
\end{aligned}
\end{cases}
\end{equation}
and
\begin{equation}\label{3-5-1}
\begin{cases}
\begin{aligned}
	&\partial_r \bigg(F_1\Psi_1^{(k)}\bigg)+\mu_{1}( \bm\Psi^{(k-1)}+\bar{\bf R})\partial_t \bigg(F_1\Psi_1^{(k)}\bigg)\\
&=\frac{F_1}r\bigg(-\ma_{12}(r)\Psi_1^{(k-1)}+(\ma_{13}-\ma_{11}-\ma_{12})(r)\Psi_4^{(k)}
+\ma_{14}(r)\Psi_2^{(k-1)}
+\ma_{15}(r)\Psi_3^{(k-1)}\bigg) \\
&\  +F_1\bigg(\mu_1(\bm\Psi^{(k-1)}+\bar{\bf R})-\mu_1(\bar{\bf R})\bigg)\bigg(-\mm_{11}(r)\Psi_1^{(k-1)}+\mm_{12}(r)\Psi_2^{(k-1)}+ \mm_{13}(r)\Psi_3^{(k-1)}
\\
 &\quad +\mm_{14}(r)\Psi_4^{(k)}\bigg) +F_1\mu_1( \bm\Psi^{(k-1)}+\bar{\bf R})\bigg(\frac{(\Psi_2^{(k-1)})^2}{r}+\frac{\gamma-1}{8r}\bigg((\Psi_4^{(k)})^2  - (\Psi_1^{(k-1)})^2\bigg)\\
&\quad +\frac{\bar   c}{4\gamma r }\Psi_3^{(k-1)}(\Psi_1^{(k-1)}+ \Psi_4^{(k)})\bigg) -F_1\bigg(\frac{{\Psi}_4^{(k)} - {\Psi}_1^{(k-1)}}{4\gamma} + \frac{\bar c\Psi_3^{(k-1)}}{2\gamma}\bigg)\bigg(  \partial_r \Psi_3^{(k-1)}
\\
&\quad +\mu_1( \bm\Psi^{(k-1)}+\bar{\bf R})\partial_t \Psi_3^{(k-1)}\bigg),\\
&r=r_1: \Psi_{1}^{(k)}(t,r_1)=\mh_{1}(t)-\Psi_{4}^{(k)}(t,r_1).
\end{aligned}
\end{cases}
\end{equation}
By  \eqref{1-16} and \eqref{3-10-1},  we obtain that if $F_{2}(r)\Psi_{4}^{(k)}(t,r)$, $F_{1}(r)\Psi_{1}^{(k)}(t,r)$ and $ \Psi_j^{(k)}(t,r)$ $(j=2,3)$ solve problems \eqref{3-6-1}-\eqref{3-5-1} and \eqref{3-7}-\eqref{3-8}, then $F_{2}(r)\Psi_{4}^{(k)}(t+T,r)$, $F_{1}(r)\Psi_{1}^{(k)}(t+T,r)$ and $ \Psi_j^{(k)}(t+T,r)$ $(j=2,3)$ also solves it. By the uniqueness of solutions to the  linear system, one has $F_{2}(r)\Psi_{4}^{(k)}(t+T,r)=F_{2}(r)\Psi_{4}^{(k)}(t,r)$, $F_{1}(r)\Psi_{1}^{(k)}(t+T,r)=F_{1}(r)\Psi_{1}^{(k)}(t,r)$ and $ \Psi_j^{(k)}(t+T,r)=\Psi_j^{(k)}(t,r)$ for $j=2,3$. Thus \eqref{3-10-2} is proved.
\par { \bf Step 2. Proof of \eqref{3-11-2}.} Integrating \eqref{3-6-1} along the fourth characteristic curve $t_{4}^{(k)}(s;t,r)$ from $r_0$ to $ r$ leads to
\begin{equation}\label{3-6-2}
\begin{aligned}
F_{2}(r)\Psi_{4}^{(k)}(t,r)-F_{2}(r_0)\Psi_{4}^{(k)}
(t_4^{(k)}(r_0;t,r),r_0)=\sum_{i=1}^3 I_{2i},\\
\end{aligned}
\end{equation}
where
\begin{equation*}
\begin{aligned}
F_2(r_0)&=1, \ \ \Psi_{4}^{(k)}(t_4^{(k)}(r_0;t,r),r_1)= K_4 {\Psi}_1^{(k-1)}(t_4^{(k)}(r_0;t,r),r_0) +\mh_4 (t_4^{(k)}(r_0;t,r)),\\
I_{21}&=\int_{r_0}^{r}-F_2'(r)\Psi_1^{(k-1)}(t_4^{(k)}(s;t,r),s)\de s,\\
I_{22}&=\int_{r_0}^r\frac{F_2(s)}s\bigg(
\ma_{42}(s)\Psi_1^{(k-1)} +\ma_{43}(s) \Psi_2^{(k-1)}-\ma_{44}(s)\Psi_3^{(k-1)}\bigg)(t_4^{(k)}(s;t,r),s)\de s,\\
 I_{23}&=\int_{r_0}^r\bigg(F_2(s)\bigg(\mm_{41}(s)\Psi_1^{(k-1)}
 -\mm_{44}(s)\Psi_4^{(k-1)}
 +\mm_{42}(r)
 \Psi_2^{(k-1)}-\mm_{43}(s)\Psi_3^{(k-1)}\bigg)\\
&\qquad+F_2(s)\mu_4( \bm\Psi^{(k-1)}+\bar{\bf R})\bigg(\frac{(\Psi_2^{(k-1)})^2}{s} - \frac{\gamma-1}{8s}\bigg((\Psi_4^{(k-1)})^2- (\Psi_1^{(k-1)})^2\bigg) \\
&\qquad -\frac{(\gamma-1)\mathfrak {c} }{4s}\Psi_3^{(k-1)}(\Psi_1^{(k-1)}+ \Psi_4^{(k-1)})\bigg)\bigg)
 +F_2(s)\bigg(\frac{{\Psi}_4^{(k-1)} - {\Psi}_1^{(k-1)}}{4\gamma}\frac{\mathfrak{c}\Psi_3^{(k-1)}}{2\gamma}\bigg)\bigg(\partial_r \Psi_3^{(k-1)} \\
&\qquad  + \mu_4( \bm\Psi^{(k-1)}+\bar{\bf R}) \partial_t \Psi_3^{(k-1)}\bigg)\bigg)(t_{4}^{(k)}(s;t,r),s)\de s.
\end{aligned}
\end{equation*}
Furthermore,  we integrate \eqref{3-5-1} along the first characteristic curve $t_{1}^{(k)}(s;t,r)$ from $r_1$ to $r$ to get
\begin{equation}\label{3-5-2}
\begin{aligned}
F_{1}(r)\Psi_{1}^{(k)}(t,r)-F_{1}(r_1)\Psi_{1}^{(k)}(t_{1}^{(k)}(r_1;t,r),r_1)=\sum_{i=1}^3 I_{1i},\\
\end{aligned}
\end{equation}
where
\begin{equation*}
\begin{aligned}
&F_1(r_1)=1, \ \ \Psi_{1}^{(k)}(t_{1}^{(k)}(r_1;t,r),r_1)=\mh_{1}(t_{1}^{(k)}(r_1;t,r))
-\Psi_{4}^{(k)}(t_{1}^{(k)}(r_1;t,r),r_1),\\
&I_{11}=\int_{r_1}^{r}-F_1'(r)\Psi_4^{(k)}(t_{1}^{(k)}(s;t,r),s)\de s,\\
&I_{12}=\int_{r_1}^{r}\frac{F_1(s)}s\bigg(
-\ma_{12}(s)\Psi_1^{(k-1)}+(\ma_{13}-\ma_{12})(s)\Psi_4^{(k)}
 +\ma_{14}(s)\Psi_2^{(k-1)}
+\ma_{15}(s)\Psi_3^{(k-1)}\bigg)(t_{1}^{(k)}(s;t,r),s)\de s,\\
 &I_{13}=\int_{r_1}^{r}\bigg( F_1(s)\bigg(\mu_1(\bm\Psi^{(k-1)}+\bar{\bf R})-\mu_1(\bar{\bf R})\bigg)\bigg(-\mm_{11}(s)\Psi_1^{(k-1)}+\mm_{12}(s)\Psi_2^{(k-1)} + \mm_{13}(s)\Psi_3^{(k-1)}+\mm_{14}(s)\Psi_4^{(k)}\bigg)
\\
 &\qquad\ \   +F_1(s)\mu_1( \bm\Psi^{(k-1)}+\bar{\bf R})\bigg(\frac{(\Psi_2^{(k-1)})^2}{s}+\frac{\gamma-1}{8s}\bigg((\Psi_4^{(k)})^2  (\Psi_1^{(k-1)})^2\bigg)+\frac{\bar   c}{4\gamma s}\Psi_3^{(k-1)}(\Psi_1^{(k-1)}+ \Psi_4^{(k)})\bigg)\\
&\qquad \ \ -F_1(s)\bigg(\frac{{\Psi}_4^{(k)} - {\Psi}_1^{(k-1)}}{4\gamma} + \frac{\bar c\Psi_3^{(k-1)}}{2\gamma}\bigg)\bigg(  \partial_r \Psi_3^{(k-1)}+\mu_1( \bm\Psi^{(k-1)}+\bar{\bf R})\partial_t \Psi_3^{(k-1)}\bigg)\bigg)(t_{1}^{(k)}(s;t,r),s)\de s.
\end{aligned}
\end{equation*}
For the first and third terms, it is easy to obtain that
\begin{equation}\label{3-5-2-es1}
\begin{cases}
\begin{aligned}
&|I_{21}|\leq M_1\sigma(F_2(r)-1), \  \ |I_{23}|\leq C_1 (M_1\sigma)^2,\\
&|I_{11}|\leq \|\Psi_{4}^{(k)}\|_{C^{0}(D)}(F_1(r)-1), \  \ |I_{13}|\leq C_1 \bigg((M_1\sigma)^2+\|\Psi_{4}^{(k)}\|_{C^{0}(D)}^2\bigg).
\end{aligned}\end{cases}
\end{equation}
 Here  $ C_1$ is a positive constant depending only on $ (\gamma,r_0,r_1,\rho_0,U_{1,0},U_{2,0},S_0)$.
For the second term,  a simple computation shows that
 \begin{equation*}
\begin{aligned}
&(\bar c-\bar U_1)'=\frac{\gamma-1}2(\gamma e^{S_0})^{\frac12}\bar \rho^{\frac{\gamma-3}2}\frac{\bar U_1^2+\bar U_2^2}{r(\bar c^2-\bar U^2)}\bar\rho +\frac{\bar c^2+ \bar U_{2}^2}{r(\bar c^2-\bar U_1^2)}U_{1}>0 \ {\rm{and}} \ \bar c-\bar U_1\geq c_0- U_{1,0}.
\end{aligned}
\end{equation*}
This, together with \eqref{3-18-b}, yields that
\begin{equation}\label{3-5-2-es2-es}
\begin{cases}
\begin{aligned}
&\ma_{12}(r)=\frac{ (\gamma+1)\bar U_2^2 +(\gamma-1)\bar U_1^2+2\bar c^2}{ 2(\bar U_1-\bar c)^2(\frac{\bar U_1}{\bar c}+1)}\leq \frac{(\gamma-1)B_0+U_{2,0}^2}{( c_0- U_{1,0})^2},\\
&\ma_{13}(r)=\frac{ (\gamma-1)(\bar U_1^2+\bar U_2^2)}{ 2(\bar U_1-\bar c)^2 }\leq \frac{2(\gamma-1)B_0}{2( c_0- U_{1,0})^2}=\frac{(\gamma-1)B_0}{( c_0- U_{1,0})^2},  \\ &\ma_{14}(r)=\frac{2\bar U_2}{\bar c-\bar U_1}\leq  \frac{2U_{2,0}}{ c_0- U_{1,0}},\\
& \ma_{15}(r)=\frac{\bar c^2(\bar U_1^2+\bar U_2^2)}{\gamma (\bar c+\bar U_1)(\bar c-\bar U_1)^2}\leq  \frac{2(\gamma-1)B_0^2}{\gamma (\bar c-\bar U_1)^3}\leq  \frac{2(\gamma-1)B_0^2}{\gamma ( c_0- U_{1,0})^3},
\end{aligned}
\end{cases}
\end{equation}
and
\begin{equation}\label{es1}
\begin{cases}
\begin{aligned}
&\ma_{42}(r)=\frac{ (\gamma-1)(\bar U_1^2+\bar U_2^2)}{ 2(\bar U_1+\bar c)^2 }< \frac{ (\gamma-1)(\bar U_1^2+\bar U_2^2)}{ 2(\bar c-\bar U_1)^2 }\leq\frac{(\gamma-1)B_0}{( c_0- U_{1,0})^2},  \\ &\ma_{43}(r)=\frac{2\bar U_2}{\bar c+\bar U_1}< \frac{2\bar U_2}{\bar c-\bar U_1} \leq \frac{2U_{2,0}}{ c_0- U_{1,0}},\\
&\ma_{44}(r)=\frac{\bar c^2(\bar U_1^2+\bar U_2^2)}{\gamma (\bar c+\bar U_1)^2(\bar c-\bar U_1)} \leq \frac{2(\gamma-1)B_0^2}{\gamma ( c_0- U_{1,0})^3}.
\end{aligned}
\end{cases}
\end{equation}
Combining the above estimates obtains
\begin{equation}\label{3-5-2-es2}
|I_{22}|\leq \me \ln \frac{r_1}{r_0}M_1A_0\sigma,\ \
|I_{12}|\leq \me \ln \frac{r_1}{r_0}A_0\bigg(M_1\sigma +\|\Phi_{4}^{(k)}\|_{C^{0}(D)}\bigg),
\end{equation}
where
$$A_0=\frac{(\gamma-1)B_0}{( c_0- U_{1,0})^2}+\frac{(\gamma-1)B_0+U_{2,0}^2}{( c_0- U_{1,0})^2}+\frac{2(\gamma-1)B_0^2}{\gamma ( c_0- U_{1,0})^3}.
$$
Therefore,  it follows from \eqref{1-16},  and \eqref{3-5-2-es1}-\eqref{3-5-2-es2} that one derives
\begin{equation}\label{3-6-2-es2-1}
\begin{aligned}
\|\Psi_{4}^{(k)}\|_{C^{0}(D)}\leq \frac{\sigma+| K_4|M_1\sigma}{F_2(r)}
+\frac{M_1\sigma(F_2(r)-1)+C_1 (M_1\sigma)^2}{F_2(r)}+\frac{\me M_1 A_0\sigma }{F_2(r)}\ln \frac{r_1}{r_0},
\end{aligned}
\end{equation} and
 \begin{equation}\label{3-5-2-es2-1}
\begin{aligned}
\|\Psi_{1}^{(k)}\|_{C^{0}(D)}&\leq \frac{\sigma+\|\Psi_{4}^{(k)}\|_{C^{0}(D)}}{F_1(r)}
+\frac{\|\Psi_{4}^{(k)}\|_{C^{0}(D)}(F_1(r)-1)+C_1\bigg( (M_1\sigma)^2+\|\Psi_{4}^{(k)}\|_{C^{0}(D)}^2\bigg)}{F_1(r)}\\
&\quad+\frac{\me A_0}{F_1(r)}\ln \frac{r_1}{r_0} \bigg(M_1\sigma +\|\Psi_{4}^{(k)}\|_{C^{0}(D)}\bigg).
\end{aligned}
\end{equation}
\par Set \begin{equation}\label{3-5-2-de}M_1=\frac{100\me}{1-K}, \ K=\max\Big\{|K_2|,|K_3|,|K_4|\Big\}, \ \ \sigma_1^\ast=\frac1{2(C_1M_1^2+1)}.\end{equation} Then it holds that
\begin{equation}\label{3-5-2-es4}
M_1\geq \max\Big\{|K_2|,|K_3|,|K_4|\Big\}M_1+100\me.
\end{equation}
Furthermore,  for any $\sigma\in(0,\sigma_1^\ast)$, one has
$C_1(M_1\sigma)^2<\sigma.$ Define $$\frac{r_1}{r_0}=1+\xi, \ \ A_1=\frac{(\gamma-1)B_0+U_{2,0}^2}{2( c_0^2- U_{1,0}^2)},$$
and
$$\mf_1(\xi)=\ln \frac{r_1}{r_0} \me M_1 A_0 =\frac{100}{1-K}A_0(1+\xi)^{2A_1}\ln(1+\xi).$$
Here we used \eqref{3-18-b-de}. Then a direct computation yields that
\begin{equation*}
\begin{aligned}
\mf_1'(\xi)&=\frac{100}{1-K}A_0(1+\xi)^{2A_1-1}
(2A_1\ln(1+\xi)+1),\\
\mf_1''(\xi)&=\frac{100}{1-K}A_0(1+\xi)^{2A_1-2}
(2A_1(A_1-1)\ln(1+\xi)+4A_1-1),\\
\mf_1'(0)&=\frac{100}{1-K}A_0
>0,\\
\mf_1''(0)&=\frac{100}{1-K}A_0(4A_1-1)=\frac{100}{1-K} A_0
\frac{2(\gamma-1)B_0+2U_{2,0}^2- c_0^2+ U_{1,0}^2}{ c_0^2- U_{1,0}^2}>0.
\end{aligned}
\end{equation*}
Thus there exists a small constant $\xi_1>0$ such that for any $\xi\in(0,\xi_1)$, one has
$$\mf_1''(\xi)\leq 2\mf_1''(0)$$
and
$$\mf_1(\xi)=\mf_1(0)+\mf_1'(0)\xi+\frac12\mf_1''(s)\xi^2\leq \mf_1'(0)\xi+\mf_1''(0)\xi^2 \leq (\mf_1'(0)+\xi_1\mf_1''(0))\xi, \ \ s\in(0,\xi).$$
 Set
\begin{equation}\label{3-5-2-es2-6-1}
\xi_0=\min\bigg\{\xi_1,\frac{1}{\mf_1'(0)+\xi_1\mf_1''(0)+1}\bigg\}.\end{equation}
Then for any $ \xi\in(0,\xi_0)$, one has
\begin{equation}\label{3-5-2-es2-6-1-2}
\mf_1(\xi)={\me}M_1 A_0\ln \frac{r_1}{r_0} \leq 1.\end{equation}
If $\sigma\in(0,\sigma_1^\ast)$  ,  it follows from \eqref{3-6-2-es2-1}, \eqref{3-5-2-es4} and \eqref{3-5-2-es2-6-1-2} that \begin{equation}\label{3-6-2-es2-7}
\begin{aligned}
\|\Psi_{4}^{(k)}\|_{C^{0}(D)}&\leq \frac{\sigma+|K_4|M_1\sigma}{F_2(r)}
+\frac{M_1\sigma(F_2(r)-1)+C_1 (M_1\sigma)^2}{F_2(r)}+\frac{\me M_1 A_0\sigma }{F_2(r)}\ln \frac{r_1}{r_0}\\
&\leq M_1\sigma-\frac{99\me\sigma}{F_2(r)}+\frac{\sigma}{F_2(r)}+\frac{\me M_1A_0\sigma }{F_2(r)}\ln \frac{r_1}{r_0} \\
&
 \leq M_1\sigma-\frac{99\me\sigma}{F_2(r)}+\frac{\sigma}{F_2(r)}+\frac{\mf_1(\xi)\sigma}{F_2(r)} \leq M_1\sigma-\frac{99\me\sigma}{F_2(r)}+\frac{2\sigma}{F_2(r)}
\leq\chi_{1}M_{1}\sigma,\end{aligned}
\end{equation}
where  $\chi_1=1-\frac{97}{M_1}$ is a constant satisfying $0<\chi_1<1$. This, together with \eqref{3-5-2-es2-1}, yields that
\begin{equation}\label{3-5-2-es2-1-1}
\begin{aligned}
\|\Psi_{1}^{(k)}\|_{C^{0}(D)}&\leq \frac{\sigma+\chi_{1}M_{1}\sigma}{F_1(r)}
+\frac{\chi_{1}M_{1}\sigma(F_1(r)-1)+C_1\bigg( (M_1\sigma)^2+(\chi_{1}M_{1}\sigma)^2\bigg)}{F_1(r)}\\
&\quad+\frac{\me A_0}{F_1(r)}\ln \frac{r_1}{r_0} 2M_1\sigma
\leq \chi_{1}M_{1}\sigma+\frac{3\sigma}{F_1(r)}+\frac{2\me M_1 A_0\sigma }{F_1(r)}\ln \frac{r_1}{r_0} \\ &=\chi_{1}M_{1}\sigma+\frac{3\sigma}{F_1(r)}+\frac{2\mf_1(\xi)\sigma}{F_1(r)}
\leq \chi_{1}M_{1}\sigma+\frac{5\sigma}{F_1(r)}= \chi_{2}M_{1}\sigma,
\end{aligned}
\end{equation}
where  $\chi_2=1-\frac{92}{M_1}$ is a constant satisfying $0<\chi_1<\chi_2<1$. In the following, by integrating \eqref{3-7} and \eqref{3-8} along the characteristic curves $t_{2}^{(k)}(s;t,r)$ and $t_{3}^{(k)}(s;t,r)$ from $r_0$ to $r$, respectively, one derives
\begin{equation}\label{3-7-1-es1}
\begin{cases}
\begin{aligned}
\|\Psi_{2}^{(k)}\|_{C^{0}(D)}&=\bigg\|\frac{r_0}r\bigg( K_2{\Psi}_1^{(k-1)}(\cdot,r_0) + \mh_2(\cdot)\bigg)\bigg\|_{C^{0}(\mathbb{R})}\leq \sigma+|K_2|M_{1}\sigma  < M_{1}\sigma,\\
\|\Psi_{3}^{(k)}\|_{C^{0}(D)}&=\bigg\| K_3{\Psi}_1^{(k-1)}(\cdot,r_0) + \mh_3(\cdot)
\bigg\|_{C^{0}(\mathbb{R})} \leq \sigma+|K_3|M_{1}\sigma <M_{1}\sigma.\\
\end{aligned}
\end{cases}
\end{equation}
\par Next, we derive the $C^1$ estimates for $\Psi_{i}^{(k)}$ $(i=1,2,3,4)$. Set
$\Upsilon_{i}^{(k)}=\p_t\Psi_{i}^{(k)}$.
Differentiating  \eqref{3-6-1}-\eqref{3-5-1}  and \eqref{3-7}-\eqref{3-8}  with respect to $ t $, one  gets
\begin{equation*}\begin{aligned}
&\partial_r \bigg(F_2\Upsilon_4^{(k)}\bigg)+\mu_{4}( \bm\Psi^{(k-1)}+\bar{\bf R})\partial_t \bigg(F_2\Upsilon_4^{(k)}\bigg)\\
&=\frac{F_2}r\bigg(
(\ma_{42}-\ma_{41})(r)\Upsilon_1^{(k-1)} +\ma_{43}(r) \Upsilon_2^{(k-1)}-\ma_{44}(r)\Upsilon_3^{(k-1)}\bigg)\\
&\ +F_2\bigg(\mu_{4}( \bm\Psi^{(k-1)}+\bar{\bf R}) -\mu_4(\bar{\bf R})\bigg)\bigg(\mm_{41}(r)\Upsilon_1^{(k-1)}-\mm_{44}(r)\Upsilon_4^{(k-1)}\\
&\quad +\mm_{42}(r)
 \Upsilon_2^{(k-1)}-\mm_{43}(r)\Upsilon_3^{(k-1)}\bigg)+F_2\mu_4( \bm\Psi^{(k-1)}+\bar{\bf R})\p_t\bigg(\frac{(\Psi_2^{(k-1)})^2}{r} \\
&\quad - \frac{\gamma-1}{8r}\bigg((\Psi_4^{(k-1)})^2 - (\Psi_1^{(k-1)})^2\bigg)-\frac{\bar c  }{4\gamma r }\Psi_3^{(k-1)}(\Psi_1^{(k-1)}+ \Psi_4^{(k-1)})\bigg) \\ \end{aligned}
\end{equation*}\begin{equation}\label{3-6-t}\begin{aligned}
&\ +F_2\sum_{i=1}^4\frac{\p\mu_{4}(\bm\Psi^{(k-1)}+\bar{\bf R})}{\p \Psi_i^{(k-1)}}\Upsilon_i^{(k-1)}\bigg(-\Upsilon_4^{(k)}+\mm_{41}(r)\Psi_1^{(k-1)}
-\mm_{44}(r)\Psi_4^{(k-1)}\\
&\quad  +\mm_{42}(r)
 \Psi_2^{(k-1)}-\mm_{43}(r)\Psi_3^{(k-1)}+\frac{(\Psi_2^{(k-1)})^2}{r}- \frac{\gamma-1}{8r}\bigg((\Psi_4^{(k-1)})^2- (\Psi_1^{(k-1)})^2\bigg)\\
&\quad-\frac{\bar c  }{4\gamma r }\Psi_3^{(k-1)}(\Psi_1^{(k-1)}+ \Psi_4^{(k-1)})+\bigg(\frac{{\Psi}_4^{(k-1)} - {\Psi}_1^{(k-1)}}{4\gamma} + \frac{\bar c\Psi_3^{(k-1)}}{2\gamma}\bigg)\Upsilon_3^{(k-1)}\bigg)\\
&\ \ +F_2\bigg(\frac{{\Upsilon}_4^{(k-1)} - {\Upsilon}_1^{(k-1)}}{4\gamma}+ \frac{\bar c\Upsilon^{(k-1)}}{2\gamma}\bigg)\bigg(\partial_r \Psi_3^{(k-1)} + \mu_4( \bm\Psi^{(k-1)}+\bar{\bf R}) \partial_t \Psi_3^{(k-1)}\bigg)\\
&\ \  +F_2\bigg(\frac{{\Psi}_4^{(k-1)} - {\Upsilon}_1^{(k-1)}}{4\gamma}+ \frac{\bar c\Psi_3^{(k-1)}}{2\gamma}\bigg)\bigg(\partial_r \Upsilon_3^{(k-1)} + \mu_4( \bm\Psi^{(k-1)}+\bar{\bf R}) \partial_t \Upsilon_3^{(k-1)}\bigg),\\
\end{aligned}
\end{equation}and
\begin{equation}\label{3-5-t}\begin{aligned}
&\partial_r \bigg(F_1\Upsilon_1^{(k)}\bigg)+\mu_{1}( \bm\Psi^{(k-1)}+\bar{\bf R})\partial_t \bigg(F_1\Upsilon_1^{(k)}\bigg)\\
&=\frac{F_1}r
\bigg(-\ma_{12}(r)\Upsilon_1^{(k-1)}+(\ma_{13}-\ma_{11}-\ma_{12})(r)\Upsilon_4^{(k)}  +\ma_{14}(r)\Upsilon_2^{(k-1)}\\
&\quad
+\ma_{15}(r)\Upsilon_3^{(k-1)}\bigg)+F_1\bigg(\mu_1(\bm\Psi^{(k-1)}+\bar{\bf R})-\mu_1(\bar{\bf R})\bigg)\bigg(-\mm_{11}(r)\Upsilon_1^{(k-1)}\\
  &\quad+\mm_{12}(r)\Upsilon_2^{(k-1)}+ \mm_{13}(r)\Upsilon_3^{(k-1)}+\mm_{14}(r)\Upsilon_4^{(k)}\bigg)
  +F_1\mu_1( \bm\Psi^{(k-1)}+\bar{\bf R})\p_t\bigg(\frac{(\Psi_2^{(k-1)})^2}{r}\\
&\quad+\frac{\gamma-1}{8r}\bigg((\Psi_4^{(k)})^2  (\Psi_1^{(k-1)})^2\bigg)  +\frac{\bar   c}{4\gamma r }\Psi_3^{(k-1)}(\Psi_1^{(k-1)}+ \Psi_4^{(k)})\bigg)\\
&\ +F_1\sum_{i=1}^4\frac{\p\mu_{1}(\bm\Psi^{(k-1)}+\bar{\bf R})}{\p \Psi_i^{(k-1)}}\Upsilon_i^{(k-1)}\bigg(-\Upsilon_1^{(k)}
-\mm_{11}(r)\Psi_1^{(k-1)}+\mm_{12}(r)\Psi_2^{(k-1)}
  \\
&\quad + \mm_{13}(r)\Psi_3^{(k-1)}+\mm_{14}(r)\Psi_4^{(k)}+\frac{(\Psi_2^{(k-1)})^2}{r}  + \frac{\gamma-1}{8r}\bigg((\Psi_4^{(k-1)})^2+ (\Psi_1^{(k-1)})^2\bigg)\\
&\quad +\frac{\bar c  }{4\gamma r }\Psi_3^{(k-1)}(\Psi_1^{(k-1)}- \Psi_4^{(k-1)})+\bigg(\frac{{\Psi}_4^{(k-1)} + {\Psi}_1^{(k-1)}}{4\gamma}- \frac{\bar c\Psi_3^{(k-1)}}{2\gamma}\bigg)\Upsilon_3^{(k-1)}\bigg)\\
&\ -F_1\bigg(\frac{{\Upsilon}_4^{(k-1)} - {\Upsilon}_1^{(k-1)}}{4\gamma}+ \frac{\bar c\Upsilon^{(k-1)}}{2\gamma}\bigg)\bigg(\partial_r \Psi_3^{(k-1)} + \mu_1( \bm\Psi^{(k-1)}+\bar{\bf R}) \partial_t \Psi_3^{(k-1)}\bigg)\\
&\ -F_1\bigg(\frac{{\Psi}_4^{(k-1)} - {\Upsilon}_1^{(k-1)}}{4\gamma}+ \frac{\bar c\Psi_3^{(k-1)}}{2\gamma}\bigg)\bigg(\partial_r \Upsilon_3^{(k-1)} + \mu_1( \bm\Psi^{(k-1)}+\bar{\bf R}) \partial_t \Upsilon_3^{(k-1)}\bigg),\\
\end{aligned}
\end{equation} and
\begin{equation}\label{3-7-t}
\begin{aligned}
 \partial_r (r\Upsilon_2^{(k)})+  \mu_2(\bm\Psi^{(k-1)}+\bar{\bf R})}{\p \Psi_1^{(k-1)})\partial_t (r\Upsilon_2^{(k)})
=-r\bigg(\sum_{i=1}^4\frac{\p\mu_{2}(\bm\Psi^{(k-1)}+\bar{\bf R})}{\p \Psi_i^{(k-1)}}\Upsilon_i^{(k-1)}\bigg)\Upsilon_2^{(k)},\\
\end{aligned}
\end{equation}
and
\begin{equation}\label{3-8-t}
\begin{aligned}
 \partial_r \Upsilon_3^{(k)}+  \mu_3(\bm\Psi^{(k-1)}+\bar{\bf R})}{\p \Psi_1^{(k-1)})\partial_t \Upsilon_3^{(k)}
=-\bigg(\sum_{i=1}^4\frac{\p\mu_{3}(\bm\Psi^{(k-1)}+\bar{\bf R})}{\p \Psi_i^{(k-1)}}\Upsilon_i^{(k-1)}\bigg)\Upsilon_3^{(k)}.\\
\end{aligned}
\end{equation}
Furthermore, it follows from  \eqref{1-14} that
\begin{equation}\label{1-14-mod-11}
\begin{cases}
r=r_0:
\begin{cases}
\Upsilon_{4}^{(k)}(t,r_0) =  K_4\Upsilon_1^{(k-1)}(t,r_0)  +\mh_4' (t),\\
\Upsilon_{2}^{(k)}(t,r_0) =  K_2\Upsilon_1^{(k-1)}(t,r_0) +
\mh_2'(t), \\
\Upsilon_{3}^{(k)}(t,r_0) =  K_3\Upsilon_1^{(k-1)}(t,r_0)+ \mh_3'(t), \\
\end{cases} \\
r=r_1: \Upsilon_{1}^{(k)}(t,r_1) = \mh_{1}'(t) - \Upsilon_{4}^{(k)}(t,r_1).
\end{cases}
\end{equation}
\par Note that \eqref{3-6-t} has the term related to $\Upsilon_4^{(k)}$
and this term cannot be estimated by
 a priori estimate.  To this end, one can rewrite   \eqref{3-6-t} as
 \begin{equation}\label{3-5-t-1}\begin{aligned}
&\partial_r \Upsilon_4^{(k)}+\mu_{4}(\bm\Psi^{(k-1)}+\bar{\bf R})\partial_t \Upsilon_4^{(k)}\\
&=-\bigg(\frac{\ma_{41}(r)}r+\sum_{i=1}^4\frac{\p\mu_{4}(\bm\Psi^{(k-1)}+\bar{\bf R})}{\p \Psi_i^{(k-1)}}\Upsilon_i^{(k-1)}\bigg)\Upsilon_4^{(k)}
 +O(\sigma)\\
&\quad+\bigg(\frac{{\Psi}_4^{(k-1)} - {\Psi}_1^{(k-1)}}{4\gamma}
 + \frac{\bar{c}\Psi_3^{(k-1)}}{2\gamma}\bigg)\bigg(\partial_r \Upsilon_3^{(k-1)}+ \mu_4( \bm\Psi^{(k-1)}+\bar{\bf R}) \partial_t \Upsilon_3^{(k-1)}\bigg).\\
\end{aligned}
\end{equation}
Then
 we integrate \eqref{3-5-t-1} along the fourth characteristic curve $t_{4}^{(k)}(s;t,r)$ from $r_0$ to $r$.  It can be checked that the last term  in \eqref{3-5-t-1} which needs a
further analysis:
\begin{equation}\label{3-5-t-1-1}
\begin{aligned}
&\int_{r_0}^r\bigg(\frac{{\Psi}_4^{(k-1)} - {\Psi}_1^{(k-1)}}{4\gamma}
 + \frac{\bar{c}\Psi_3^{(k-1)}}{2\gamma}\bigg)\bigg(\partial_r \Upsilon_3^{(k-1)}+ \mu_4( \bm\Psi^{(k-1)}+\bar{\bf R}) \partial_t \Upsilon_3^{(k-1)}\bigg)(t_{4}^{(k)}(s;t,r),s)\de s\\
 &=\int_{r_0}^r\bigg(\frac{{\Psi}_4^{(k-1)} - {\Psi}_1^{(k-1)}}{4\gamma}
 + \frac{\bar{c}\Psi_3^{(k-1)}}{2\gamma}\bigg)\frac{\de }{\de s}\Upsilon_3^{(k-1)}(t_{4}^{(k)}(s;t,r),s)\de s\\
 &=\bigg(\frac{{\Psi}_4^{(k-1)} - {\Psi}_1^{(k-1)}}{4\gamma}
 + \frac{\bar{c}\Psi_3^{(k-1)}}{2\gamma}\bigg)\bigg(\Upsilon_3^{(k-1)}(t_{4}^{(k)}(r_0;t,r),r_0)
 -\Upsilon_3^{(k-1)}(t,r)\bigg)\\
&\quad -\int_{r_0}^r
 \Upsilon_3^{(k-1)}(t_{4}^{(k)}(s;t,r),s)\frac{\de}{\de s}\bigg(\frac{{\Psi}_4^{(k-1)} - {\Psi}_1^{(k-1)}}{4\gamma}
 + \frac{\bar{c}\Psi_3^{(k-1)}}{2\gamma}\bigg)(t_{4}^{(k)}(s;t,r),s)\de s\\
 &\leq C_2\sigma.
\end{aligned}
\end{equation}
Therefore, one has
\begin{equation*}
\begin{aligned}
\bigg|\Upsilon_4^{(k)}(t,r)\bigg|&=
\bigg|\Upsilon_4^{(k)}(t_{4}^{(k)}(r_0;t,r),r_0)+O(\sigma)\\
&\qquad-\int_{r_0}^r\bigg(\frac{\ma_{41}(s)}s+\sum_{i=1}^4\frac{\p\mu_{4}(\bm\Psi^{(k-1)}+\bar{\bf R})}{\p \Psi_i^{(k-1)}}\Upsilon_i^{(k-1)}\bigg)\Upsilon_4^{(k)}(t_{4}^{(k)}(s;t,r),s)\de s\bigg|\\
&=\bigg| K_4\Upsilon_1^{(k-1)}(t_{4}^{(k)}(r_0;t,r),r_0) +\mh_4' (t_{4}^{(k)}(r_0;t,r))+O(\sigma)\\
&\quad-\int_{r_0}^r\bigg(\frac{\ma_{41}(s)}s+\sum_{i=1}^4\frac{\p\mu_{4}(\bm\Psi^{(k-1)}+\bar{\bf R})}{\p \Psi_i^{(k-1)}}\Upsilon_i^{(k-1)}\bigg)\Upsilon_4^{(k)}(t_{4}^{(k)}(s;t,r),s)\de s\bigg|\\
&\quad\leq C_3\sigma+\int_{r_0}^rC_3|\Upsilon_4^{(k)}(t_{4}^{(k)}(s;t,r),s)|\de s.
\end{aligned}
\end{equation*}
By the   Gronwall inequality, there holds
\begin{equation}\label{3-5-t-2}
|\Upsilon_4^{(k)}|\leq C_4\sigma, \quad\forall(t,r)\in \mathbb{R}\times[r_0,r_1].
\end{equation}
Here  $ C_i$ $(i=2,3,4)$  are positive constants depending only on $ (\gamma,r_0,r_1,\rho_0,U_{1,0},U_{2,0},S_0)$.
\par Next, integrating \eqref{3-6-t} along $t_{4}^{(k)}(s;t,r)$ and using the estimates \eqref{1-16}, \eqref{3-11-1}, \eqref{3-5-2-es2-es}, \eqref{es1}, \eqref{3-5-2-es4} and \eqref{3-5-t-1-1}-\eqref{3-5-t-2} lead to
 \begin{equation}\label{1-4th}
\begin{aligned}
\|\Upsilon_{4}^{(k)}\|_{C^{0}(D)}&\leq \frac{\sigma+|K_4|M_1\sigma}{F_2(r)}
+\frac{M_1\sigma(F_2(r)-1)+C_5 (M_1\sigma)^2}{F_2(r)}+\frac{\me M_1 A_0\sigma }{F_2(r)}\ln \frac{r_1}{r_0},\\
\end{aligned}
\end{equation}
 where  $C_5$  is a  positive constant depending only on $ (\gamma,r_0,r_1,\rho_0,U_{1,0},U_{2,0},S_0)$. Set \begin{equation}\label{3-7-1-es2-es1}
 \sigma_2^\ast=\min\bigg\{\sigma_1^\ast,\frac1{C_5M_1^2+1}\bigg\},\end{equation}
 where $ \sigma_1^\ast$ is defined in \eqref{3-5-2-de}.
Then for any $\sigma\in(0,\sigma_2^\ast)$  and $ \xi\in(0,\xi_0)$   with   $\xi_0$ defined in  \eqref{3-5-2-es2-6-1}, it follows from \eqref{1-4th} and \eqref{3-5-2-es2-6-1-2} that
\begin{equation}\label{3-5-2-es3-7}
\begin{aligned}
\|\Upsilon_{4}^{(k)}\|_{C^{0}(D)}
&\leq M_1\sigma-\frac{99\me\sigma}{F_2(r)}+\frac{(1+\mf_1(\xi))\sigma}{F_2(r)}=\chi_{1}M_{1}\sigma.\end{aligned}\end{equation}
Similarly,   integrating \eqref{3-5-t} along the first characteristic curve $t_{1}^{(k)}(s;t,r)$ from $r_1$ to $r$ obtains
 \begin{equation}\label{3-5-2-es3-8}
 \begin{aligned}
\|\Upsilon_{1}^{(k)}\|_{C^{0}(D)}&\leq \chi_{1}M_{1}\sigma+\frac{3\sigma}{F_1(r)}+\frac{2\mf_1(\xi)\sigma}{F_1(r)}
\leq \chi_{1}M_{1}\sigma+\frac{5\sigma}{F_1(r)}= \chi_{2}M_{1}\sigma.
\end{aligned}
\end{equation}
These, together with \eqref{3-6}-\eqref{3-5}, \eqref{3-11-1},
\eqref{3-6-2-es2-7}-\eqref{3-5-2-es2-1-1}, yield that
\begin{equation}\label{3-5-2-es3-9}
\begin{cases}
\|\p_r\Psi_{4}^{(k)}\|_{C^{0}(D)}\leq\mk\chi_{1}M_{1}\sigma+C_6(1+\chi_{1}+\chi_{2})M_{1}\sigma
+C_5(M_{1}\sigma)^2< M_2\sigma,\\
\|\p_r\Psi_{1}^{(k)}\|_{C^{0}(D)}\leq\mk\chi_{2}M_{1}\sigma+C_6(1+\chi_{1}+\chi_{2})M_{1}\sigma
+C_5(M_{1}\sigma)^2< M_2\sigma,\\
\end{cases}
\end{equation}
where $M_2>2\mk M_{1}+C_6(3M_{1}+1)$ and $ C_6$ is a positive constant depending only on $ (\gamma,r_0,r_1,\rho_0,U_{1,0},$\\$U_{2,0},S_0)$.
\par In the following, by integrating \eqref{3-7-t} and \eqref{3-8-t} along the characteristic curves $t_{2}^{(k)}(s;t,r)$ and  $t_{3}^{(k)}(s;t,r)$ from $r_0$ to $r$, respectively, one derives
\begin{equation}\label{a1}
\begin{aligned}
  r\Upsilon_2^{(k)}(t,r)&= K_2\Upsilon_1^{(k-1)}(t_{2}^{(k)}(r_0;t,r),r_0) +
\mh_2'(t_{2}^{(k)}(r_0;t,r))\\
&\quad-\int_{r_0}^rs\bigg(\sum_{i=1}^4\frac{\p\mu_{4}(\bm\Psi^{(k-1)}+\bar{\bf R})}{\p \Psi_i^{(k-1)}}\Upsilon_i^{(k-1)}\bigg)\Upsilon_2^{(k)}(t_{2}^{(k)}(s;t,r),s)\de s,
\end{aligned}
\end{equation}
and
\begin{equation}\label{a2}
\begin{aligned}
 \Upsilon_3^{(k)}(t,r)&= K_3\Upsilon_1^{(k-1)}(t_{3}^{(k)}(r_0;t,r),r_0) +\mh'_3(t_{3}^{(k)}(r_0;t,r))\\
&\quad-\int_{r_0}^r\bigg(\sum_{i=1}^4\frac{\p\mu_{3}(\bm\Psi^{(k-1)}+\bar{\bf R})}{\p \Psi_i^{(k-1)}}\Upsilon_i^{(k-1)}\bigg)\Upsilon_3^{(k)}(t_{3}^{(k)}(s;t,r),s)\de s.
\end{aligned}
\end{equation}
With the aid of \eqref{1-15} and \eqref{3-11-1}, one gets
\begin{equation*}
\begin{aligned}
 &-\bigg(\sum_{i=1}^4\frac{\p\mu_{i}(\bm\Psi^{(k-1)}+\bar{\bf R})}{\p \Psi_i^{(k-1)}}\Upsilon_i^{(k-1)}\bigg)
 \frac2{(\Psi_1^{(k-1)}+\bar R_1+\Psi_4^{(k-1)}+\bar R_4)^2}(\Upsilon_1^{(k-1)}
+\Upsilon_4^{(k-1)})\leq 4\mk^2M_1\sigma.
\end{aligned}
\end{equation*}
 By  the Gronwall inequality, there holds
\begin{equation}\label{3-7-t-1-es}
\|\Upsilon_{2}^{(k)}\|_{C^{0}(D)}\leq e^{4\mk^2M_1\sigma(r_1-r_0)}(1+ KM_{1}) \sigma ,\ \
\|\Upsilon_{3}^{(k)}\|_{C^{0}(D)}\leq e^{4\mk^2M_1\sigma(r_1-r_0)}(1+KM_{1})\sigma.\\
\end{equation}Define
$$g_1(\sigma)=e^{4\mk^2M_1\sigma(r_1-r_0)}(1+KM_{1}).$$
    By the continuity of the exponential function $g_1(\sigma)$ and $g_1(0)=(1+KM_{1})<M_1$, there exists a sufficiently small $ \sigma_3^\ast>0 $ such that for any $\sigma\in(0,\sigma_3^\ast)$,
\begin{equation*}
 e^{4\mk^2M_1\sigma(r_1-r_0)}(1+KM_{1})\sigma\leq M_1\sigma.
\end{equation*}
With the above estimates, one can further use \eqref{3-6} and \eqref{3-7} to obtain
\begin{equation}\label{3-7-t-2-es}
\|\p_r\Psi_{2}^{(k)}\|_{C^{0}(D)}\leq \mk M_1\sigma<  M_2\sigma,\ \
\|\p_r\Psi_{3}^{(k)}\|_{C^{0}(D)}\leq  \mk M_1\sigma  <  M_2\sigma.\\
\end{equation}
 Therefore,
set
\begin{equation}\label{3-7-8-es2} \sigma_4^\ast=\min\{\sigma_3^\ast,\sigma_2^\ast\}.
\end{equation}
For any $ \frac{r_1}{r_0}=1+\xi$ with $ \xi\in(0,\xi_0) $ and $ \sigma\in (0,\sigma_4^\ast ), $   the estimates \eqref{3-11-2} and \eqref{3-11-2-r} and  can be easily checked by using \eqref{3-6-2-es2-7}-\eqref{3-7-1-es1}, \eqref{3-5-2-es3-7}-\eqref{3-5-2-es3-9} and \eqref{3-7-t-1-es}-\eqref{3-7-t-2-es}.
\par
 { \bf Step 3. Proof of \eqref{3-15-2}.} Firstly, one gets directly from \eqref{1-16}, \eqref{3-9} and \eqref{3-6-2-es2-7}-\eqref{3-5-2-es2-1-1} that
 \begin{equation}\label{3-15-2-es1}
\begin{cases}
\begin{aligned}
&\|\Psi_{1}^{(1)}-\Psi_{1}^{(0)}\|_{C^{0}(D)}\leq \chi_2M_1\sigma \leq M_1\sigma \eta,\\
&\|\Psi_{4}^{(1)}-\Psi_{4}^{(0)}\|_{C^{0}(D)}\leq \chi_1M_1\sigma \leq M_1\sigma \eta,\\
&\|\Psi_{2}^{(1)}-\Psi_{2}^{(0)}\|_{C^{0}(D)}=\bigg\|\frac{r_0}r \mh_2\bigg\|_{C^{0}(D)}\leq\sigma   \leq M_1\sigma \eta,\\
&\|\Psi_{3}^{(1)}-\Psi_{3}^{(0)}\|_{C^{0}(D)}=\|\mh_3
\|_{C^{0}(D)}\leq \sigma \leq M_1\sigma  \eta,\\
\end{aligned}
\end{cases}
\end{equation}
where
$$\max\{\chi_1,\chi_2,{1- K},K\}< \eta<1.$$
 Next, we prove \eqref{3-15-2} for $ k\geq 2$. By  \eqref{3-6-1}-\eqref{3-5-1} and \eqref{3-7}-\eqref{3-8}, it holds that
\begin{equation*}
\begin{aligned}
	&\partial_r \bigg(F_2(\Psi_4^{(k)}-\Psi_4^{(k-1)})\bigg)+\mu_{4}( \bm\Psi^{(k-1)}+\bar{\bf R})\partial_t \bigg(F_2(\Psi_4^{(k)}-\Psi_4^{(k-1)})\bigg)\\
&=\frac{F_2}r\bigg(
(\ma_{42}-\ma_{41})(r)(\Psi_1^{(k-1)}-\Psi_1^{(k-2)}) +\ma_{43}(r) (\Psi_2^{(k-1)}-\Psi_2^{(k-2)})-\ma_{44}(r)(\Psi_3^{(k-1)}\\
&\quad -\Psi_3^{(k-2)})\bigg)+F_2\bigg(\mu_{4}( \bm\Psi^{(k-1)}+\bar{\bf R}) -\mu_4(\bar{\bf R})\bigg)\bigg(\mm_{41}(r)(\Psi_1^{(k-1)}-\Psi_1^{(k-2)})-\mm_{44}(r)
(\Psi_4^{(k-1)}\\
&\quad-\Psi_4^{(k-2)})+\mm_{42}(r)
 (\Psi_2^{(k-1)}-\Psi_2^{(k-2)})-\mm_{43}(r)(\Psi_3^{(k-1)}-\Psi_3^{(k-2)})\bigg) \\
&\  \ +F_2\mu_4( \bm\Psi^{(k-1)}+\bar{\bf R})\bigg(\bigg(\frac{(\Psi_2^{(k-1)})^2}{r}  - \frac{\gamma-1}{8r}\bigg((\Psi_4^{(k-1)})^2 - (\Psi_1^{(k-1)})^2\bigg)-\frac{\bar c  }{4\gamma r }\Psi_3^{(k-1)}(\Psi_1^{(k-1)}\\
&\quad + \Psi_4^{(k-1)})\bigg)-\bigg(\frac{(\Psi_2^{(k-2)})^2}{r} - \frac{\gamma-1}{8r}\bigg((\Psi_4^{(k-2)})^2 - (\Psi_1^{(k-2)})^2\bigg)-\frac{\bar c  }{4\gamma r }\Psi_3^{(k-2)}(\Psi_1^{(k-2)}+ \Psi_4^{(k-2)})\bigg)\bigg)\\
&\ \ +F_2\bigg(\mu_{4}( \bm\Psi^{(k-1)}+\bar{\bf R})-\mu_{4}(\bm\Psi^{(k-2)}+\bar{\bf R})\bigg)\bigg(-\partial_t \Psi_4^{(k-1)}+\mm_{41}(r)\Psi_1^{(k-2)}\\ &\quad -\mm_{44}(r)\Psi_4^{(k-2)}+\mm_{42}(r)
 \Psi_2^{(k-2)} -\mm_{43}(r)\Psi_3^{(k-2)}+\frac{(\Psi_2^{(k-2)})^2}{r} - \frac{\gamma-1}{8r}\bigg((\Psi_4^{(k-2)})^2 \\
&\quad - (\Psi_1^{(k-2)})^2\bigg)-\frac{\bar c  }{4\gamma r }\Psi_3^{(k-2)}(\Psi_1^{(k-2)}+ \Psi_4^{(k-2)})+\p_t\Psi_3^{(k-2)}\bigg(\frac{{\Psi}_4^{(k-1)} - {\Psi}_1^{(k-1)}}{4\gamma}+ \frac{\bar c\Psi_3^{(k-1)}}{2\gamma}\bigg)\bigg)\\ \end{aligned}
\end{equation*}\begin{equation}\label{3-6-1-d}
\begin{aligned}
&\  +F_2\bigg(\frac{{\Psi}_4^{(k-1)} - {\Psi}_1^{(k-1)}}{4\gamma}+ \frac{\bar c\Psi_3^{(k-1)}}{2\gamma}\bigg)\bigg(\partial_r (\Psi_3^{(k-1)}-\Psi_3^{(k-2)}) + \mu_4( \bm\Psi^{(k-1)}+\bar{\bf R}) \partial_t (\Psi_3^{(k-1)}-\Psi_3^{(k-2)})\bigg)\\
&\ +F_2\bigg(\frac{{\Psi}_4^{(k-1)}-{\Psi}_4^{(k-2)}+{\Psi}_1^{(k-1)} - {\Psi}_1^{(k-2)}}{4\gamma}+ \frac{\bar c\Psi_3^{(k-1)}-\bar c\Psi_3^{(k-2)}}{2\gamma}\bigg)\bigg(\partial_r \Psi_3^{(k-2)} + \mu_4( \bm\Psi^{(k-2)}+\bar{\bf R}) \partial_t \Psi_3^{(k-2)}\bigg),\\
\end{aligned}
\end{equation} and
\begin{equation}\label{3-6-2-d}
\begin{aligned}
	&\partial_r \bigg(F_1(\Psi_1^{(k)}-\Psi_1^{(k-1)})\bigg)+\mu_{1}( \bm\Psi^{(k-1)}+\bar{\bf R})\partial_t \bigg(F_1(\Psi_1^{(k)}-\Psi_1^{(k-1)})\bigg)\\
&=\frac{F_1}r\bigg(
-\ma_{12}(r)(\Psi_1^{(k-1)} -\Psi_1^{(k-2)})+(\ma_{13}-\ma_{11}-\ma_{12})(r)
(\Psi_4^{(k)}-\Psi_4^{(k-1)})
\\
&\qquad\quad +\ma_{14}(r)(\Psi_2^{(k-1)}-\Psi_2^{(k-2)})+\ma_{15}(r)(\Psi_3^{(k-1)}-\Psi_3^{(k-2)})
\\
& \ \ +F_1\bigg(\mu_{1}( \bm\Psi^{(k-1)}+\bar{\bf R}) -\mu_1(\bar{\bf R})\bigg)\bigg(-\mm_{11}(r)(\Psi_1^{(k-1)}-\Psi_1^{(k-2)})
 +\mm_{14}(r)
(\Psi_4^{(k)}-\Psi_4^{(k-1)})\\
&\qquad\quad+\mm_{12}(r)
 (\Psi_2^{(k-1)}-\Psi_2^{(k-2)})+\mm_{13}(r)(\Psi_3^{(k-1)}-\Psi_3^{(k-2)})\bigg)\\
 &\ \ +F_1\mu_1( \bm\Psi^{(k-1)}+\bar{\bf R})\bigg(\bigg(\frac{(\Psi_2^{(k-1)})^2}{r} +\frac{\gamma-1}{8r}\bigg((\Psi_4^{(k)})^2 - (\Psi_1^{(k-1)})^2\bigg)+\frac{\bar c  }{4\gamma r }\Psi_3^{(k-1)}(\Psi_1^{(k-1)}\\
&\qquad + \Psi_4^{(k)})\bigg)-\bigg(\frac{(\Psi_2^{(k-2)})^2}{r} + \frac{\gamma-1}{8r}\bigg((\Psi_4^{(k-1)})^2 - (\Psi_1^{(k-2)})^2\bigg)+\frac{\bar c  }{4\gamma r }\Psi_3^{(k-2)}(\Psi_1^{(k-2)}+ \Psi_4^{(k-1)})\bigg)\bigg)\\
&\ \ +F_1\bigg(\mu_{1}( \bm\Psi^{(k-1)}+\bar{\bf R})-\mu_{1}(\bm\Psi^{(k-2)}+\bar{\bf R})\bigg)\bigg(-\partial_t \Psi_1^{(k-1)}-\mm_{11}(r)\Psi_1^{(k-2)}+\mm_{14}(r)\Psi_4^{(k-1)}\\ &\qquad +\mm_{12}(r)
 \Psi_2^{(k-2)} +\mm_{13}(r)\Psi_3^{(k-2)}+\frac{(\Psi_2^{(k-2)})^2}{r} + \frac{\gamma-1}{8r}\bigg((\Psi_4^{(k-1)})^2 - (\Psi_1^{(k-2)})^2\bigg)\\
&\qquad+\frac{\bar c  }{4\gamma r }\Psi_3^{(k-2)}(\Psi_1^{(k-2)}+ \Psi_4^{(k-1)})-\p_t\Psi_3^{(k-2)}\bigg(\frac{{\Psi}_4^{(k-1)} - {\Psi}_1^{(k-1)}}{4\gamma}+ \frac{\bar c\Psi_3^{(k-1)}}{2\gamma}\bigg)\bigg)\\
&\ \  -F_1\bigg(\frac{{\Psi}_4^{(k-1)} - {\Psi}_1^{(k-1)}}{4\gamma}+ \frac{\bar c\Psi_3^{(k-1)}}{2\gamma}\bigg)\bigg(\partial_r (\Psi_3^{(k-1)}-\Psi_3^{(k-2)}) + \mu_1( \bm\Psi^{(k-1)}+\bar{\bf R}) \partial_t (\Psi_3^{(k-1)}-\Psi_3^{(k-2)})\bigg)\\
&\ \ -F_1\bigg(\frac{{\Psi}_4^{(k-1)}-{\Psi}_4^{(k-2)}+{\Psi}_1^{(k-1)} - {\Psi}_1^{(k-2)}}{4\gamma}+ \frac{\bar c\Psi_3^{(k-1)}-\bar c\Psi_3^{(k-2)}}{2\gamma}\bigg)\bigg(\partial_r \Psi_3^{(k-2)}+ \mu_1( \bm\Psi^{(k-2)}+\bar{\bf R}) \partial_t \Psi_3^{(k-2)}\bigg),\\
\end{aligned}
\end{equation} and
\begin{equation}\label{3-6-3-d}
\begin{aligned}
	&\partial_r \bigg(r(\Psi_2^{(k)}-\Psi_2^{(k-1)})\bigg)+\mu_{2}( \bm\Psi^{(k-1)}+\bar{\bf R})\partial_t \bigg(r(\Psi_2^{(k)}-\Psi_2^{(k-1)})\bigg)\\
&=-\bigg(\mu_2(\bm\Psi^{(k-1)}+\bar{\bf R})-\mu_2(\bm\Psi^{(k-2)}+\bar{\bf R})\bigg)\partial_t (r\Psi_2^{(k-1)}),
\end{aligned}
\end{equation} and
\begin{equation}\label{3-6-4-d}
\begin{aligned}
	&\partial_r \bigg(\Psi_3^{(k)}-\Psi_3^{(k-1)}\bigg)+\mu_{3}( \bm\Psi^{(k-1)}+\bar{\bf R})\partial_t \bigg(\Psi_3^{(k)}-\Psi_3^{(k-1)}\bigg)\\
&=-\bigg(\mu_3(\bm\Psi^{(k-1)}+\bar{\bf R})-\mu_3(\bm\Psi^{(k-2)}+\bar{\bf R})\bigg)\partial_t \Psi_3^{(k-1)}.
\end{aligned}
\end{equation}
Furthermore, at the boundary,  one has
\begin{equation}\label{3-66-bu}
\begin{cases} r=r_0: ({\Psi}_{4}^{(k)}-{\Psi}_{4}^{(k-1)})(t,r_0) =   K_4 ({\Psi}_1^{(k-1)}-{\Psi}_1^{(k-2)})(t,r_0),\\
r=r_1:(\Psi_{1}^{(k)}-\Psi_{1}^{(k-1)})
(t,r_1)=-(\Psi_{4}^{(k)}-\Psi_{4}^{(k-1)})(t,r_1).\\\end{cases}
\end{equation}
Then  integrating \eqref{3-6-1-d} along $t_{4}^{(k)}(s;t,r)$ from $r_0$ to $r$ together with \eqref{3-11-1} and \eqref{3-15-1} yields that
\begin{equation}\label{3-15-2-es3}
\begin{aligned}
\|\Psi_4^{(k)}-\Psi_4^{(k-1)}\|_{C^{0}(D)}&\leq \frac{| K_4|M_1\sigma \eta^{k-1}}{F_2(r)}
+\frac{M_1\sigma \eta^{k-1}(F_2(r)-1)+C_7 M_1\sigma M_1\sigma \eta^{k-1}}{F_2(r)}\\
&\quad+\frac{A_0\me M_1\sigma \eta^{k-1} }{F_2(r)}\ln \frac{r_1}{r_0} .\\
\end{aligned}
\end{equation}
 Here  $C_7$ is a positive constant depending only on $ (\gamma,r_0,r_1,\rho_0,U_{1,0},U_{2,0},S_0)$. Set \begin{equation}\label{3-7-1-es2-es1-es}
 \sigma_5^\ast=\min\bigg\{\sigma_4^\ast,\frac1{2(C_7M_1^2+1)}\bigg\},\end{equation}
 where $ \sigma_4^\ast$ is given in \eqref{3-7-8-es2}.
Then for any $\sigma\in(0,\sigma_5^\ast)$  and $ \xi\in(0,\xi_0)$   with   $\xi_0$ defined by  \eqref{3-5-2-es2-6-1}, the estimates \eqref{3-5-2-es2-6-1-2} and \eqref{3-15-2-es3} imply that
\begin{equation}\label{3-5-2-es3-es-2}
\begin{aligned}
&\|\Psi_4^{(k)}-\Psi_4^{(k-1)}\|_{C^{0}(D)}
\leq M_1\sigma \eta^{k-1}-\frac{100\me\sigma \eta^{k-1}}{F_2(r)}+\frac{\sigma\eta^{k-1}}{F_1(r)}+\frac{A_0\me M_1\sigma \eta^{k-1}}{F_2(r)}\ln \frac{r_1}{r_0} \\
&\leq M_1\sigma\eta^{k-1}-\frac{99\me\sigma\eta^{k-1}}{F_2(r)}+\frac{\mf_1(\xi)\sigma \eta^{k-1}}{F_2(r)}
\leq\chi_{1}M_{1}\sigma \eta^{k-1}< M_{1}\sigma \eta^{k}.
\end{aligned}
\end{equation}
Similarly, we  integrate \eqref{3-6-2-d} along $t_{1}^{(k)}(s;t,r)$ from $r_1$ to $r$ together with the second equation in \eqref{3-66-bu} to get
\begin{equation}\label{3-5-2-es3-es-3}
\begin{aligned}
&\|\Psi_1^{(k)}-\Psi_1^{(k-1)}\|_{C^{0}(D)}\\
&\leq \frac{\chi_{1}M_{1}\sigma\eta^{k-1}}{F_1(r)}
+\frac{\chi_{1}M_{1}\sigma\eta^{k-1}(F_1(r)-1)+C_7\bigg( M_1\sigma M_1\sigma \eta^{k-1}+(\chi_{1}M_{1}\sigma \eta^{k-1})^2\bigg)}{F_1(r)}\\
&\quad+\frac{A_0\me}{F_1(r)}\ln \frac{r_1}{r_0} M_1\sigma  \eta^{k-1}\leq\chi_{1}M_{1}\sigma \eta^{k-1}+\frac{2\sigma\eta^{k-1}}{F_1(r)}+\frac{A_0\me M_1\sigma  \eta^{k-1} }{F_1(r)}\ln \frac{r_1}{r_0} \\
&\leq \chi_{1}M_{1}\sigma  \eta^{k-1}+\frac{3\sigma  \eta^{k-1}}{F_1(r)}\leq \chi_{2}M_{1}\sigma  \eta^{k-1}< M_{1}\sigma  \eta^{k}.
\end{aligned}
\end{equation}
\par Next, at the boundary $r=r_0$, one has \begin{equation}\begin{cases}
(\Psi_{2}^{(k)}-\Psi_2^{(k-1)})(t,r_0)=K_{2}(\Psi_{1}^{(k-1)}
-\Psi_{1}^{(k-2)})(t,r_0), \\
(\Psi_3^{(k)}-\Psi_3^{(k-1)})(t,r_0)=K_3(\Psi_{1}^{(k-1)}
-\Psi_{1}^{(k-2)})(t,r_0).\\ \end{cases}
\end{equation}
Then by integrating \eqref{3-6-3-d} and \eqref{3-6-4-d} along $t_{2}^{(k)}(s;t,r)$ and $t_{3}^{(k)}(s;t,r)$ from $r_0$ to $r$, respectively, there holds
\begin{equation}\label{3-7-t-1-es-1}
\begin{cases}
\|\Psi_2^{(k)}-\Psi_2^{(k-1)}\|_{C^{0}(D)}\leq (K+C_7 M_1\sigma) M_1\sigma \eta^{k-1}   ,\\
\|\Psi_3^{(k)}-\Psi_3^{(k-1)}\|_{C^{0}(D)}\leq  (K+C_7 M_1\sigma) M_1\sigma \eta^{k-1}   .\\
\end{cases}
\end{equation}
 Set
\begin{equation}\label{3-7-1-es2-es1-es2}
 \sigma_6^\ast=\min\bigg\{\sigma_5^\ast,\frac{\eta-K}{C_7M_1+1}\bigg\},
 \end{equation} where $ \sigma_5^\ast$ is given in \eqref{3-7-1-es2-es1-es}.
 Then for any $\sigma\in(0,\sigma_6^\ast)$, one derives from \eqref{3-7-t-1-es-1} that
 \begin{equation*}
\|\Psi_2^{(k)}-\Psi_2^{(k-1)}\|_{C^{0}(D)}\leq M_1\sigma \eta^{k}   ,\ \
\|\Psi_3^{(k)}-\Psi_3^{(k-1)}\|_{C^{0}(D)}\leq  M_1\sigma \eta^{k}   .\\
\end{equation*} Hence we proved \eqref{3-15-2}.
 \par { \bf Step 4. Proof of \eqref{3-16-2}-\eqref{3-17-2}.} We will show the modulus of continuity for $\Psi_{i}^{(k)}(i=1,2,3,4)$ on the temporal direction \eqref{3-16-2}, which is very important to prove \eqref{3-17-2}. For $ \delta\in(0,1)$, set
 \begin{equation*}
 \begin{aligned}
 \mn(\delta)&=\frac{24\me}{1-K}[\mk+1]\bigg(\sqrt{\sigma}\delta+\varpi(\delta|\mh'_1)
+\varpi(\delta|\mh'_2)+\varpi(\delta|\mh'_3) +\varpi(\delta|\mh'_4)\bigg),
 \end{aligned}\end{equation*}
 where $K=\max\{|K_2|,|K_3|,|K_4|\}.$ Note that $\varpi(\delta|\mh'_{i})(i=1,2,3,4)$ are monotonically increasing, bounded and continuous concave functions of $\delta$ and $\mathop{\lim}\limits_{\delta\rightarrow0^{+}}\varpi(\delta|\mh'_{i})=0$. Thus $\mn(\delta)$ has the same features and
$$\mathop{\lim}\limits_{\delta\rightarrow0^{+}}\mn(\delta)=0.$$
At the boundary $r=r_0$, for any given $t^{1},t^{2}\in\mathbb{R}$ with $|t^{1}-t^{2}|\leq\delta\ll1$, one has
$$
|\Upsilon_{4}^{(k)}(t^{1},r_0)-\Upsilon_{4}^{(k)}(t^{2},r_0)|
\leq \left|\mh'_4(t^{1})-\mh'_4(t^{2})\right|+|K_4|\left|\Upsilon_{1}^{(k-1)}(t^{1},r_0)
-\Upsilon_{1}^{(k-1)}(t^{2},r_0)\right|.
$$
This, together with \eqref{3-16-1}, yields that
\begin{equation}\label{3-7-1-es2-es4}
\begin{aligned}
\varpi\bigg(\delta|\Upsilon_{4}^{(k)}(\cdot,r_0)\bigg)&\leq \varpi(\delta|\mh'_4)+|K_4|\varpi\bigg(\delta|\Upsilon_{1}^{(k-1)}(\cdot,r_0)\bigg)
<\frac{1-K}{24\me[\mathcal{K}+1]}\mn(\delta)+\frac{|K_4|}{8[\mathcal{K}+1]}
\mn(\delta)\\
&<\frac{1-K}{24\me[\mathcal{K}+1]}\mn(\delta)+\frac{K}{8[\mathcal{K}+1]}
\mn(\delta)<\bigg(\frac{1}{24[\mathcal{K}+1]}
+\frac{K}{12[\mathcal{K}+1]}\bigg)\mn(\delta)\\
&<\frac{\mn(\delta)}{8[\mathcal{K}+1]}.
\end{aligned}\end{equation}
In the domain $D$, for any $r\in[r_0,r_1]$ and $t^{1},t^{2}\in\mathbb{R}$ with $|t^{1}-t^{2}|\leq\delta$, by the definition of the characteristic curve, there holds
$$
t_{4}^{(k)}(z;t^{*},r)=\int_{r}^{z}\mu_{4}
( \bm\Psi^{(k-1)}+\bar{\bf R})(t_{4}^{(k)}(s;t^\ast,r),s)\de s+t^{*}.
$$
Then one has
\begin{equation*}
\begin{aligned}
&|t_{4}^{(k)}(z;t^{1},r)-t_{4}^{(k)}(z;t^{2},r)|\\
\leq&|t^{1}-t^{2}|+\int_{r}^{z}
\bigg|\mu_{4}
( \bm\Psi^{(k-1)}+\bar{\bf R})(t_{4}^{(k)}(s;t^1,r),s)-\mu_{4}
( \bm\Psi^{(k-1)}+\bar{\bf R})(t_{4}^{(k)}(s;t^2,r),s)\bigg|\de s\\
\leq&|t^{1}-t^{2}|+\int_{r}^{z}\sum_{i=1}^4
\bigg|\frac{\partial\mu_{4}}{\p\Psi_i^{(k-1)}}\bigg|  \bigg|\Upsilon_i^{(k-1)}\bigg| \bigg|
t_{4}^{(k)}(s;t^{1},r)
-t_{4}^{(k)}(s;t^{2},r)\bigg|\de s.
\end{aligned}\end{equation*}
By the Gronwall inequality and \eqref{3-11-1}, one derives that for sufficiently small $ \sigma$,
\begin{align}
|t_{4}^{(k)}(s;t^{1},r)
-t_{4}^{(k)}(s;t^{2},r)\bigg|\leq e^{C\sigma}|t^{1}-t^{2}|\leq(1+\sqrt{\sigma})\delta.\label{c55}
\end{align}
Noticing the concavity of $\mn(\delta)$, there holds
$$\frac{\mn((1+\sqrt{\sigma})\delta)}{1+\sqrt{\sigma}}
+\frac{\sqrt{\sigma}\mn(0)}{1+\sqrt{\sigma}}\leq \mn(\delta).$$
 This implies
$$\mn((1+\sqrt{\sigma})\delta)\leq(1+\sqrt{\sigma})\mn(\delta).$$
Thus it follows from \eqref{3-16-1} that
\begin{equation}\label{c56}
\begin{aligned}
&|\Upsilon_{i}^{(k-1)}(t_{4}^{(k)}(z;t^{1},r),z)-\Upsilon_{i}^{(k-1)}
(t_{4}^{(k)}(z;t^{2},r),z)|\\
\leq &\frac{\mn((1+\sqrt{\sigma})\delta)}{8[\mathcal{K}+1]}
\leq\frac{(1+\sqrt{\sigma})\mn(\delta)}{8[\mathcal{K}+1]}, \ i=1,2,3,4.\end{aligned}\end{equation}
Integrating \eqref{3-6-t} along $t=t_{4}^{(k)}(z;t^{1},r)$ and $t=t_{4}^{(k)}(z;t^{2},r)$ respectively and subtracting the two results, one gets from \eqref{3-11-1}, \eqref{3-7-1-es2-es4}-\eqref{c56} and the Gronwall's inequality that
\begin{align}
|\Upsilon_{4}^{(k)}(t^{1},r)-\Upsilon_{4}^{(k)}(t^{2},r)|\leq\frac{C_{8}\mn(\delta)}{8[\mathcal{K}+1]}
,\label{c57}
\end{align}
where $C_{8}>0$ is a constant independent of $k$. Then we  use the integral expression \eqref{3-6-t} of $F_{2}(r)\Upsilon_4^{(k)}(t,r)$ together with  the estimates \eqref{3-11-1}, \eqref{3-16-2} \eqref{3-5-2-es2-6-1-2} and \eqref{3-7-1-es2-es4}-\eqref{c57} to obtain that
\begin{equation}\label{c58}
\begin{aligned}
 \varpi\bigg(\delta|\Upsilon_{4}^{(k)}(\cdot,r)\bigg)&\leq \frac{\mn(\delta)(1+\sqrt{\sigma})}{F_2(r)}\bigg(\frac{1}{24[\mathcal{K}+1]}
+\frac{K}{12[\mathcal{K}+1]}\bigg)\\
&\quad+\frac{F_2(r)-1}{F_2(r)}\frac{(1+\sqrt{\sigma})\mn(\delta)}
{8[\mathcal{K}+1]}+\frac{C_9M_1\sigma}{F_2(r)}\bigg(
\frac{(1+\sqrt{\sigma})\mn(\delta)}{8[\mathcal{K}+1]}+(1+\sqrt{\sigma})\delta\bigg)\\
&\quad+\frac{A_0\me M_1 }{F_2(r)M_1}\ln \frac{r_1}{r_0} \frac{(1+\sqrt{\sigma})\mn(\delta)}
{8[\mathcal{K}+1]}\\
&\leq \frac{(1+\sqrt{\sigma})\mn(\delta)}
{8[\mathcal{K}+1]}-\frac{\mn(\delta)(1+\sqrt{\sigma})}{\me}
\frac{1-K}{12[\mathcal{K}+1]}+C_9M_1\sigma\frac{(1+\sqrt{\sigma})
\mn(\delta)}{8[\mathcal{K}+1]}\\
&\quad+C_9M_1\sigma(1+\sqrt{\sigma})+ \frac{(1-K)\mf_1(\xi) }{100\me}\frac{(1+\sqrt{\sigma})\mn(\delta)}
{8[\mathcal{K}+1]}\\
&< \frac{(1+\sqrt{\sigma})\mn(\delta)}
{8[\mathcal{K}+1]}+\frac{\mn(\delta)(1+\sqrt{\sigma})}{\me}
\frac{(1-K)}{[\mathcal{K}+1]}\bigg(\frac1{800}-\frac1{12}\bigg)\\
&\quad+C_9M_1\sigma\frac{(1+\sqrt{\sigma})
\mn(\delta)}{8[\mathcal{K}+1]}+C_9M_1\sigma(1+\sqrt{\sigma})\delta\\
&< \frac{(1+\sqrt{\sigma})\mn(\delta)}
{8[\mathcal{K}+1]}-\frac{\mn(\delta)(1+\sqrt{\sigma})}{\me}
\frac{(1-K)}{13[\mathcal{K}+1]}\\
&\quad+C_9M_1\sigma\frac{(1+\sqrt{\sigma})
\mn(\delta)}{8[\mathcal{K}+1]}+C_9M_1(\sigma+\sqrt\sigma)\frac{\mn(\delta)}
{[\mathcal{K}+1]}.\\
\end{aligned}\end{equation}
Here $C_9>0$  is a constant depending only on $ (\gamma,r_0,r_1,\rho_0,U_{1,0},U_{2,0},S_0)$. Define
$$g_2(\sigma)= \frac{\sqrt{\sigma}}
{8}-\frac{(1+\sqrt{\sigma})}{\me}
\frac{1-K}{13}+C_9M_1\sigma\frac{(1+\sqrt{\sigma})
}{8}+C_9M_1(\sigma+\sqrt\sigma).$$   It follows from the continuity of the  function $ g_2(\sigma)$ and $ g_2(0) =-\frac{1-K}{13\me}$ that there exists a sufficiently small $ \sigma_7^\ast>0 $ such that for any $\sigma\in(0,\sigma_7^\ast)$,
\begin{equation*}
g_2(\sigma)<g_2(0)+\frac{|g_2(0)|}4=-\frac{3(1-K)}{52\me}.
 \end{equation*} This, together with the estimate \eqref{c58}, implies that
 \begin{equation}\label{c58-1}
\begin{aligned}
 \varpi\bigg(\delta|\Psi_{4}^{(k)}(\cdot,r)\bigg)<
 \frac{\chi_3\mn(\delta)}
{8[\mathcal{K}+1]},
\end{aligned}\end{equation}where $\chi_3=1-\frac{24(1-K)}{52\me}$ is a constant satisfying $0<\chi_3<1$.
\par At the boundary $ r=r_1$,  one gets
\begin{equation}\label{cc58}
\begin{aligned}
\varpi\bigg(\delta|\Upsilon_{1}^{(k)}(\cdot,r_1)\bigg)&\leq \varpi(\delta|\mh'_1)+\varpi\bigg(\delta|\Upsilon_{4}^{(k)}(\cdot,r_1)\bigg)
<\frac{1-K}{24\me[\mathcal{K}+1]}\mn(\delta)+\frac{\chi_{3}}
{8[\mathcal{K}+1]}\mn(\delta)\\
&<\bigg(1-\frac{24(1-K)}{52\me}+\frac{1-K}{3\me}\bigg)
\frac{\mn(\delta)}{8[\mathcal{K}+1]}
=\bigg(1-\frac{5(1-K)}{39\me}\bigg)
\frac{\mn(\delta)}{8[\mathcal{K}+1]}\\
&<\frac{\mn(\delta)}{8[\mathcal{K}+1]}.
\end{aligned}\end{equation}
Furthermore, similar to \eqref{c55}-\eqref{c57}, one derives
 \begin{equation*}
\begin{aligned}
&\left|t_{1}^{(k)}(s;t^{1},r)
-t_{1}^{(k)}(s;t^{2},r)\right|\leq e^{C\sigma}|t^{1}-t^{2}|\leq(1+\sqrt{\sigma})\delta,\\
&\left|\Upsilon_{i}^{(k-1)}(t_{1}^{(k)}(z;t^{1},r),z)-\Upsilon_{i}^{(k-1)}
(t_{1}^{(k)}(z;t^{2},r),z)\right|
\leq\frac{(1+\sqrt{\sigma})\mn(\delta)}{8[\mathcal{K}+1]}, \ i=1,2,3,4,\\
&\left|\Upsilon_{1}^{(k)}(t^{1},r)
-\Upsilon_{1}^{(k)}(t^{2},r)\right|\leq\frac{C_{8}\mn(\delta)}{8[\mathcal{K}+1]}.
\end{aligned}\end{equation*}
Then  one can conclude from  the integral expression \eqref{3-5-t} of $F_{1}(r)\Upsilon_1^{(k)}(t,r)$   that
\begin{equation}\label{c59}
\begin{aligned}
 \varpi\bigg(\delta|\Upsilon_{1}^{(k)}(\cdot,r)\bigg)&\leq \frac{(1+\sqrt{\sigma})}{F_1(r)}\bigg(1-\frac{5(1-K)}{39\me}
 \bigg)
\frac{\mn(\delta)}{8[\mathcal{K}+1]}\\
&\quad+\frac{F_1(r)-1}{F_1(r)}\frac{(1+\sqrt{\sigma})\mn(\delta)}
{8[\mathcal{K}+1]}+\frac{C_9M_1\sigma}{F_1(r)}
\bigg(\frac{(1+\sqrt{\sigma})\mn(\delta)}{8[\mathcal{K}+1]}+(1+\sqrt{\sigma})\bigg)\\
&\quad+\frac{A_0\me M_1 }{F_1(r)M_1}\ln \frac{r_1}{r_0} \frac{(1+\sqrt{\sigma})\mn(\delta)}
{8[\mathcal{K}+1]}\\
&< \frac{(1+\sqrt{\sigma})\mn(\delta)}
{8[\mathcal{K}+1]}-\frac{\mn(\delta)(1+\sqrt{\sigma})}{\me}
\frac{5(1-K)}{312[\mathcal{K}+1]}
+C_9M_1\sigma\frac{(1+\sqrt{\sigma})
\mn(\delta)}{8[\mathcal{K}+1]}\\
&\quad+C_9M_1\sigma(1+\sqrt{\sigma})+ \frac{(1-K)\mf_1(\xi)) }{100\me}\frac{(1+\sqrt{\sigma})\mn(\delta)}
{8[\mathcal{K}+1]}\\
&<
 \frac{(1+\sqrt{\sigma})\mn(\delta)}
{8[\mathcal{K}+1]}-\frac{\mn(\delta)(1+\sqrt{\sigma})}{\me}
\frac{(1-K)}{68[\mathcal{K}+1]}
+C_9M_1\sigma\frac{(1+\sqrt{\sigma})
\mn(\delta)}{8[\mathcal{K}+1]}\\
&\quad+C_9M_1(\sigma+\sqrt\sigma)\frac{\mn(\delta)}
{[\mathcal{K}+1]}.
\\
\end{aligned}\end{equation}
Define
$$g_3(\sigma)= \frac{\sqrt{\sigma}}
{8}-\frac{(1+\sqrt{\sigma})}{\me}
\frac{1-K}{68}+C_9M_1\sigma\frac{(1+\sqrt{\sigma})
}{8}+C_9M_1(\sigma+\sqrt\sigma).$$   By the continuity of the  function $ g_3(\sigma)$ and $ g_3(0) =-\frac{1-K}{68\me}$, there exists a sufficiently small $ \sigma_8^\ast>0 $ such that for any $\sigma\in(0,\sigma_8^\ast)$,
\begin{equation*}
g_3(\sigma)<g_3(0)+\frac{|g_3(0)|}2=-\frac{1-K}{136\me}.
 \end{equation*} Thus one gets \begin{equation*}
\begin{aligned}
 \varpi\bigg(\delta|\Upsilon_{1}^{(k)}(\cdot,r)\bigg)<
 \frac{\chi_4\mn(\delta)}
{8[\mathcal{K}+1]},
\end{aligned}\end{equation*}where $\chi_4=1-\frac{1-K}{17\me}$ is a constant satisfying $0<\chi_4<1$.
\par In the following,
at the boundary $ r=r_0$,  one gets
\begin{equation}\label{cc60}
\begin{cases}
\begin{aligned}
&\varpi\bigg(\delta|\Upsilon_{2}^{(k)}(\cdot,r_0)\bigg)\leq \varpi(\delta|\mh'_2)+|K_2|\varpi\bigg(\delta|\Upsilon_{1}^{(k-1)}(\cdot,r_0)\bigg)
<\frac{1-K}{24\me[\mathcal{K}+1]}\mn(\delta)+\frac{|K_2|}{8[\mathcal{K}+1]}
\mn(\delta)\\
&<\bigg(\frac{1}{24[\mathcal{K}+1]}
+\frac{K}{12[\mathcal{K}+1]}\bigg)\mn(\delta)
<\frac{\mn(\delta)}{8[\mathcal{K}+1]},\\
&\varpi\bigg(\delta|\Upsilon_{3}^{(k)}(\cdot,r_0)\bigg)\leq \varpi(\delta|\mh'_3)+|K_3|\varpi\bigg(\delta|\Upsilon_{1}^{(k-1)}(\cdot,r_0)\bigg)
<\frac{1-K}{24\me[\mathcal{K}+1]}\mn(\delta)+\frac{|K_3|}{8[\mathcal{K}+1]}
\mn(\delta)\\
&<\bigg(\frac{1}{24[\mathcal{K}+1]}
+\frac{K}{12[\mathcal{K}+1]}\bigg)\mn(\delta)
<\frac{\mn(\delta)}{8[\mathcal{K}+1]}.\\
\end{aligned}\end{cases}\end{equation}
Then it follows from \eqref{3-7-t}-\eqref{3-8-t} together with the  Gronwall inequality  and the estimates \eqref{3-11-1}, \eqref{3-16-1} and \eqref{cc60} that
\begin{equation}\label{cc60-1}\begin{cases}
\begin{aligned}
\varpi\bigg(\delta|\Upsilon_{2}^{(k)}(\cdot,r)\bigg)&\leq \bigg(\bigg(\frac{1}{24[\mathcal{K}+1]}
+\frac{K}{12[\mathcal{K}+1]}\bigg)(1+\sqrt{\sigma})\mn(\delta)+
{C_{10}M_1\sigma}\bigg(
\frac{(1+\sqrt{\sigma})\mn(\delta)}{8[\mathcal{K}+1]}\\
&\qquad+(1+\sqrt{\sigma})\delta\bigg)\bigg)e^{C_{10}M_1\sigma(r_1-r_0)},\\
\varpi\bigg(\delta|\Upsilon_{3}^{(k)}(\cdot,r)\bigg)&\leq \bigg(\bigg(\frac{1}{24[\mathcal{K}+1]}
+\frac{K}{12[\mathcal{K}+1]}\bigg)(1+\sqrt{\sigma})\mn(\delta)+
{C_{10}M_1\sigma}\bigg(
\frac{(1+\sqrt{\sigma})\mn(\delta)}{8[\mathcal{K}+1]}\\
&\qquad+(1+\sqrt{\sigma})\delta\bigg)\bigg)e^{C_{10}M_1\sigma(r_1-r_0)}.\\
\end{aligned}\end{cases}\end{equation}
Here $C_{10}>0$  is a constant depending only on $ (\gamma,r_0,r_1,\rho_0,U_{1,0},U_{2,0},S_0)$. Then  there exists a sufficiently small $ \sigma_{9}^\ast>0 $ such that for any $\sigma\in(0,\sigma_{9}^\ast)$,
\begin{equation*}
\varpi\bigg(\delta|\Upsilon_{2}^{(k)}(\cdot,r)\bigg)\leq \frac{\mn(\delta)}
{8[\mathcal{K}+1]}, \ \ \varpi\bigg(\delta|\Upsilon_{3}^{(k)}(\cdot,r)\bigg)\leq \frac{\mn(\delta)}
{8[\mathcal{K}+1]}
.\end{equation*}
\par Finally, we prove \eqref{3-17-2}. We first consider the special case that two given points $(t^{1},r^{1})$ and $(t^{2},r^{2})$ with $|t^{1}-t^{2}|\leq\delta,|r^{1}-r^{2}|\leq\delta$ locate on the same characteristic curve $t=t_{4}^{(k)}(s;t,r)$, namely, $t^{2}=t_{4}^{(k)}(r^2;t^{1},r^{1})$. Using the similar method of \eqref{3-5-2-es3-7}, one gets
\begin{align}
\left|\Upsilon_{4}^{(k)}(t^{1},r^{1})-\Upsilon_{4}^{(k)}(t^{2},r^{2})\right|\leq C\sigma\delta\leq\frac{\mn(\delta)}{12}.\label{c59-d}
\end{align}
Then, for general two points $(t^{1},r^1)$ and $(t^{2},r^2)$ with $|t^{1}-t^{2}|\leq\delta,|r^{1}-r^{2}|\leq\delta$, we can choose a point $(t^{3},r^{1})$ locating on the  fourth characteristic curve passing through $(t^{2},r^{2})$, namely, $t^{3}=t_{4}^{(k)}(r^1;t^{2},r^{2})$. By \eqref{1-15} and \eqref{c24}, there hold
$$|t^{3}-t^{2}|\leq|\mu_{4}||r^{1}-r^{2}|\leq \mathcal{K}\delta \ \
{\rm{and}}
\  \ |t^{3}-t^{1}|\leq|t^{3}-t^{2}|+|t^{2}-t^{1}|\leq(\mathcal{K}+1)\delta.$$
Then  we combine estimates \eqref{c58-1} and \eqref{c59-d} to get
\begin{equation}\label{cc60-2}
\begin{aligned}
&\left|\Upsilon_{4}^{(k)}(t^{1},r^{1})-\Upsilon_{4}^{(k)}(t^{2},r^{2})\right|
\leq\bigg|\Upsilon_{4}^{(k)}(t^{1},r^{1})-\Upsilon_{4}^{(k)}
\bigg(\frac{[\mathcal{K}+1]t^{1}+t^{3}}{\mathcal{K}+1},r^{1}\bigg)\bigg|\\
&+\bigg|\Upsilon_{4}^{(k)}\bigg(\frac{[\mathcal{K}+1]t^{1}+t^{3}}{\mathcal{K}+1},r^{1}\bigg)
-\Upsilon_{4}^{(k)}\bigg(\frac{([\mathcal{K}+1]-1)t^{1}+2t^{3}}{\mathcal{K}+1},r^{1}\bigg)\bigg|\\
&+\ldots+\bigg|\Upsilon_{4}^{(k)}\bigg(\frac{t^{1}+[\mathcal{K}+1]t^{3}}
{\mathcal{K}+1},r^1\bigg)-\Upsilon_{4}^{(k)}(t^{3},r^{1})\bigg|
+\left|\Upsilon_{4}^{(k)}(t^{3},r^{1})-\Upsilon_{4}^{(k)}(t^{2},r^{2})\right|\\
\leq&\frac{[\mathcal{K}+1]+1}{8[\mathcal{K}+1]}
\mn(\delta)+\frac{\mn(\delta)}{12}
\leq\frac{\mn(\delta)}{3}.
\end{aligned}\end{equation}
The combination of \eqref{c59-d} and \eqref{cc60-2} leads to
\begin{align}
\varpi\bigg(\delta|\Upsilon_{4}^{(k)}\bigg)\leq\frac{\mn(\delta)}{3}.\label{c61}
\end{align}
Then it follows from \eqref{3-6} together with the estimates \eqref{1-15}  \eqref{3-11-2}, \eqref{c55} and  \eqref{c61} that
\begin{align}
\varpi\bigg(\delta|\p_r\Psi_{4}^{(k)}\bigg)\leq\frac{\mk\mn(\delta)}{2}.\label{c61-1}
\end{align}
In a similar way, one obtains
\begin{equation}\label{c62}
\begin{cases}
\begin{aligned}
&\varpi\bigg(\delta|\Upsilon_{1}^{(k)}\bigg)\leq\frac{\mn(\delta)}{3},\ \ \varpi\bigg(\delta|\p_r\Psi_{1}^{(k)}\bigg)\leq\frac{\mk\mn(\delta)}{2},\\
&\varpi\bigg(\delta|\Upsilon_{2}^{(k)}\bigg)\leq\frac{\mn(\delta)}{3},\ \ \varpi\bigg(\delta|\p_r\Psi_{2}^{(k)}\bigg)\leq\frac{\mk\mn(\delta)}{2},\\
&\varpi\bigg(\delta|\Upsilon_{3}^{(k)}\bigg)\leq\frac{\mn(\delta)}{3},\ \ \varpi\bigg(\delta|\p_r\Psi_{3}^{(k)}\bigg)\leq\frac{\mk\mn(\delta)}{2}.\\
\end{aligned}
\end{cases}
\end{equation}
Thus collecting \eqref{c61-1}-\eqref{c62} yields  \eqref{3-17-2}. Set
$$ \sigma_1=\min\{\sigma_6^\ast,\sigma_7^\ast, \sigma_8^\ast,\sigma_9^\ast\},$$ where  $\sigma_6^\ast$ is defined by \eqref{3-7-1-es2-es1-es2},  $\sigma_7^\ast$,  $\sigma_8^\ast$  and $\sigma_9^\ast$ are given in { \bf Step 4}.
For any $\sigma\in(0,\sigma_1)$, the estimates \eqref{3-10}-\eqref{3-16} hold. \end{proof}
\par Now, we  prove Theorem \ref{th1}. By  \eqref{3-15},
the sequence $\{{\bm\Psi}^{(k)}\}_{k=1}^{\infty}$ is a Cauchy sequence in $C^0(D)$, and thus converges to some  function ${\bm\Psi}^{(T)}$ uniformly in $C^0(D)$. Furthermore, \eqref{3-10} implies  ${\bm\Psi}^{(T)}$  is time-periodic. Then by \eqref{3-11} and \eqref{3-16} together with Arzela-Ascoli
theorem, we know that there exists a subsequence of $\{{\bm\Psi}^{(k)}\}_{k=1}^{\infty}$, which converges uniformly in $C^1(D)$. By the
uniqueness of the limit, we know that the whole original sequence $\{{\bm\Psi}^{(k)}\}_{k=1}^{\infty}$ converges to ${\bm\Psi}^{(T)}$ in $C^1(D)$.
Thus ${\bm\Psi}^{(T)}$ is a $C^1$ smooth function. Then we can take
${\bm\Psi}_{0}(r) ={ \bm\Psi}^{(T)}(0, r)$ to get the initial data, which clearly satisfies \eqref{1-19}.
\section{ $C^{0}$ stability of the time-periodic solution}\noindent
\par In this section, we will prove Theorem \ref{th2}.  By the results of \cite{Li,Yuw},  we know that for $ \sigma\in(0,\sigma_1)$, the   global existence of the classical solutions $\bm\Psi(t,r)$ and $ {\bm\Psi}^{(T)}(t,r)$  to the corresponding
problem \eqref{1-13-mo-1}-\eqref{1-14} satisfy
\begin{equation}
\mathop{\max}\limits_{i=1,2,3,4}\bigg\{\|\Psi_i\|_{C^{1}(D)},\|{\Psi}_i^{(T)}\|_{C^{1}(D)}\bigg\}\leq \mc_1\sigma.
\end{equation}
In order to prove \eqref{1-24} inductively, suppose that for any $t\in[t_{0},t_0+T_{0}]$ with $ t_0>0$  and $l\in\mathbb{N}$,
\begin{equation}\label{d22}
\begin{cases}
\|\Psi_{1}(t,\cdot)-\Psi_{1}^{(T)}(t,\cdot)\|_{C^{0}([r_0,r_1])}\leq A_2\mc_2\sigma \varepsilon^{l},\
\|\Psi_{2}(t,\cdot)-\Psi_{2}^{(T)}(t,\cdot)\|_{C^{0}([r_0,r_1])}\leq A_2\mc_2 \sigma \varepsilon^{l},\\
\|\Psi_{3}(t,\cdot)-\Psi_{3}^{(T)}(t,\cdot)\|_{C^{0}([r_0,r_1])}\leq A_2\mc_2 \sigma \varepsilon^{l},\
\|\Psi_{4}(t,\cdot)-\Psi_{4}^{(T)}(t,\cdot)\|_{C^{0}([r_0,r_1])}\leq \mc_2 \sigma \varepsilon^{l},
\end{cases}
\end{equation}
 we will show that for any $t\in[t_{0}+T_{0},t_{0}+2T_{0}]$,
\begin{equation}\label{d22-1}\begin{cases}
\|\Psi_{1}(t,\cdot)-\Psi_{1}^{(T)}(t,\cdot)\|_{C^{0}([r_0,r_1])}\leq A_2\mc_2\sigma \varepsilon^{l+1},\
\|\Psi_{2}(t,\cdot)-\Psi_{2}^{(T)}(t,\cdot)\|_{C^{0}([r_0,r_1])}\leq A_2\mc_2 \sigma \varepsilon^{l+1},\\
\|\Psi_{3}(t,\cdot)-\Psi_{3}^{(T)}(t,\cdot)\|_{C^{0}([r_0,r_1])}\leq A_2\mc_2 \sigma \varepsilon^{l+1},\
\|\Psi_{4}(t,\cdot)-\Psi_{4}^{(T)}(t,\cdot)\|_{C^{0}([r_0,r_1])}\leq \mc_2 \sigma \varepsilon^{l+1}.
\end{cases}
\end{equation}
Here $A_2>1$ and $ \varepsilon\in(0,1)$  are constants will be determined   and ${\bm\Psi}^{(T)}(t,r)$ is the time-periodic solution obtained in Theorem \ref{th1}.
\par Firstly, one can follow from \eqref{3-6-1}-\eqref{3-5-1} and \eqref{3-3}-\eqref{3-4} to derive that
\begin{equation}\label{5-6-1}
\begin{aligned}
	&\partial_r \bigg(F_2(\Psi_{4}-\Psi_{4}^{(T)})\bigg)+\mu_{4}( \bm\Psi+\bar{\bf R})\partial_t \bigg(F_2(\Psi_{4}-\Psi_{4}^{(T)})\bigg) \\ &=\frac{F_2}r\bigg(
(\ma_{42}-\ma_{41})(r)(\Psi_{1}-\Psi_{1}^{(T)})+\ma_{43}(r) (\Psi_{2}-\Psi_{2}^{(T)})-\ma_{44}(r)(\Psi_{3}-\Psi_{3}^{(T)})\bigg)\\
&\quad  +F_2\bigg(\mu_{4}( \bm\Psi+\bar{\bf R})-\mu_4(\bar{\bf R})\bigg)\bigg(\mm_{41}(r)(\Psi_{1}-\Psi_{1}^{(T)})-\mm_{44}(r)
(\Psi_{4}-\Psi_{4}^{(T)})\\
&\qquad+\mm_{42}(r)
 (\Psi_{2}-\Psi_{2}^{(T)})-\mm_{43}(r)(\Psi_3-\Psi_{3}^{(T)})\bigg)\\
&\quad +F_2\mu_4( \bm\Psi+\bar{\bf R})\bigg(\bigg(\frac{(\Psi_2)^2}{r} - \frac{\gamma-1}{8r}\bigg((\Psi_4)^2 - (\Psi_1)^2\bigg)-\frac{\bar c  }{4\gamma r }\Psi_3(\Psi_1+ \Psi_4)\bigg)\\
&\qquad-\bigg(\frac{(\Psi_2^{(T)})^2}{r} - \frac{\gamma-1}{8r}\bigg((\Psi_4^{(T)})^2 - (\Psi_1^{(T)})^2\bigg)-\frac{\bar c  }{4\gamma r }\Psi_3^{(T)}(\Psi_1^{(T)}+ \Psi_4^{(T)})\bigg)\bigg)\\
&\quad +F_2\bigg(\mu_4(\bm\Psi+\bar{\bf R})-\mu_4(\bm\Psi^{(T)}+\bar{\bf R}\bigg)\bigg(-\partial_t\Psi_{4}^{(T)}+ \mm_{41}(r)\Psi_{1}^{(T)}+\mm_{44}(r)
\Psi_{4}^{(T)}+\mm_{42}(r)
 \Psi_{2}^{(T)}\\
 &\qquad\quad -\mm_{43}(r)\Psi_{3}^{(T)}+\bigg(\frac{(\Psi_2^{(T)})^2}{r} - \frac{\gamma-1}{8r}\bigg((\Psi_4^{(T)})^2 - (\Psi_1^{(T)})^2\bigg)-\frac{\bar c  }{4\gamma r }\Psi_3^{(T)}(\Psi_1^{(T)}+ \Psi_4^{(T)})\bigg)\\
 &\qquad\quad+\p_t\Psi_3^{(T)}\bigg(\frac{{\Psi}_4 - {\Psi}_1}{4\gamma}+ \frac{\bar c\Psi_3}{2\gamma}\bigg)\bigg) \\
&\quad  +F_2\bigg(\frac{{\Psi}_4 - {\Psi}_1}{4\gamma}+ \frac{\bar c\Psi_3}{2\gamma}\bigg)\bigg(\partial_r (\Psi_3-\Psi_{3}^{(T)}) + \mu_4( \bm\Psi+\bar{\bf R}) \partial_t (\Psi_3-\Psi_{3}^{(T)})\bigg)\\
 &\quad +F_2\bigg(\frac{{\Psi}_4-{\Psi}_4^{(T)}+{\Psi}_1 - {\Psi}_1^{(T)}}{4\gamma}+ \frac{\bar c\Psi_3-\bar c\Psi_3^{(T)}}{2\gamma}\bigg)\bigg(\partial_r \Psi_3^{(T)} + \mu_4( \bm\Psi^{(T)}+\bar{\bf R}) \partial_t \Psi_3^{(T)}\bigg),\\
\end{aligned}
\end{equation}
and
\begin{equation*}
\begin{aligned}
	&\partial_r \bigg(F_1(\Psi_{1}-\Psi_{1}^{(T)})\bigg)+\mu_{1}( \bm\Psi+\bar{\bf R})\partial_t \bigg(F_1(\Psi_{1}-\Psi_{1}^{(T)})\bigg) \\ &=\frac{F_1}r\bigg(
-\ma_{12}(r)(\Psi_{1}-\Psi_{1}^{(T)})+(\ma_{13}-\ma_{11}-\ma_{12})(r)
(\Psi_{4}-\Psi_{4}^{(T)})
+\ma_{14}(r)(\Psi_{2}-\Psi_{2}^{(T)})
\\
&\qquad  +\ma_{15}(r)(\Psi_{3}-\Psi_{3}^{(T)})\bigg)+F_1\bigg(\mu_{1}( \bm\Psi+\bar{\bf R})-\mu_1(\bar{\bf R})\bigg)\bigg(-\mm_{11}(r)(\Psi_{1}-\Psi_{1}^{(T)})\\
&\qquad+\mm_{14}(r)
(\Psi_{4}-\Psi_{4}^{(T)})+\mm_{12}(r)
 (\Psi_{2}-\Psi_{2}^{(T)})+\mm_{13}(r)(\Psi_3-\Psi_{3}^{(T)})\bigg)\\
&\quad +F_1\mu_1( \bm\Psi+\bar{\bf R})\bigg(\bigg(\frac{(\Psi_2)^2}{r} + \frac{\gamma-1}{8r}\bigg((\Psi_4)^2 - (\Psi_1)^2\bigg)+\frac{\bar c  }{4\gamma r }\Psi_3(\Psi_1+ \Psi_4)\bigg)\\
&\qquad-\bigg(\frac{(\Psi_2^{(T)})^2}{r} +\frac{\gamma-1}{8r}\bigg((\Psi_4^{(T)})^2 - (\Psi_1^{(T)})^2\bigg)+\frac{\bar c  }{4\gamma r }\Psi_3^{(T)}(\Psi_1^{(T)}+ \Psi_4^{(T)})\bigg)\bigg)\\
&\quad +F_1\bigg(\mu_1(\bm\Psi+\bar{\bf R})-\mu_1(\bm\Psi^{(T)}+\bar{\bf R}\bigg)\bigg(-\partial_t\Psi_{1}^{(T)}-\mm_{11}(r)\Psi_{1}^{(T)}+\mm_{44}(r)
\Psi_{4}^{(T)}+\mm_{12}(r)
 \Psi_{2}^{(T)}\\
 &\qquad\quad +\mm_{13}(r)\Psi_{3}^{(T)}+\bigg(\frac{(\Psi_2^{(T)})^2}{r} + \frac{\gamma-1}{8r}\bigg((\Psi_4^{(T)})^2 - (\Psi_1^{(T)})^2\bigg)+\frac{\bar c  }{4\gamma r }\Psi_3^{(T)}(\Psi_1^{(T)}+ \Psi_4^{(T)})\bigg)\\ \end{aligned}
\end{equation*}\begin{equation}\label{5-6-2}\begin{aligned}
 &\qquad\quad-\p_t\Psi_3^{(T)}\bigg(\frac{{\Psi}_4 - {\Psi}_1}{4\gamma}+ \frac{\bar c\Psi_3}{2\gamma}\bigg)\bigg) \\
&\quad  -F_1\bigg(\frac{{\Psi}_4 - {\Psi}_1}{4\gamma}+ \frac{\bar c\Psi_3}{2\gamma}\bigg)\bigg(\partial_r (\Psi_3-\Psi_{3}^{(T)}) + \mu_1( \bm\Psi+\bar{\bf R}) \partial_t (\Psi_3-\Psi_{3}^{(T)})\bigg)\\
 &\quad -F_1\bigg(\frac{{\Psi}_4-{\Psi}_4^{(T)}+{\Psi}_1 - {\Psi}_1^{(T)}}{4\gamma}+ \frac{\bar c\Psi_3-\bar c\Psi_3^{(T)}}{2\gamma}\bigg)\bigg(\partial_r \Psi_3^{(T)} + \mu_1( \bm\Psi^{(T)}+\bar{\bf R}) \partial_t \Psi_3^{(T)}\bigg),\\
\end{aligned}
\end{equation}
and \begin{equation}\label{5-6-3} \partial_r \bigg(r(\Psi_2-\Psi_{2}^{(T)})\bigg)+  \mu_2(\bm\Psi+\bar{\bf R})\p_t\bigg(r(\Psi_2-\Psi_{2}^{(T)})\bigg)=\bigg(\mu_2(\bm\Psi+\bar{\bf R})-\mu_2(\bm\Psi^{(T)}+\bar{\bf R})\bigg)\partial_t(r\Psi_{2}^{(T)}),\\
\end{equation}
and
\begin{equation}\label{5-6-4}   \partial_r \bigg(\Psi_3-\Psi_{3}^{(T)}\bigg)+\mu_3(\bm\Psi+\bar{\bf R})\p_t\bigg(\Psi_3-\Psi_{2}^{(T)})\bigg)=\bigg(\mu_3(\bm\Psi+\bar{\bf R})-\mu_3(\bm\Psi^{(T)}+\bar{\bf R})\bigg)\partial_t\Psi_3^{(T)}.\\
\end{equation}
For $ i=1,2,3,4$, define the $i$-th characteristic curve
\begin{align}
\left\{
\begin{aligned}
&\frac{\de t_{i}}{\de s}(s;\hat{t},\hat{r})=\mu_{i}
(\bm\Psi+\bar{\bf R})(t_i(s;\hat t,\hat r),s),\\
&t_{i}(\h r;\hat{t},\h r)=\hat{t}.
\end{aligned}\right.\label{1-a}
\end{align}
Noting $T_{0}=(r_1-r_0)\mathop{\max}\limits_{i=1,2,3,4}\mathop{\sup}\limits_{\bm\Psi\in \mathfrak{V}}|\mu_{i}(\bm\Psi+\bar{\bf R})|$, for each points $(\hat{t},\hat{r})\in[t_0+T_{0},\tau]\times[r_0,r_1]$, the  curve $t=t_{i}(r;\hat{t},\hat{r})$ will intersect the boundary in a time interval shorter than $T_{0}$, namely,
$$t_{1}(r_1;\hat{t},\hat{r})\in[\hat{t}-T_{0},\hat{t}]\subseteq[t_0,\tau],\ \ t_{j}(r_0;\hat{t},\hat{r})\in[\hat{t}-T_{0},\hat{t}]\subseteq[t_0,\tau], \ j=2,3,4,\ \forall(\hat{t},\hat{r})\in[t_0+T_{0},\tau]\times[r_0,r_1].$$
Furthermore, at the boundary $ r=r_0$ or $ r=r_1$, the following  estimates hold:
\begin{equation}\label{aa}
\begin{cases}
\begin{aligned}
&\bigg|\Psi_{4}(t,r_0)-\Psi_{4}^{(T)}(t,r_0)\bigg|=|K_4|\bigg|\Psi_{1}(t,r_0)
-\Psi_{1}^{(T)}(t,r_0)\bigg| \leq A_2K\mc_2 \sigma \varepsilon^{l},\ \ \forall t\in[t_{0},t_{0}+T_{0}],\\
&\bigg|\Psi_{1}(t,r_1)-\Psi_{1}^{(T)}(t,r_1)\bigg|=\bigg|\Psi_{4}(t,r_1)
-\Psi_{4}^{(T)}(t,r_1)\bigg| \leq \mc_2 \sigma \varepsilon^{l},\qquad\qquad\forall t\in[t_{0},t_{0}+T_{0}].\\
\end{aligned}
\end{cases}
\end{equation}
Here $K=\max\{|K_2|,|K_3|,|K_4|\}$ is given in \eqref{3-5-2-de}.
Then by integrating \eqref{5-6-1} and \eqref{5-6-2} along the characteristic curves $t_{4}(s;\hat{t},\hat{r})$ and $t_{1}(s;\hat{t},\hat{r})$, respectively,  one can derive from \eqref{1-22}, \eqref{3-5-2-es2-es}, \eqref{3-5-2-es2-6-1-2} and \eqref{aa} that \begin{equation}\label{3-2-time-2}
\begin{aligned}
\bigg|\Psi_{4}(\h t,\h r)-\Psi_{4}^{(T)}(\h t,\h r)\bigg|&\leq \frac{A_2K\mc_2 \sigma \varepsilon^{l}}{F_2(\h r)}+\frac{A_2(F_2(\h r)-1)\mc_2 \sigma \varepsilon^{l}}{F_2(\h r)}+\frac{A_2C_{11}\mc_1\sigma \me \mc_2 \sigma \varepsilon^{l}}{F_2(\h r)}\\
&\quad+\frac{A_2 A_0\me  M_1 }{ M_1 F_2(\h r)}\ln \frac{r_1}{r_0} \mc_2 \sigma \varepsilon^{l}\\
&\leq A_2\mc_2 \sigma \varepsilon^{l}-\frac{A_2(1-K)}{F_1(\h r)}\mc_2 \sigma \varepsilon^{l}+\frac{C_{11}\mc_1\sigma \me \mc_2 \sigma \varepsilon^{l}}{F_1(\h r)}\\
&\qquad+\frac{A_2(1-K) }{100\me}\mc_2 \sigma \varepsilon^{l}\\
&\leq A_2\mc_2 \sigma \varepsilon^{l}+A_2C_{11}\mc_1\sigma \me \mc_2 \sigma \varepsilon^{l} -\frac{99A_2 (1-K) }{100\me}\mc_2 \sigma \varepsilon^{l},\\
\end{aligned}
\end{equation}
and \begin{equation}\label{3-1-time-1}
\begin{aligned}
\bigg|\Psi_{1}(\h t,\h r)-\Psi_{1}^{(T)}(\h t,\h r)\bigg|&\leq \frac{\mc_2 \sigma \varepsilon^{l}}{F_1(\h r)}+\frac{(F_1(\h r)-1)\mc_2 \sigma \varepsilon^{l}}{F_1(\h r)}+\frac{A_2C_{11}\mc_1\sigma \me \mc_2 \sigma \varepsilon^{l}}{F_1(\h r)}\\
&\quad+\frac{A_2 A_0\me M_1  }{M_1F_1(\h r)}\ln \frac{r_1}{r_0} \mc_2 \sigma \varepsilon^{l}\\
&\leq \mc_2 \sigma \varepsilon^{l}+\frac{C_{11}\mc_1\sigma \me \mc_2 \sigma \varepsilon^{l}}{F_1(\h r)}+\frac{ A_2(1-K) }{100\me}\mc_2 \sigma \varepsilon^{l}\\
&\leq \mc_2 \sigma \varepsilon^{l}+A_2C_{11}\mc_1\sigma \me \mc_2 \sigma \varepsilon^{l} +\frac{ A_2(1-K) }{100\me}\mc_2 \sigma \varepsilon^{l}.\\
\end{aligned}
\end{equation}
 Here $C_{11}>0$  is a constant depending only on $ (\gamma,r_0,r_1,\rho_0,U_{1,0},U_{2,0},S_0)$.
Set \begin{equation}\label{de} A_2 = 1 + \frac{1-K}{25\me},  \ \ \sigma_{2}=\min\bigg\{\sigma_1,\frac{1-K}{100 A_2 C_{11} \mc_1 \me^2}\bigg\},\ \ \varepsilon = 1 - \frac{1-K}{200\me}.\end{equation} For any $\sigma\in(0,\sigma_2)$, it holds that
\begin{equation}
\begin{cases}
\begin{aligned}
\bigg|\Psi_{4}(\h t,\h r)-\Psi_{4}^{(T)}(\h t,\h r)\bigg|&\leq \bigg(A_2+\frac{1-K}{100\me}-\frac{99A_2 (1-K) }{100\me}\bigg)\mc_2 \sigma \varepsilon^{l}\\
&=(1-\frac{1-K}{200\me}+\frac{11(1-K)}{200\me}-\frac{99A_2 (1-K) }{100\me}\bigg)\mc_2 \sigma \varepsilon^{l}\\
&\leq \mc_2 \sigma \varepsilon^{l+1},\\
\bigg|\Psi_{1}(\h t,\h r)-\Psi_{1}^{(T)}(\h t,\h r)\bigg|&\leq \bigg(1+\frac{1-K}{100\me}+\frac{ A_2(1-K) }{100\me}\bigg)\mc_2 \sigma \varepsilon^{l}, \\
&= \bigg(A_2-A_2\frac{1-K}{200\me}-\frac{3(1-K)}{100\me}+A_2\frac{3 (1-K) }{200\me}\bigg)\mc_2 \sigma \varepsilon^{l} \\ &\leq A_2\mc_2 \sigma \varepsilon^{l+1}.\\
\end{aligned}
\end{cases}
\end{equation}
\par Next, at the  boundary $ r=r_0$, one has
\begin{equation}
\begin{cases}
\begin{aligned}
&\bigg|\Psi_{2}(t,r_0)-\Psi_{2}^{(T)}(t,r_0)\bigg|=|K_2|\bigg|\Psi_{1}(t,r_0)
-\Psi_{1}^{(T)}(t,r_0)\bigg| \leq A_2K\mc_2 \sigma \varepsilon^{l},\quad\forall t\in[t_{0},t_{0}+T_{0}],\\
&\bigg|\Psi_{3}(t,r_0)-\Psi_{3}^{(T)}(t,r_0)\bigg|=|K_3|\bigg|\Psi_{1}(t,r_0)
-\Psi_{1}^{(T)}(t,r_0)\bigg| \leq A_2K\mc_2 \sigma \varepsilon^{l},\quad\forall t\in[t_{0},t_{0}+T_{0}].
\end{aligned}
\end{cases}
\end{equation}
Then  by integrating \eqref{5-6-3}-\eqref{5-6-4} along the  characteristic curves $t_{2}(s;\hat{t},\hat{r})$ and $t_{3}(s;\hat{t},\hat{r})$  respectively, there holds
\begin{equation}
\begin{cases}
\begin{aligned}
&\bigg|\Psi_{2}(\h t,\h r)-\Psi_{2}^{(T)}(\h t,\h r)\bigg|\leq (A_2K+A_2C_{11} \mc_1\sigma) \mc_2 \sigma \varepsilon^{l} \leq A_2\sigma\mc_2 \varepsilon^{l+1}   ,\\
&\bigg|\Psi_{3}(\h t,\h r)-\Psi_{3}^{(T)}(\h t,\h r)\bigg|\leq (A_2K+A_2C_{11} \mc_1\sigma) \mc_2 \sigma \varepsilon^{l} \leq A_2\sigma\mc_2 \varepsilon^{l+1}   .
\end{aligned}
\end{cases}
\end{equation}
Since $\hat{t}\in[t_{0}+T_{0},\tau]$ is arbitrary, the estimates in \eqref{d22} hold.
\section{Regularity of the time-periodic solution}\noindent
\par In this section, we will prove the higher regularity of the time-periodic solutions, provided that all boundary functions $(H_{1},H_{2},H_3,H_4)(t)$ possess higher regularity. To this end, we recall the iteration scheme \eqref{3-5}-\eqref{3-8} introduced in Section 2, and then present the uniform $W^{2,\infty} $ a priori estimate for the approximate sequence as follows.
\begin{proposition}\label{p4}
For the iteration scheme \eqref{3-5}-\eqref{3-8}, under hypotheses \eqref{1-25}, for large enough
$\mc_3>0 $ and any given $ k\in \mathbb{N}_+$, one has
\begin{equation}\label{4-2}
\|\partial_{t}^{2}\Psi_{i}^{(k)}\|_{L^{\infty}(D)}\leq \mc_3, \ \
\|\partial_{t}\partial_{r}\Psi_{i}^{(k)}\|_{L^{\infty}(D)}\leq \mathcal{K}\mc_3, \ \
\|\partial_{r}^{2}\Psi_{i}^{(k)}\|_{L^{\infty}(D)}\leq \mathcal{K}^{2}\mc_3
\end{equation}
under the assumptions
\begin{equation}
\|\partial_{t}^{2}\Psi_{i}^{(k-1)}\|_{L^{\infty}(D)}\leq \mc_3, \ \
\|\partial_{t}\partial_{r}\Psi_{i}^{(k-1)}\|_{L^{\infty}(D)}\leq \mathcal{K}\mc_3,\ \
\|\partial_{r}^{2}\Psi_{i}^{(k-1)}\|_{L^{\infty}(D)}\leq \mathcal{K}^{2}\mc_3,
\end{equation}
where $i=1,2,3,4.$
\end{proposition}
\begin{proof}
We use  the same sequence constructed in Section 2. Firstly, by Proposition \ref{pr1}, we  have \eqref{3-11} for each $k$ and
\begin{equation}\label{4-3}
\mathop{\max}\limits_{i=1,2,3,4}\|\Psi_{i}^{(k)}\|_{C^{1}(D)}\leq (M_1+M_2)\sigma.
\end{equation}Set
$$\psi_{i}^{(k)}=\partial_{t}\Upsilon_{i}^{(k)}=\partial_{t}^{2}\Psi_{i}^{(k)},\ \ i=1,2,3,4, \ \ k\in \mathbb{N}_{+}.$$
At the boundary $ r=r_0$, one has
$$|\psi_{4}^{(k)}(t,r_0)|=|\mh_{4}''(t)+K_{4}\psi_{1}^{(k-1)}(t,r_0)|\leq M_3+K\mc_3.$$
Then  by differentiating temporal derivative of \eqref{3-6-t}, one can follow the similar arguments as in \eqref{3-5-t-1}-\eqref{1-4th} to derive
\begin{equation}\label{1-4th-1}
\begin{aligned}
\|\psi_{4}^{(k)}\|_{C^{0}(D)}&\leq \frac{M_3+K\mc_3}{F_2(r)}
+\frac{\mc_3(F_2(r)-1)+C_{11} \bigg(\mc_3 (M_1+M_2)\sigma+(M_1+M_2)^2\sigma^2\bigg)}{F_2(r)}\\
&\quad+\frac{A_0\me   M_1\mc_3 }{M_1F_2(r)}\ln \frac{r_1}{r_0} \\
&\leq \mc_3-\frac{(1-K)}{\me}\mc_3+C_{11} \bigg(\mc_3 (M_1+M_2)\sigma+(M_1+M_2)^2\sigma^2\bigg)+\frac{1-K}{100\me}C_3+M_3
\end{aligned}
\end{equation}
for some positive constant $C_{11}$  depending only on $ (\gamma,r_0,r_1,\rho_0,U_{1,0},U_{2,0},S_0)$. Let $\mc_3$ satisfy
$$ \frac{1-K}{100\me}\mc_3>M_3+1.$$ Then   one can choose  sufficiently small $ \sigma_{11}^\ast>0 $ such that for any
$\sigma<\sigma_4=\min\{\sigma_2,\sigma_{11}^\ast\},$  \begin{equation}\|\psi_{4}^{(k)}\|_{C^{0}(D)}\leq\chi_5\mc_3,\end{equation} where $0<\chi_5<1$ is a constant.  This, together with  \eqref{3-6-t} and the estimate \eqref{4-3}, yields that
\begin{align}
\|\partial_{r}\partial_{t}\Psi_{4}^{(k)}\|_{C^{0}(D)}&\leq \chi_5\mathcal{K}\mc_3+C_{11}\bigg((M_1+M_2)\sigma+(M_1+M_2)^2\sigma^2\bigg)
\leq \chi_6\mathcal{K}\mc_3,\label{e12}
\end{align}
where  $\chi_6$ is a constant satisfying $\chi_5<\chi_6<1$.
Furthermore, taking the spatial derivative to \eqref{3-2} and using the above estimates lead to
\begin{align}
\|\partial_{r}^2\Psi_{4}^{(k)}\|_{C^{0}(D)}&\leq \chi_6\mathcal{K}^2\mc_3+C_{11}\bigg((M_1+M_2)\sigma+(M_1+M_2)^2\sigma^2\bigg)
\leq \mathcal{K}^2\mc_3.\label{e12-1}
\end{align}
In a similar way, one has
\begin{equation}\label{4-2-11}
\|\partial_{t}^{2}\Psi_{1}^{(k)}\|_{C^{0}(D)}\leq \chi_7\mc_3, \ \
\|\partial_{t}\partial_{r}\Psi_{1}^{(k)}\|_{C^{0}(D)}\leq \chi_8\mathcal{K}\mc_3, \ \
\|\partial_{r}^{2}\Psi_{1}^{(k)}\|_{C^{0}(D)}\leq \mathcal{K}^{2}\mc_3.
\end{equation}
where the constants $\chi_7$ and $\chi_8$ satisfy $\chi_5<\chi_7<\chi_8<1$.
\par Next, at the boundary $ r=r_0$, there holds
\begin{equation*}
\begin{cases}
|\psi_{2}^{(k)}(t,r_0)|=|\mh_{2}''(t)+K_{2}\psi_{1}^{(k-1)}(t,r_0)|\leq M_3+K\mc_3,\\
|\psi_{3}^{(k)}(t,r_0)|=|\mh_{3}''(t)+K_3\psi_{1}^{(k-1)}(t,r_0)|\leq M_3+K\mc_3.\\
\end{cases}
\end{equation*} Then taking the temporal derivative of \eqref{3-7-t}-\eqref{3-8-t}
and applying the  Gronwall inequality derive that
\begin{equation}
\begin{cases}
\|\psi_{2}^{(k)}\|_{C^{0}(D)}\leq \bigg(M_3+K\mc_3
+C_{11}\bigg((M_1\sigma)^3+M_1\sigma\mc_3\bigg)\bigg) e^{C_{11}M_1\sigma(r_1-r_0)}\leq \chi_9\mc_3,\\
\|\psi_{3}^{(k)}\|_{C^{0}(D)}\leq \bigg(M_3+K\mc_3
+C_{11}\bigg((M_1\sigma)^3+M_1\sigma\mc_3\bigg)\bigg) e^{C_{11}M_1\sigma(r_1-r_0)}\leq \chi_9\mc_3,\\
\end{cases}
\end{equation}
where $0<\chi_9<1$ is a constant.   These, together with \eqref{3-7-t}-\eqref{3-8-t}, lead to
\begin{equation}
\|\partial_{r}\partial_{t}\Psi_{i}^{(k)}\|_{C^{0}(D)}\leq \chi_9\mathcal{K}\mc_3+C_{11}(M_1+M_2)^2\sigma^2
\leq \chi_{10}\mathcal{K}\mc_3, \ \ i=2,3,
\end{equation}
where the constants $\chi_9<\chi_{10}<1$. Finally, one can take the spatial derivative to \eqref{3-3}-\eqref{3-4} and combine the above estimates to obtain
\begin{equation}
\|\partial_{r}^2\Psi_{i}^{(k)}\|_{C^{0}(D)}\leq\mathcal{K}^2\mc_3, \ \ i=2,3.
\end{equation}
Therefore, we complete the proof of \eqref{4-2}.
\end{proof}
\par Now, we  prove Theorem \ref{th3}. By \eqref{4-2}, we know that $\{\bm\Psi^{(k)}\}_{k=1}^{\infty}$ is uniformly $W^{2,\infty}$ bounded and then $weak^{*}$ convergent. Moreover, noting that $\{\bm\Psi^{(k)}\}_{k=1}^{\infty}$ converges strongly to ${\bm\Psi}^{(T)}$ in $C^{1}(D)$, so we get the $W^{2,\infty}$ regularity of ${\bm\Psi}^{(T)}$.
\section{ $C^{1}$ stability of the time-periodic solution}\noindent
\par In this section, we will prove Theorem \ref{th4}. Firstly, for any $[lT_{0},(l+1)T_{0}]$ with $l\in\mathbb{N}$, we have obtained the $C^{0}$ exponential convergence in  Theorem \ref{th2} as follows:
\begin{equation}\label{6-1}
\begin{cases}
\|\Psi_{1}(t,\cdot)-\Psi_{1}^{(T)}(t,\cdot)\|_{C^{0}([r_0,r_1])}\leq A_2\mc_2\sigma \varepsilon^{l}, \ \|\Psi_{2}(t,\cdot)-\Psi_{2}^{(T)}(t,\cdot)\|_{C^{0}([r_0,r_1])}\leq A_2\mc_2 \sigma \varepsilon^{l},\ \\ \|\Psi_{3}(t,\cdot)-\Psi_{3}^{(T)}(t,\cdot)\|_{C^{0}([r_0,r_1])}\leq A_2\mc_2 \sigma \varepsilon^{l}, \
\|\Psi_{4}(t,\cdot)-\Psi_{4}^{(T)}(t,\cdot)\|_{C^{0}([r_0,r_1])}\leq \mc_2 \sigma \varepsilon^{l},
\end{cases}\end{equation}
which also show that for any $t\in[(l+1)T_{0},(l+2)T_{0}]$,
\begin{equation}\label{6-2}
\begin{cases}
\|\Psi_{1}(t,\cdot)-\Psi_{1}^{(T)}(t,\cdot)\|_{C^{0}([r_0,r_1])}\leq A_2\mc_2\sigma \varepsilon^{l+1}, \
\|\Psi_{2}(t,\cdot)-\Psi_{2}^{(T)}(t,\cdot)\|_{C^{0}([r_0,r_1])}\leq A_2\mc_2 \sigma \varepsilon^{l+1},\\
\|\Psi_{3}(t,\cdot)-\Psi_{3}^{(T)}(t,\cdot)\|_{C^{0}([r_0,r_1])}\leq A_2\mc_2 \sigma \varepsilon^{l+1},\
\|\Psi_{4}(t,\cdot)-\Psi_{4}^{(T)}(t,\cdot)\|_{C^{0}([r_0,r_1])}\leq \mc_2 \sigma \varepsilon^{l+1}.
\end{cases}\end{equation}
Furthermore, by Theorems \ref{th1} and \ref{th3}, as well as the results in \cite{Li,Q85}, one has
\begin{equation}\label{6-3}
\mathop{\max}\limits_{i=1,2,3,4}\bigg\{\|\Psi_i\|_{C^{1}(D)},\|{\Psi}_i^{(T)}\|_{C^{1}(D)}\bigg\}\leq \mc_1\sigma,
\end{equation}
and
\begin{equation}\label{6-5}
\|{\bm\Psi}^{(T)}\|_{W^{2,\infty}(D)}
\leq(1+\mk)^{2}\mc_3.\end{equation}
Next, by the continuity, we will inductively get the estimates for the convergence of the first derivatives. Suppose that for any $t\in[lT_{0},\tau]$  with $l\in \mathbb{N}_{+}$ and $\tau\in[(l+1)T_{0},(l+2)T_{0}]$,
\begin{equation}\label{6-5-1}
\begin{cases}
\|\partial_{t}\Psi_{1}(t,\cdot)-\partial_{t}\Psi_{1}^{(T)}(t,\cdot)\|_{C^{0}([r_0,r_1])}\leq A_3\mc_4\sigma\varepsilon^{l},\ \  \ \ \|\partial_{t}\Psi_{2}(t,\cdot)-\partial_{t}\Psi_{2}^{(T)}(t,\cdot)\|_{C^{0}([r_0,r_1])}\leq A_3\mc_4\sigma\varepsilon^{l},\\
\|\partial_{t}\Psi_{3}(t,\cdot)-\partial_{t}\Psi_{3}^{(T)}(t,\cdot)\|_{C^{0}([r_0,r_1])}\leq A_3\mc_4\sigma\varepsilon^{l},\ \ \ \ \|\partial_{t}\Psi_{4}(t,\cdot)-\partial_{t}\Psi_{4}^{(T)}(t,\cdot)\|_{C^{0}([r_0,r_1])}\leq \mc_4\sigma\varepsilon^{l},\\
\|\partial_{r}\Psi_{1}(t,\cdot)-\partial_{r}\Psi_{1}^{(T)}(t,\cdot)\|_{C^{0}([r_0,r_1])}\leq A_5\mk\mc_4\sigma\varepsilon^{l},\ \|\partial_{r}\Psi_{2}(t,\cdot)-\partial_{r}\Psi_{2}^{(T)}(t,\cdot)\|_{C^{0}([r_0,r_1])}\leq A_4\mk\mc_4\sigma\varepsilon^{l},\\
\|\partial_{r}\Psi_{3}(t,\cdot)-\partial_{r}\Psi_{3}^{(T)}(t,\cdot)\|_{C^{0}([r_0,r_1])}\leq A_4\mk\mc_4\sigma\varepsilon^{l},\ \|\partial_{r}\Psi_{4}(t,\cdot)-\partial_{r}\Psi_{4}^{(T)}(t,\cdot)\|_{C^{0}([r_0,r_1])}\leq A_5\mk\mc_4\sigma\varepsilon^{l},\\
\end{cases}
\end{equation}
we will prove that for any $t\in[(l+1)T_{0},\tau]$,
\begin{equation}\label{6-6}
\begin{cases}
\|\partial_{t}\Psi_{1}(t,\cdot)-\partial_{t}\Psi_{1}^{(T)}(t,\cdot)\|_{C^{0}([r_0,r_1])}\leq A_3\mc_4\sigma\varepsilon^{l+1}, \\ \|\partial_{t}\Psi_{2}(t,\cdot)-\partial_{t}\Psi_{2}^{(T)}(t,\cdot)\|_{C^{0}([r_0,r_1])}\leq A_3\mc_4\sigma\varepsilon^{l+1},\\
\|\partial_{t}\Psi_{3}(t,\cdot)-\partial_{t}\Psi_{3}^{(T)}(t,\cdot)\|_{C^{0}([r_0,r_1])}\leq A_3\mc_4\sigma\varepsilon^{l+1}, \\ \|\partial_{t}\Psi_{4}(t,\cdot)-\partial_{t}\Psi_{4}^{(T)}(t,\cdot)\|_{C^{0}([r_0,r_1])}\leq \mc_4\sigma\varepsilon^{l+1},\\
\end{cases}
\end{equation}
and
\begin{equation}\label{6-5-1-r}
\begin{cases}
\|\partial_{r}\Psi_{2}(t,\cdot)-\partial_{r}\Psi_{2}^{(T)}(t,\cdot)\|_{C^{0}([r_0,r_1])}\leq A_4\mk\mc_4\sigma\varepsilon^{l+1},\\
\|\partial_{r}\Psi_{3}(t,\cdot)-\partial_{r}\Psi_{3}^{(T)}(t,\cdot)\|_{C^{0}([r_0,r_1])}\leq A_4\mk\mc_4\sigma\varepsilon^{l+1},\\
\end{cases}
\end{equation}
and
\begin{equation}\label{6-5-1-r-1}
\begin{cases}
\|\partial_{r}\Psi_{1}(t,\cdot)-\partial_{r}\Psi_{1}^{(T)}(t,\cdot)\|_{C^{0}([r_0,r_1])}\leq A_5\mk\mc_4\sigma\varepsilon^{l+1},\\ \|\partial_{r}\Psi_{4}(t,\cdot)-\partial_{r}\Psi_{4}^{(T)}(t,\cdot)\|_{C^{0}([r_0,r_1])}\leq A_5\mk\mc_4\sigma\varepsilon^{l+1},\\
\end{cases}
\end{equation}
where $A_i>1$ $(i=3,4,5)$  are constants will be determined later.
\par Set
$$\Upsilon_{i}=\p_t\Psi_{i}, \ W_i=\p_r\Psi_{i},\ \Upsilon_{i}^{(T)}=\p_t\Psi_{i}^{(T)}, \ W_i^{(T)}=\p_r\Psi_{i}^{(T)},  \ i=1,2,3,4.$$
At the boundary   $ r=r_0$ or $ r=r_1$, it holds that
\begin{equation*}
\begin{cases}
\begin{aligned}
&\bigg|\Upsilon_{4}(t,r_0)-\Upsilon_{4}^{(T)}(t,r_0)\bigg|=|K_4|\bigg|\Upsilon_{1}(t,r_0)
-\Upsilon_{1}^{(T)}(t,r_0)\bigg| \leq A_4K\mc_4 \sigma \varepsilon^{l},\\
&\bigg|\Upsilon_{1}(t,r_1)-\Upsilon_{1}^{(T)}(t,r_1)\bigg|=\bigg|\Upsilon_{4}(t,r_1)
-\Upsilon_{4}^{(T)}(t,r_1)\bigg| \leq \mc_4 \sigma \varepsilon^{l}.\\
\end{aligned}
\end{cases}
\end{equation*}
 Then by  taking  the temporal derivatives of \eqref{5-6-1}-\eqref{5-6-2} and integrating along the corresponding  characteristic curves $t_{4}(s;\hat{t},\hat{r})$ and $t=t_{1}(s;\hat{t},\hat{r})$, respectively, one has
\begin{equation}\label{3-2-time-2-1}
\begin{aligned}
\bigg|\Upsilon_{4}(\h t,\h r)-\Upsilon_{4}^{(T)}(\h t,\h r)\bigg|&\leq \frac{A_3K\mc_4 \sigma \varepsilon^{l}}{F_2(\h r)}+\frac{A_3(F_2(\h r)-1)\mc_4 \sigma \varepsilon^{l}}{F_2(\h r)}+\frac{C_{12}\mc_1\sigma\me (A_2\mc_2 +A_3\mc_4) \sigma\varepsilon^{l}}{F_2(\h r)}\\
&\quad+\frac{A_3 A_0\me  M_1 }{ M_1 F_2(\h r)}\ln \frac{r_1}{r_0} \mc_4 \sigma \varepsilon^{l}+
\sup_{\bm\Psi\in \mathfrak{V}}|\n\mu_{4}|(1+\mathcal{K})^{2}\mc_3 A_2\mc_2\sigma\varepsilon^{l}\\
&\leq A_3\mc_4 \sigma \varepsilon^{l}-\frac{A_3(1-K)}{F_2(\h r)}\mc_4 \sigma \varepsilon^{l}+\frac{C_{12}\mc_1\sigma\me (A_2\mc_2 +A_3\mc_4) \sigma\varepsilon^{l}}{F_1(\h r)}\\
&\quad+\frac{A_3(1-K) }{100\me}\mc_4 \sigma \varepsilon^{l}+
\sup_{\bm\Psi\in \mathfrak{V}}|\n\mu_{4}|(1+\mathcal{K})^{2}\mc_3A_2\mc_2\sigma\varepsilon^{l},\\
\end{aligned}
\end{equation}
and
\begin{equation}\label{3-1-time-1-1}
\begin{aligned}
\bigg|\Upsilon_{1}(\h t,\h r)-\Upsilon_{1}^{(T)}(\h t,\h r)\bigg|&\leq \frac{\mc_4 \sigma \varepsilon^{l}}{F_1(\h r)}+\frac{(F_1(\h r)-1)\mc_4 \sigma \varepsilon^{l}}{F_1(\h r)}+\frac{C_{12}\mc_1\sigma\me (A_2\mc_2 +A_3\mc_4) \sigma\varepsilon^{l}}{F_1(\h r)}\\
&\quad+\frac{A_3 A_0\me M_1  }{M_1F_1(\h r)}\ln \frac{r_1}{r_0} \mc_4 \sigma \varepsilon^{l}+
\sup_{\bm\Psi\in \mathfrak{V}}|\n\mu_{1}|(1+\mathcal{K})^{2}\mc_3A_2\mc_2\sigma\varepsilon^{l}\\
&\leq \mc_4 \sigma \varepsilon^{l}+\frac{C_{12}\mc_1\sigma\me (A_2\mc_2 +A_3\mc_4) \sigma\varepsilon^{l}}{F_1(\h r)}+\frac{ A_3(1-K) }{100\me}\mc_4 \sigma \varepsilon^{l}\\
&\quad+
\sup_{\bm\Psi\in \mathfrak{V}}|\n\mu_{1}|(1+\mathcal{K})^{2}\mc_3\mc_2\sigma\varepsilon^{l}.\\
\end{aligned}
\end{equation}
Here $C_{12}>0$  is a constant depending only on $ (\gamma,r_0,r_1,\rho_0,U_{1,0},U_{2,0},S_0)$.
We choose \begin{equation}\label{ss}\begin{cases}\begin{aligned}
&\mc_4>\frac{100\me}{1-K}\max_{i=1,2,3,4}\sup_{\bm\Psi\in \mathfrak{V}}|\n\mu_i|(1+\mathcal{K})^{2}\mc_3A_2\mc_2+\mc_2,\\
&A_3=A_2+\frac{(1-K)}{25\me}=1+\frac{2(1-K)}{25\me},\\ &\sigma_{5}=\min\bigg\{\sigma_4,\frac{1-K}{100  C_{12} \mc_1(A_2+A_3) \me^2}\bigg\},\\ &\varepsilon = 1 - \frac{1-K}{200\me},\end{aligned} \end{cases}\end{equation} where $A_2$  is defined by \eqref{de} and $ \sigma_4 $ is given in Proposition \ref{p4}. Then for any $\sigma\in(0,\sigma_2)$, one gets
\begin{equation}
\begin{cases}
\begin{aligned}
&\bigg|\Upsilon_{1}(\h t,\h r)-\Upsilon_{1}^{(T)}(\h t,\h r)\bigg|\leq \bigg(1+\frac{1-K}{50\me}+A_3\frac{ (1-K) }{100\me}\bigg)\mc_4 \sigma \varepsilon^{l}\\
&=\bigg(A_3-A_3\frac{ (1-K) }{200\me}-\frac{3(1-K)}{50\me}+A_3\frac{ 3(1-K) }{200\me}\bigg)\mc_4 \sigma \varepsilon^{l}=\chi_{11}A_3\mc_4 \sigma \varepsilon^{l+1}, \\
&\bigg|\Upsilon_{4}(\h t,\h r)-\Upsilon_{4}^{(T)}(\h t,\h r)\bigg|\leq \bigg(A_3+\frac{1-K}{50\me}-A_3\frac{ 99(1-K) }{100\me}\bigg)\mc_4 \sigma \varepsilon^{l}\\
&=\bigg(1-\frac{ 1-K }{200\me}+\frac{21 (1-K) }{200\me}-A_3\frac{ 99(1-K) }{100\me}\bigg)\mc_4 \sigma \varepsilon^{l}=\chi_{12}\mc_4 \sigma \varepsilon^{l+1}, \\
\end{aligned}\end{cases}\end{equation} where  \begin{equation*}\begin{aligned}
\chi_{11}=1-\frac{1-K}{50\me  \varepsilon}+A_3\frac{ (1-K) }{200\me  \varepsilon}<1, \ \ \chi_{12}=1 +\frac{21 (1-K) }{200\me \varepsilon}-A_3\frac{ 99(1-K) }{100\me  \varepsilon}<1. \end{aligned}\end{equation*}
\par In the following, at the boundary $ r=r_0$, one has
\begin{equation*}
\begin{cases}
\begin{aligned}
&\bigg|\Upsilon_{2}(t,r_0)-\Upsilon_{2}^{(T)}(t,r_0)\bigg|=|K_2|\bigg|\Upsilon_{1}(t,r_0)
-\Upsilon_{1}^{(T)}(t,r_0)\bigg| \leq |A_3|K\mc_4 \sigma \varepsilon^{l},\\
&\bigg|\Upsilon_{3}(t,r_0)-\Upsilon_{3}^{(T)}(t,r_0)\bigg|=|K_3|\bigg|\Upsilon_{1}(t,r_0)
-\Upsilon_{1}^{(T)}(t,r_0)\bigg| \leq |A_3|K\mc_4 \sigma \varepsilon^{l}.
\end{aligned}
\end{cases}
\end{equation*} Then by taking  the temporal derivatives of
\eqref{5-6-3}-\eqref{5-6-4} and integrating along the characteristic curves $t_{2}(s;\hat{t},\hat{r})$ and $t_{3}(s;\hat{t},\hat{r})$ defined in \eqref{1-a}, respectively,  there hold \begin{equation}
\begin{aligned}
&\bigg|\Upsilon_{2}(\h t,\h r)-\Upsilon_{2}^{(T)}(\h t,\h r)\bigg|\\
&\leq A_3K\mc_4\sigma \varepsilon^{l} +C_{12}\mc_1\sigma(A_2\mc_2 +A_3\mc_4) \sigma\varepsilon^{l} +\sup_{\bm\Psi\in \mathfrak{V}}|\n\mu_{2}|(1+\mathcal{K})^{2}\mc_3A_2\mc_2\sigma\varepsilon^{l}\\
&\leq\bigg(A_3K+A_3\frac{1-K}{50\me}\bigg)\mc_4\sigma \varepsilon^{l}=\chi_{13}A_3\mc_4\sigma \varepsilon^{l+1}, \\
\end{aligned}
\end{equation} and  \begin{equation}
\begin{aligned}
&\bigg|\Upsilon_{3}(\h t,\h r)-\Upsilon_{3}^{(T)}(\h t,\h r)\bigg|\\
&\leq A_3K\mc_4\sigma \varepsilon^{l} +C_{12}\mc_1\sigma(A_2\mc_2 +A_3\mc_4) \sigma\varepsilon^{l} +\sup_{\bm\Psi\in \mathfrak{V}}|\n\mu_{3}|(1+\mathcal{K})^{2}\mc_3A_2\mc_2\sigma\varepsilon^{l}\\
&\leq\bigg(A_3K+A_3\frac{1-K}{50\me}\bigg)\mc_4\sigma \varepsilon^{l}=\chi_{13}A_3\mc_4\sigma \varepsilon^{l+1}, \\
\end{aligned}
\end{equation}
where
\begin{equation*}
\begin{aligned} &\chi_{13}=\frac{K+\frac{1-K}{50\me}}{\varepsilon}
=\frac{200\me K+4(1-K)}{200\me K+4(1-K)+200\me(1-K)-5(1-K)}<1.
\end{aligned}
\end{equation*}
Thus, by the arbitrariness of $\hat{t}\in[(l+1)T_{0},\tau]$, the estimates  in \eqref{6-6} holds.
\par Next, it follows from \eqref{3-1}-\eqref{3-4} that
\begin{equation}\label{7-1}\begin{aligned}
&W_{1}-W_{1}^{(T)}=\frac1r\bigg(-(\ma_{11}+\ma_{12})(r)
(\Psi_1-\Psi_1^{(T)})+(\ma_{13}-\ma_{11}-\ma_{12})(r)(\Psi_4-\Psi_4^{(T)})
\\
&\ \ +\ma_{14}(r)(\Psi_2-\Psi_2^{(T)})+\ma_{15}(r)(\Psi_3-\Psi_3^{(T)})\bigg)
+\bigg(\mu_1(\bm\Psi+\bar{\bf R})-\mu_1(\bar{\bf R})\bigg)\bigg(-\mm_{11}(r)(\Psi_1\\
&\ \ -\Psi_1^{(T)})+\mm_{12}(r)(\Psi_2-\Psi_2^{(T)})+\mm_{13}(r)
(\Psi_3-\Psi_3^{(T)})+\mm_{14}(r)(\Psi_4-\Psi_4^{(T)})\bigg)\\
&\ \ +\mu_1(\bm\Psi+\bar{\bf R})\bigg(\frac{\Psi_2^2}{r}+\frac{\gamma-1}{8r}(\Psi_4^2 - \Psi_1^2)+\frac{\bar c}{4\gamma r }\Psi_3(\Psi_1+ \Psi_4)\bigg)-\bigg(\frac{(\Psi_2^{(T)})^2}{r}\\
&\ \ +\frac{\gamma-1}{8r}((\Psi_4^{(T)})^2 - (\Psi_1^{(T)})^2)+\frac{\bar c}{4\gamma r }\Psi_3^{(T)}(\Psi_1^{(T)}+ \Psi_4^{(T)})\bigg)\bigg)-\bigg(\frac{{\Psi}_4-{\Psi}_4^{(T)} - ({\Psi}_1-{\Psi}_1^{(T)}) }{4\gamma} \\
&\ \ + \frac{\bar{c}(\Psi_3-{\Psi}_3^{(T)})}{2\gamma}\bigg)\bigg(  W_3+\mu_1(\bm\Psi+\bar{\bf R})\Upsilon_3\bigg)-\mu_{1}(\bm\Psi+\bar{\bf R})\bigg(\Upsilon_1-\Upsilon_1^{(T)}\bigg)+\bigg(\mu_{1}(\bm\Psi+\bar{\bf R})\\
&\ \ -\mu_{1}(\bm\Psi^{(T)}+\bar{\bf R})\bigg)\bigg(-\Upsilon_{1}^{(T)}-\mm_{11}(r)\Psi_1^{(T)}+\mm_{12}(r)\Psi_2^{(T)}+\mm_{13}(r)
\Psi_3^{(T)}+\mm_{14}(r)\Psi_4^{(T)}\\
&\ \ +\bigg(\frac{(\Psi_2^{(T)})^2}{r}+\frac{\gamma-1}{8r}((\Psi_4^{(T)})^2 - (\Psi_1^{(T)})^2)+\frac{\bar c}{4\gamma r }\Psi_3^{(T)}(\Psi_1^{(T)}+ \Psi_4^{(T)})\bigg)\bigg)\\
&\ \ -\bigg(\frac{{\Psi}_4 - {\Psi}_1) }{4\gamma} + \frac{\bar{c}\Psi_3}{2\gamma}\bigg)\Upsilon_3^{(T)}\bigg)
-\bigg(\frac{{\Psi}_4^{(T)} - {\Psi}_1^{(T)}) }{4\gamma} + \frac{\bar{c}{\Psi}_3^{(T)}}{2\gamma}\bigg)\bigg(  (W_3-W_3^{(T)})\\
&\ \  +\mu_1(\bm\Psi^{(T)}+\bar{\bf R})(\Upsilon_3-\Upsilon_3^{(T)})\bigg),
\end{aligned}
\end{equation}
\begin{equation*}\begin{aligned}
&W_{4}-W_{4}^{(T)}=\frac1r\bigg(-\ma_{41}(r)
(\Psi_4-\Psi_4^{(T)})+(\ma_{42}-\ma_{41})(r)(\Psi_1-\Psi_1^{(T)})
\\
&\ \ +\ma_{43}(r)(\Psi_2-\Psi_2^{(T)})-\ma_{44}(r)(\Psi_3-\Psi_3^{(T)})\bigg)
+\bigg(\mu_4(\bm\Psi+\bar{\bf R})-\mu_4(\bar{\bf R})\bigg)\bigg(\mm_{41}(r)(\Psi_1\\
&\ \ -\Psi_1^{(T)})+\mm_{12}(r)(\Psi_2-\Psi_2^{(T)})-\mm_{13}(r)
(\Psi_3-\Psi_3^{(T)})-\mm_{44}(r)(\Psi_4-\Psi_4^{(T)})\bigg)\\
&\ \ +\mu_4(\bm\Psi+\bar{\bf R})\bigg(\frac{\Psi_2^2}{r}-\frac{\gamma-1}{8r}(\Psi_4^2 - \Psi_1^2)-\frac{\bar c}{4\gamma r }\Psi_3(\Psi_1+ \Psi_4)\bigg)-\bigg(\frac{(\Psi_2^{(T)})^2}{r}\\
&\ \ -\frac{\gamma-1}{8r}((\Psi_4^{(T)})^2 - (\Psi_1^{(T)})^2)-\frac{\bar c}{4\gamma r }\Psi_3^{(T)}(\Psi_1^{(T)}+ \Psi_4^{(T)})\bigg)\bigg)+\bigg(\frac{{\Psi}_4-{\Psi}_4^{(T)} - ({\Psi}_1-{\Psi}_1^{(T)}) }{4\gamma} \\ \end{aligned}
\end{equation*}\begin{equation}\label{7-2}\begin{aligned}
&\ \ + \frac{\bar{c}(\Psi_3-{\Psi}_3^{(T)})}{2\gamma}\bigg)\bigg(  W_3+\mu_4(\bm\Psi+\bar{\bf R})\Upsilon_3\bigg)-\mu_{4}(\bm\Psi+\bar{\bf R})\bigg(\Upsilon_4-\Upsilon_4^{(T)}\bigg)+\bigg(\mu_{4}(\bm\Psi+\bar{\bf R})\\
&\ \ -\mu_{4}(\bm\Psi^{(T)}+\bar{\bf R})\bigg)\bigg(-\Upsilon_{4}^{(T)}+\mm_{41}(r)\Psi_1^{(T)}+\mm_{42}(r)\Psi_2^{(T)}-\mm_{43}(r)
\Psi_3^{(T)}-\mm_{44}(r)\Psi_4^{(T)}\\
&\ \ +\bigg(\frac{(\Psi_2^{(T)})^2}{r}-\frac{\gamma-1}{8r}((\Psi_4^{(T)})^2 - (\Psi_1^{(T)})^2)-\frac{\bar c}{4\gamma r }\Psi_3^{(T)}(\Psi_1^{(T)}+ \Psi_4^{(T)})\bigg)\bigg)\\
&\ \ +\bigg(\frac{{\Psi}_4 - {\Psi}_1) }{4\gamma} + \frac{\bar{c}\Psi_3}{2\gamma}\bigg)\Upsilon_3^{(T)}\bigg)
+\bigg(\frac{{\Psi}_4^{(T)} - {\Psi}_1^{(T)}) }{4\gamma} + \frac{\bar{c}{\Psi}_3^{(T)}}{2\gamma}\bigg)\bigg(  (W_3-W_3^{(T)})\\
&\ \ +\mu_4(\bm\Psi^{(T)}+\bar{\bf R})(\Upsilon_3-\Upsilon_3^{(T)})\bigg),\\
\end{aligned}
\end{equation}
\begin{equation}\label{7-3}\begin{aligned}
&W_{2}-W_{2}^{(T)}=-\frac{\Psi_{2}-\Psi_{2}^{(T)}}r-\mu_{2}(\bm\Psi+\bar{\bf R})\bigg(\Upsilon_2-\Upsilon_{2}^{(T)}\bigg)-\bigg(\mu_{2}(\bm\Psi+\bar{\bf R})-\mu_{2}(\bm\Psi^{(T)}+\bar{\bf R})\bigg)\Upsilon_{2}^{(T)},\end{aligned}
\end{equation}
and \begin{equation}\label{7-4}
W_{3}-W_{3}^{(T)}=-\mu_{3}(\bm\Psi+\bar{\bf R})\bigg(\Upsilon_3-\Upsilon_{3}^{(T)}\bigg)
-\bigg(\mu_{3}(\bm\Psi+\bar{\bf R})-\mu_{3}(\bm\Psi^{(T)}+\bar{\bf R})\bigg)\Upsilon_{3}^{(T)}.
\end{equation}
Then  the straightforward computations together with the above estimates obtain  that
\begin{equation}
\begin{cases}
\begin{aligned}
&\|W_{2}(t,\cdot)-W_{2}^{(T)}(t,\cdot)\|_{C^{0}([r_0,r_1])}\\
&\leq \chi_{13}A_3\mc_4 \mk\sigma \varepsilon^{l+1}+(1+\mk)C_{13}(A_2\mc_2\mc_4 +\mc_1\sigma A_2\mc_2\mc_4 )\sigma \varepsilon^{l+1},\\
&\|W_{3}(t,\cdot)-W_{3}^{(T)}(t,\cdot)\|_{C^{0}([r_0,r_1])}\\
&\leq \chi_{13}A_3\mc_4 \mk\sigma \varepsilon^{l+1}+(1+\mk)C_{13}(A_2\mc_2\mc_4 +\mc_1\sigma A_2\mc_2\mc_4 )\sigma \varepsilon^{l+1}.\\
\end{aligned}
\end{cases}
\end{equation} Here $C_{13}>0$  is a constant depending only on $ (\gamma,r_0,r_1,\rho_0,U_{1,0},U_{2,0},S_0)$.
Set
$$ A_4>\chi_{13}A_3 +\frac{(1+\mk)C_{13}(A_2\mc_2 +\mc_1\sigma_5 A_2\mc_2 )}{\mk}+1.$$ Therefore,  the proof of \eqref{6-5-1-r} is completed. This, together with \eqref{7-1} and  \eqref{7-2}, yields that
\begin{equation}\label{7-1}
\begin{cases}
\begin{aligned}
&\|W_{1}(t,\cdot)-W_{1}^{(T)}(t,\cdot)\|_{C^{0}([r_0,r_1])}\\
&\leq \chi_{13}A_3\mc_4 \mk\sigma \varepsilon^{l+1}+(1+\mk)C_{13}(A_2\mc_2\mc_4 +\mc_1\sigma A_2\mc_2\mc_4+\mc_1\sigma A_5\mc_4 )\sigma \varepsilon^{l+1},\\
&\|W_{4}(t,\cdot)-W_{4}^{(T)}(t,\cdot)\|_{C^{0}([r_0,r_1])}\\
&\leq \chi_{13}A_3\mc_4 \mk\sigma \varepsilon^{l+1}+(1+\mk)C_{13}(A_2\mc_2\mc_4 +\mc_1\sigma A_2\mc_2\mc_4+\mc_1\sigma A_5\mc_4 )\sigma \varepsilon^{l+1}.\\
\end{aligned}
\end{cases}
\end{equation}
Set
$$ A_5>\chi_{13}A_3 +\frac{(1+\mk)C_{13}(A_2\mc_2 +\mc_1\sigma_5 A_2\mc_2 \mc_1\sigma_5 A_5)}{\mk}+1.$$ The estimate   \eqref{6-5-1-r-1} is obtained. Hence we prove Theorem  \ref{th4} completely.
\par {\bf Acknowledgement.} The research of   Huimin Yu is  partially supported by National Natural Science Foundation of China (No. 12271310) and Natural Science Foundation of Shandong Province (ZR2022MA088).  The research of Zihao Zhang is supported by  the China Scholarship Council (Grant No. 202606170019).
\par {\bf Data availability.} No data was used for the research described in the article.
    \par {\bf Conflict of interest.} This work does not have any conflicts of interest.
 
 \end{document}